\documentclass{amsart}

\usepackage{amsmath,amsthm,amsfonts,amscd,latexsym,amssymb,amsbsy,dsfont,mathrsfs}

\newtheorem{theorem}{Theorem}[section]
\newtheorem{theorem*}{Theorem}
\newtheorem{lemma}[theorem]{Lemma}
\newtheorem{corollary}[theorem]{Corollary}
\newtheorem{proposition}[theorem]{Proposition}

\theoremstyle{definition}
\newtheorem{definition}[theorem]{Definition}

\theoremstyle{remark}
\newtheorem{remark}[theorem]{Remark}

\numberwithin{equation}{section}

\newcommand{\C}{\mbb{C}}

\newcommand{\I}{\mc{I}}

\newcommand{\N}{\mbb{N}}

\newcommand{\R}{\mbb{R}}
\newcommand{\OO}{\Omega}

\newcommand{\LL}{\mc{L}}

\newcommand{\K}{\mc{K}}

\newcommand{\mbb}{\mathbb}
\newcommand{\mc}{\mathcal}
\newcommand{\mi}{\mathit}

\newcommand{\mr}{\mathrm}

\newcommand{\mscr}{\mathscr}
\newcommand{\lra}{\longrightarrow}
\newcommand{\pr}{\prime}

\newcommand{\vep}{\varepsilon}

\newcommand{\insec}{~}

\def\dd#1#2{\dfrac{\partial#1}{\partial#2}}

\newcommand\Sl{\mathcal S}		
\newcommand{\jS}{\Sl_{\bC_\jmath}}
\newcommand\PS{\mc{P}\Sl}	
\newcommand\PRS{\PS_{\bR}}	
\newcommand\Sc{\Sl_c}	
\newcommand\RS{\Sl_{\bR}}
\newcommand\ssp{\sigma_{S}}	
\newcommand\srho{\rho_{S}}	

\newcommand{\cS}{\mbb{S}}

\newcommand{\rr}{|}
\newcommand{\Ran}{\mi{Ran}}
\newcommand{\1}{\mr{I}}
\newcommand{\old}{\oldstylenums{1}}

\newcommand{\cG}{\mc{G}}

\newcommand{\CC}{\mscr{C}}

\def\sH{\mathsf{H}}

\def\bC{{\mathbb C}}           

\def\bH{{\mathbb H}}
\def\bN{{\mathbb N}}

\def\bR{{\mathbb R}}

\def\bS{{\mathbb S}}

\def\gB{{\mathfrak B}}

\def\beq{\begin{equation}}
\def\eeq{\end{equation}}

\def\b{\langle}
\def\k{\rangle}

\begin{document}

\title[Quaternionic slice functional calculus]{Continuous slice functional calculus \\ in quaternionic Hilbert spaces}


\author{Riccardo Ghiloni}
\address{Department of Mathematics, University of Trento, I--38123, Povo-Trento, Italy}
\email{ghiloni@science.unitn.it \\ moretti@science.unitn.it \\ perotti@science.unitn.it}
\thanks{Work partially supported by GNSAGA and  GNFM of INdAM}

\author{Valter Moretti}

\author{Alessandro Perotti}



\date{}


\begin{abstract}
The aim of this work is to define a continuous functional calculus in quaternionic Hilbert spaces, starting from basic issues regarding the notion of spherical spectrum of a normal operator. As properties of the spherical spectrum suggest, the class of continuous functions to consider in this setting is the one of slice quaternionic functions. Slice functions generalize the concept of slice regular function, which comprises power series with quaternionic coefficients on one side and that can be seen as an effective generalization to quaternions of holomorphic functions of one complex variable.
The notion of slice function allows to introduce suitable classes of real, complex and quaternionic $C^*$--algebras and to define, on each of these $C^*$--algebras, a functional calculus for quaternionic normal operators. In particular, we establish several versions of the spectral map theorem. Some of the results are proved also for unbounded operators. However, the mentioned continuous functional calculi are defined only for bounded normal operators. Some comments on the physical significance of our work are included.
\end{abstract}

\maketitle


\setcounter{tocdepth}{1}

\tableofcontents

\section{Introduction} 
Functional calculus in quaternionic Hilbert spaces has been focused especially by mathematical physicists, concerning the application of spectral theory to quantum theories (see e.g. \cite{Adler,Emch,FJSD,HB}). As a careful reading of these works reveals, most part of results have been achieved into a non completely rigorous fashion, leaving some open gaps and deserving further investigation. As a matter of fact, the classic approach for complex Hilbert spaces, starting from the continuous functional calculus and then reaching the measurable functional calculus, has been essentially disregarded, while attention has been devoted to the measurable functional calculus almost immediately. Furthermore, even more mathematically minded works on this topics, as \cite{visw}, do not present complete proofs of the claimed statements. Historically, an overall problem was the absence of a definite notion of spectrum of an operator on quaternionic Hilbert spaces. Such a notion has been introduced only few years ago \cite{libroverde} in the more general context of operators on quaternionic Banach modules. Therefore, on the one hand, no systematic investigation on the spectral properties of operators on quaternionic Hilbert spaces has been performed up to now. On the other hand, the many key--results available for non self--adjoint operators 
 have been obtained by means of a Dirac bra--ket like formalism, which formally, but often erroneously, reduces the argumentations to the case of a vague notion of point spectrum. Finally, another interesting issue concerns the existence of a model of quaternions in terms of anti--self adjoint and unitary operators, commuting with the self--adjoint parts of a given normal operator. Although that model is not completely understood and analyzed, it was extensively exploited in various technical constructions of physicists (see \cite{Adler} and Theorem~\ref{newtheorem} below).
  
The aim of this work is to provide a foundational investigation of continuous functional calculus in quaternionic Hilbert spaces, particularly starting from basic issues regarding the general notion of \textit{spherical spectrum} (Definition~\ref{def_spectrum}) and its general properties (Theorems~\ref{teopropspectrum}
and \ref{teospectra} and Propositions~\ref{propsigmap} and \ref{propspecTT}). The general relation between the theory in complex Hilbert space and the one in quaternionic Hilbert spaces will be examined extending some classic known results (see e.g.~\cite{Emch}) to the unbounded operator case (Proposition~\ref{propestensione}). In view of that general aim, up to Section~\ref{secspecprop}, we shall not confine ourselves to the bounded operator case, but we shall consider also unbounded operators. However, the proper continuous functional calculus will be discussed for bounded normal operators only, thus postponing the non--bounded case to a work in preparation \cite{GhMoPe2}. In particular, in the first part of the work, the general spectral properties of bounded and unbounded operators on quaternionic Hilbert spaces will be discussed from scratch.

The pivotal tool in our investigation is the notion of {\em slice function}, whose relevance clearly pops up once the notion of spherical spectrum is introduced. As we shall prove, in the continuous case, that notion allows one to introduce suitable classes of real, complex and quaternionic $C^*$--algebras and to define, on each of these $C^*$--algebras, a functional calculus for normal operators. In particular, we establish several versions of the {\em spectral map theorem}. As our results show, the interplay between continuous slice functions and the space of operators is much more complicated than in the case of continuous functional calculus on complex Hilbert spaces.
 
\subsection{Slice functions, a key result and the main theorems}
The concept of \emph{slice regularity} for functions of one quaternionic variable has been introduced by Gentili and Struppa in \cite{GeSt2006CR,GeSt2007Adv} and then extended to octonions,  Clifford algebras and in general real alternative $^*$--algebras in \cite{CoSaSt2009Israel,GeStRocky,GhPe_Trends,GhPe_AIM}.  This function theory comprises polynomials and power series in the quaternionic variable with quaternionic coefficients on one side. It can then be seen as an effective generalization to quaternions of the theory of holomorphic functions of one complex variable.

At the base of the definition of slice regularity, there is the ``slice'' character of the quaternionic algebra: every element $q\in \bH$ can be decomposed into the form $q=\alpha+\jmath\beta$, where $\alpha,\beta\in\bR$ and $\jmath$ is an imaginary unit in the two--dimensional sphere $\bS=\{q\in\bH\,|\, q^2=-1\}$. 
This decomposition is unique when assuming further that $\beta \geq 0$; otherwise, for non--real $q$, $\beta$ and $\jmath$ are determined up to a sign.

Equivalently, $\bH$ is the union of the (commutative) real subalgebras $\bC_\jmath\simeq\bC$ generated by  $\jmath\in\cS$, with the property that $\bC_\imath\cap\bC_\jmath=\bR$ when $\imath\ne\pm\jmath$, where $\bR$ denotes the real subalgebra of $\bH$ generated by 1.

The original definition \cite{GeSt2006CR,GeSt2007Adv} of slice regularity for a quaternionic function $f$, defined on an open domain $\OO$ of $\bH$, requires that, for every $\jmath \in \cS$, the restriction of $f$ to $\OO\cap\bC_\jmath$ is holomorphic with respect to the complex structure defined by left multiplication by $\jmath$. The approach taken in \cite{GhPe_Trends,GhPe_AIM} allows to embed the space of slice regular functions into a larger class, that of continuous \emph{slice functions}, which corresponds in some sense to the usual complex continuous functions on the complex plane.

The first step is to single out a peculiar class of subsets of $\bH$, those that are invariant with respect to the action of $\cS$. If $\K \subset \bC$ is non--empty and invariant under complex conjugation, one defines the \textit{circularization} $\OO_\K$ of $\K$ (in $\bH$) as: 
\[
\OO_\K =\{\alpha+\jmath\beta \in \bH \, | \,  \alpha,\beta \in \bR, \, \alpha+i\beta \in \K, \, \jmath \in \bS \},
\]
and call a subset of $\bH$ a \emph{circular set} if it is of the form $\OO_\K$ for some $\K$. The second step is the introduction of \emph{stem functions}. Let $\bH\otimes_{\bR}\bC$ be the complexified quaternionic algebra, represented as
\[
\bH\otimes_{\bR}\bC=\{x+iy \, | \, x,y\in \bH\},
\]
with complex conjugation $w=x+i y\mapsto \overline w=x-i y$. If a function $F: \K \lra \bH \otimes_\bR \bC$ satisfies the condition $F(\overline z)=\overline{F(z)}$ for every $z\in \K$, then $F$  is called a \emph{stem function} on $\K$. Any stem function $F:\K \lra \bH \otimes_\bR \bC$ induces a \emph{(left) slice function} $f=\I(F):\OO_\K\rightarrow \bH$: if $q=\alpha+\jmath\beta \in \OO_\K\cap \bC_\jmath$, with $\jmath\in\cS$, we set  
\[f(q):=F_1(\alpha+i\beta)+\jmath F_2(\alpha+i\beta)\;,\]
where $F_1, F_2$ are the two $\bH$--valued components of $F$. In this approach to the theory, a quaternionic function turns out to be slice regular if and only if it is the slice function induced by a holomorphic stem function.

The definition of a continuous slice function of a normal operator is based on a key result, which can interpreted as the operatorial counterpart of the slice character of $\bH$ and which describes rigorously the just mentioned model of quaternions for bounded normal operators exploited by physicists. Theorem \ref{teoext} in Section \ref{sec:commuting operator} can be reformulated as follows:

\medskip

\noindent \textbf{Theorem \textit{J}.} 
\textit{Let $\sH$ be a quaternionic Hilbert space and let $\gB(\sH)$ be the set of all bounded operators of $\sH$. Given any normal operator $T \in \gB(\sH)$, there exist three operators $A,B,J \in \gB(\sH)$ such that:}
\begin{itemize}
 \item[$(\mr{i})$] $T=A+JB$,
 \item[$(\mr{ii})$] \textit{$A$ is self--adjoint and $B$ is positive,}
 \item[$(\mr{iii})$] \textit{$J$ is anti self--adjoint and unitary,}
 \item[$(\mr{iv})$] \textit{$A$, $B$ and $J$ commute mutually.}
\end{itemize}
\textit{Furthermore, the following additional facts hold:}
\begin{itemize}
 \item \textit{$A$ and $B$ are uniquely determined by $T$: $A=(T+T^*)\frac{1}{2}$ and $B=|T-T^*|\frac{1}{2}$.}
 \item \textit{$J$ is uniquely determined by $T$ on $\mi{Ker}(T-T^*)^\perp$.}
\end{itemize}

\medskip

In the parallelism between this decomposition of $T$ and the slice decomposition of quaternions, real numbers in $\bH$ correspond to self--adjoint operators, non--negative real numbers to positive operators and quaternionic imaginary units to anti self--adjoint and unitary operators. This parallelism suggests a natural way to define the operator $f(T)$ for a continuous slice function $f$, at least in the case of \emph{$\bH$--instrinsic} slice functions, i.e.\ those slice functions such that $f(\bar q)=\overline{f(q)}$ for every $q \in \OO_\K$. Observe that $\bH$--intrinsic slice functions leave all slices $\bC_\jmath$ invariant. If $f$ is a polynomial slice function, induced by a stem function with polynomials components $F_1$, $F_2 \in \bR[X,Y]$, then $f$ is $\bH$--intrinsic and we define the normal operator $f(T) \in \gB(\sH)$ by setting
\[
f(T):=F_1(A,B)+JF_2(A,B),
\]
and then we extend the definition to continuous $\bH$--intrinsic slice functions by density. We obtain in this way a isometric $^*$--homomorphism of real Banach  $C^*$--algebras (\textit{Theorem~\ref{teofinale1}}).
In particular, continuous $\bH$--intrinsic slice  functions satisfy the spectral map property $f(\ssp(T)) =\ssp(f(T))$, where $\ssp(T)$ denotes the spherical spectrum of $T$.

The definition of $f(T)$ can be extended to other classes of continuous slice functions. In this transition, some of the nice properties of the map $f \mapsto f(T)$ are lost, but other interesting phenomena appear. In \textit{Theorem~\ref{teofinale2}}, the slice functions considered are those which leave only one slice $\bC_\jmath$ invariant. In this case, a $^*$--homomorphism of \emph{complex} Banach $C^*$--algebras is obtained. A suitable form of the spectral map property continues to hold. In \textit{Theorem~\ref{teofinale3}}, we consider \emph{circular} slice functions, those which satisfy the condition $f(\bar q)=f(q)$ for every $q$. Differently from the previous cases, these functions form a \emph{non--commutative} quaternionic Banach $C^*$--algebra. In this case, we still have an isometric  $^*$--homomorphism, but the spectral map property holds in a weaker form. In \textit{Proposition~\ref{teofinale4}}, we show how to extend the previous definitions of  $f(T)$ to a generic continuous slice function $f$, but in this case also the $^*$--homomorphism property is necessarily lost. Finally, in Section~\ref{sec:sliceregular}, we show that the continuous functional calculus defined above, when restricted to slice regular functions, coincides with the functional calculus developed in \cite{libroverde} as a generalization of the classical holomorphic functional calculus.


\subsection{Physical significance} \label{intro:physic} 
As remarked  by Birkhoff and von Neumann in their celebrated seminal work on Quantum Logic in 1936 \cite{BvN}, Quantum Mechanics may alternatively be formulated on a Hilbert space where the ground field of complex numbers is replaced for the division algebra of quaternions. Nowadays, the picture is more clear on the one hand  and more strict on the other hand, after 
the efforts started in 1964 by Piron \cite{Piron} and concluded in 1995 by Sol\`{e}r \cite{Soler}, and more recently reformulated by other researchers (see e.g. \cite{Aerts}). Indeed, it has been rigorously established that, assuming that the set of ``yes--no'' elementary propositions on a given quantum system are described by a lattice that is bounded, orthomodular, atomic, separable, irreducible, verifying the so--called covering property (see \cite{BC,EGL09,Moretti}) and finally, assuming that certain orthogonal systems exist therein, then the lattice is isomorphic to the lattice of orthogonal projectors  of a generalised Hilbert space over the fields $\bR$, $\bC$ or over the division algebra of quaternions $\bH$. No further possibility is allowed. Actually, the first possibility is only theoretical, since it has been proved that, dealing with concrete quantum systems, the description of the time--reversal operation introduces a complex structure in the field that makes,
 indeed, the real Hilbert space a complex Hilbert space (see \cite{Adler}). Therefore, it seems that the only two realistic possibilities allowed by Nature are complex Hilbert spaces and quaternionic ones. While the former coincides to the mainstream of the contemporary quantum physics viewpoint, the latter has been taken into consideration by several outstanding physicists and mathematical physicists, since the just mentioned paper of Birkhoff and von Neumann. Adler's book \cite{Adler} represents a quite complete treatise on that subject from the point of view of physics.

While all the fundamental results, like Gleason theorem and Wigner theorem, can be re--demonstrated in quaternionic quantum mechanics with minor changes (see \cite{vara}), its theoretic formulation 
 differs from the complex formulation in some key points related to  the  spectral theory,  the proper language of quantum mechanics. Perhaps, the most important is the following. Exactly as in the standard approach, observables are represented by
(generally unbounded) self--adjoint operators. However, within the complex Hilbert space picture, self--adjoint operators enter the theory also from another route due to the celebrated Stone theorem. Indeed, self--adjoint operators, when multiplied with $i$, become the generators of the continuous one--parameter groups of unitary operators representing continuous quantum symmetries.  More generally, in complex Hilbert spaces, in view of well--known results due to Nelson \cite{Nelson}, anti self--adjoint operators are the building blocks necessary to construct strongly continuous unitary representations of Lie groups of quantum symmetries. This way naturally leads to the quantum version of Noether 
 theorem relating conserved quantities ($i$ times the anti self--adjoint generators representing the Lie algebra of the  group) and symmetries (the one--parameter groups obtained by exponentiating the given anti self--adjoint generators) of a given quantum system. This nice interplay, in principle, should survive the passage from complex to quaternionic context. However, a difficult snag pops up immediately: the relationship between self--adjoint operators $S$ and anti self--adjoint operators $A$ is much more complicated in quaternionic Hilbert spaces than in complex Hilbert spaces. Indeed, in the quaternionic case, an identity as $A= JS$ holds, where $J$ is an operator replacing the trivial $i$ in complex Hilbert spaces. Nevertheless, $A$ does not uniquely fix the pair $J,S$ so that the interplay of dynamically conserved quantities and symmetries needs a deeper physical investigation in quaternionic quantum mechanics. This issue affects all the physical construction from scratch as it is already evident from the various mathematically inequivalent attempts to provide a physically sound definition of the momentum operator of a particle. Furthermore, one has to employ quite sophisticated mathematical tools as the quaternionic version of Mackey's imprimitivity theorem (see \cite{CT}). The operator $J$ has to satisfy several constraints, first of all, it has to commute with $S$ and it has to be anti self--adjoint. The polar decomposition theorem, re--formulated in quaternionic Hilbert spaces, provides such an operator, at least for bounded anti self--adjoint operators $A$. In some cases, it is convenient for technical reasons to look for an operator $J$ that is also unitary and that it is accompanied by two other similar operators $I$ and $K$, commuting with $S$, such that they define a representation of the imaginary quaternions in terms of operators (see Section 2.3 of \cite{Adler}). This is not assured by the polar decomposition theorem and the existence of such anti self--adjoint and unitary operators $I,J,K$ is by no means obvious. This is one of the key issues tackled in this work. Indeed, in Theorem \ref{teoext}, we prove that,  every normal operator $T$ can always be decomposed as $T=A+JB$, with $A,B$ self--adjoint and $J$ anti self--adjoint, unitary and commuting with both $A$ and $B$. Moreover, $J$ can always be written as $L_\imath$ for some (quaternionic) imaginary unit $\imath$, where $\bH \ni q \mapsto L_q \in \gB(\sH)$ is a $^*$--representation of $\bH$ in terms of bounded operators, commuting with $A$ and $B$ (see Theorem \ref{newtheorem}).

The systematic investigation of the properties of the above--mentioned operators $J$ will enable us to state and prove a quite general proposition (Proposition \ref{propestensione}  extending previous results by Emch \cite{Emch}) concerning the extension to the whole quaternionic Hilbert space and the properties of such extensions of (generally unbounded) operators initially defined on complex subspaces of the quaternionic Hilbert space induced by $J$.

As we have just recalled, when looking at this matter with a mathematically minded intention, the most problematic issue in the known literature on quaternionic quantum mechanics is the notion of spectrum of an operator. As is known, in the case of self--adjoint operators representing quantum observables, the Borel sets in the spectrum of an observable account for the outcomes of measurements of that observable. The definition of spectrum of an operator in quaternionic Hilbert spaces  turns out to be  problematic on its own right in view of the fact that quaternions are a non commutative ring and it generates troubles, already studying eigenvalues and eigenvectors: an eigenspace is not a subspace because it is not closed under multiplication with scalars. However, at first glance, all that could not appear as  serious as it is indeed, because self--adjoint operators should have real spectrum where quaternionic noncommutativity is ineffective. Nevertheless, as stressed above, self--adjoint operators are not the only class of operators relevant in quaternionic Hilbert spaces for quantum physics, since also unitary and anti self--adjoint play some important r\^ole. Therefore, a full fledged notion of spectrum for normal operators should be introduced. In spite of some different formulations of the spectral theorem for normal operators \cite{visw}, a notion of the spectrum and of a generic operator in quaternionic Hilbert spaces as well as a systematic  investigation on its properties do not exist. As a matter of fact, mathematical physicists \cite{Emch,FJSD,HB} always tried  to avoid to face this   issue confining their investigations to (bounded) self--adjoint operators and passing to other classes of operators by means of quite {\em ad hoc} arguments and very often, dealing with the help of Dirac bra--ket formalism, not completely justified in these contexts. This remark concerns physicists, \cite{Adler} in particular, who adopt  the popular and  formal bra--ket procedure, assume the naive starting point where the spectrum is made of  eigenvalues  even when that  approach is evidently untenable and should be handled with the help of some rigged quaternionic Hilbert space machinery similar to Gelfand's theory in complex Hilbert spaces. While all these approaches are physically sound and there are no doubts that the produced results are physically meaningful, an overall rigorous mathematical  formulation of the quaternionic spectral theory still does not exist. The absence of a suitable notion of spectrum has generated a lack in the natural development of the functional calculus. As a matter of fact, in  the complex Hilbert space theory,  the functional calculus theory on the spectrum starts by the definition of a continuous function of a given normal operator and then passes to define the notion of measurable function of the operator. The projector--valued measures exploited to formulate the spectral theory, the last step of the story,  are subsequently constructed  taking advantage of the measurable functional calculus. In quaternionic Hilbert space, instead, a formulation of the spectral theorem exists \cite{visw} without any systematic investigation of the continuous and measurable functional calculus, nor an explicit definition of the spectrum of an operator.

A pivotal notion in this paper is just that of spherical spectrum of an operator in a quaternionic Hilbert space. We state it (Definition \ref{def_spectrum}) by specializing (and re--elaborating  in the case of unbounded operators) the definition recently proposed by Colombo, Sabadini and Struppa in \cite{libroverde} for operators in quaternionic two--side Banach modules. We see that, in spite of the different definition (based on a second order polynomial), the properties of the spectrum of an operator (Theorem \ref{teospectra}) are natural generalizations of those in complex Hilbert spaces. In particular, the point spectrum, regardless an apparently inequivalent definition, turns out to be the set of eigenvalues exactly as in the complex Hilbert space case (Proposition \ref{propsigmap}). Moreover, a nice relationship appears (Proposition \ref{propinterssigma}) between the standard notion of spectrum in complex Hilbert spaces and the quaternionic notion of spectrum, for those operators that have been obtained as extensions of operators on complex Hilbert subspaces as pointed out above. The definition and the properties of the spectrum not only allow to construct a natural extension of the continuous functional calculus to the quaternionic Hilbert space case, but they permit us to discover an intriguing interplay of the continuous functional calculus and the theory of slice functions.


\section{Quaternionic Hilbert spaces}
In this part, we summarize some  basic notions about the algebra of quaternions, quaternionic Hilbert spaces and operators (even unbounded and defined in proper subspaces) on quaternionic Hilbert spaces. In almost all cases, the proofs of the various statements concerning quaternionic Hilbert spaces and operators thereon are very close to the analogues for the complex Hilbert space theory. Therefore, we shall omit the corresponding proofs barring some comments if necessary.


\subsection{Quaternions}

The \emph{space of quaternions} $\bH$ is the four dimensional real algebra with unity we go to describe. 

We denote by $0$ the null element of $\bH$ and by $1$ the multiplicative identity of~$\bH$. The space $\bH$ includes three so--called \emph{imaginary units}, which we indicate by $i,j,k$. By definition, they satisfy (we omit the symbol of product of the algebra):
\begin{equation} \label{prod}
ij = -ji = k, \; \; ki = -ik = j, \; \; jk = -kj =i, \; \; ii=jj=kk =-1\:.
\end{equation}
The elements $1,i,j,k$ are assumed to form a real vector basis of $\bH$, so that any element $q \in \bH$ takes 
the form: $q = a 1+ bi + cj + dk$, where $a$, $b$, $c$, and $d$ belong to $\bR$ and are uniquely determined by $q$ itself. We identify $\bR$ with the subalgebra generated by $1$. In other words, if $a \in \bR$, we write simply $a$ in place of $a1$. The product of two such elements is individuated by (\ref{prod}), assuming associativity and distributivity with respect to the real vector space sum. In this way, $\bH$ turns out to be a \emph{non--commutative associative real division algebra}.

Given $q=a+bi+cj+dk \in \bH$, we recall that:
\begin{itemize}
 \item $\overline{q}:=a-bi-cj-dk$ is the \emph{conjugate quaternion of $q$}.
 \item $|q|:=\sqrt{q\overline{q}}=\sqrt{a^2+ b^2+c^2+d^2} \in \bR$ is the \emph{norm of $q$}.
 \item $\mr{Re}(q):=\frac{1}{2}(q+\overline{q})=a \in \bR$ is the \emph{real part of $q$} and $\mr{Im}(q):=\frac{1}{2}(q-\overline{q})=bi+cj+dk$ is the \emph{imaginary part of $q$}.
\end{itemize}
The element $q \in \bH$ is said to be \emph{real} if $q=\mr{Re}(q)$. It is easy to see that $q$ is real if and only if $qp=pq$ for every $p \in \bH$ or, equivalently, if and only if $\overline{q}=q$. If $\overline{q}=-q$ or, equivalently, $q=\mr{Im}(q)$, then $q$ is said to be \emph{imaginary}. We denote by $\mr{Im}(\bH)$ the imaginary space of $\bH$ and by $\mathbb S$ the sphere of unit imaginary quaternions:
\[
\mr{Im}(\bH):=\{q \in \bH \, | \, q=\mr{Im}(q)\}=\{bi+cj+dk \in \bH \, | \, b,c,d \in \bR\}
\]
and
\[
\bS:=\{q \in \mr{Im}(\bH) \, | \, |q|=1\}=\{bi+cj+dk \in \bH \, | \, b,c,d \in \bR, b^2+c^2+d^2=1\}.
\]
As said in the introduction, the set $\cS$ is also equal to $\{q \in \bH \, | \, q^2=-1\}$.

We define $\bR^+:=\{a \in \bR \, | \, a \geq 0\}$ and we denote by $\bN$ the set of all non--negative integers.

\begin{remark} \label{remin}
$(1)$ Notice that $\overline{pq} = \overline{q} \, \overline{p}$, so the conjugation 
reverses the order of the factors.

$(2)$ $|\cdot|$ is, indeed, a norm on $\bH$, when it is viewed as the real vector space $\bR^4$. It also satisfies $|1|=1$, 
 $|\overline{q}|= |q|$ if $q\in \bH$ and
the remarkable identity:
\[
|pq| = |p||q| \quad \text{if $p,q \in \bH$.}
\]

$(3)$ For every $\jmath \in \bS$, denote by $\bC_\jmath$ the real subalgebra of $\bH$ generated by $\jmath$; that is, $\C_\jmath:=\{\alpha+\jmath\beta \in \bH \, | \, \alpha,\beta \in \bR\}$. Given any $q \in \bH$, we can write
\[
q=\alpha+\jmath\beta \quad \text{for some $\alpha,\beta \in \bR$ and $\jmath \in \bS$}.
\]
Evidently, $\alpha$ and $|\beta|$ are uniquely determined: $\alpha=\mr{Re}(q)$ and $|\beta|=|\mr{Im}(q)|$. If $q \in \bR$, then $\beta=0$ and hence $\jmath$ can be chosen arbitrarily in $\bS$. Thanks to the Independence Lemma (see \cite[\S 8.1]{ebb}), if $q \in \bH \setminus \bR$, then there are only two possibilities: $\beta=\pm|\mr{Im}(q)|$ and $\jmath=\beta^{-1}\mr{Im}(q)$. In other words, we have that $\bH=\bigcup_{\jmath \in \bS}\bC_\jmath$ and $\bC_\jmath \cap \bC_\kappa=\bR$ for every $\jmath,\kappa \in \bS$ with $\jmath \neq \pm\kappa$.

$(4)$ Two quaternions $p$ and $q$ are called \emph{conjugated} (to each other), if there is $s\in \bH \setminus\{0\}$ such that $p = sqs^{-1}$. The conjugacy class of $q$; that is, the set of all quaternions conjugated with $q$, is equal to the $2$--sphere $\bS_q:=\mr{Re}(q)+|\mr{Im}(q)|\bS$ of $\bH$. More explicitly, if $q=a+bi+cj+dk$, then we have that
\[
\cS_q=\{a+xi+yj+zk \in \bH \, | \, x,y,z \in \bR, \,  x^2+y^2+z^2=b^2+c^2+d^2\}.
\]
In particular, $q$ and $\overline{q}$ are always conjugated, because $\overline{q} \in \bS_q$.

It is worth stressing that the following assertions are equivalent:
\begin{itemize}
 \item $p$ and $q$ are conjugated,
 \item $\bS_p=\bS_q$,
 \item $\mr{Re}(p)=\mr{Re}(q)$ and $|\mr{Im}(p)|=|\mr{Im}(q)|$,
 \item $\mr{Re}(p)=\mr{Re}(q)$ and $|p|=|q|$.
\end{itemize}

$(5)$ Equipped with the metric topology induced by the norm $|\cdot|$, $\bH$ results to be complete. Indeed, as a normed real vector space, it is isomorphic to $\bR^4$ endowed with the standard Euclidean norm. In particular, given a sequence $\{q_n\}_{n \in \bN}$ in $\bH$, if the series $\sum_{n \in \bN}q_n$ converges absolutely; that is, $\sum_{n \in \bN} |q_n|<+\infty$, then $\sum_{n \in \bN}q_n$ converges to some element $q$ of $\bH$ and such a series can be arbitrarily re--ordered without affecting its sum $q$.
\end{remark}


\subsection{Quaternionic Hilbert spaces}
Let recall the definition of quaternionic Hilbert space (see e.g.~\cite{BDS},\cite{Ng}).
Let $\sH$ be a right $\bH$--module; that is, an abelian group with a right scalar multiplication
\[
\sH \times \bH \ni (u, q) \mapsto uq \in \sH
\]
satisfying the distributive properties with respect to the two notions of  sum:
\[
(u+v)q=uq+vq \quad \mbox{and} \quad v(p+q)=vp+vq
\quad \mbox{if $u,v \in \sH$ and $p,q\in \bH$,}
\]
and the associative property with respect to the quaternionic product:
\[
v(pq)=(vp)q \quad \mbox{if $v \in \sH$ and $p,q \in \bH$.}
\]

Such a right $\bH$--module $\sH$ is called \emph{quaternionic pre--Hilbert space} if there exists a Hermitean quaternionic scalar product; that is, a map $\sH \times \sH \ni (u,v) \mapsto \b u|v \k \in \bH$ satisfying the following three properties:
\begin{itemize}
 \item (Right linearity) $\b u| vp +  wq \k = \b u | v\k p + \b u | w\k q$ if $p,q \in \bH$ and $u,v,w \in \sH$.
 \item (Quaternionic Hermiticity) $\b u|v \k = \overline{\b v | u \k}$ if $u,v \in \sH$.
 \item (Positivity) If $u \in \sH$, then $\b u| u \k \in \bR^+$ and $u=0$ if $\b u | u \k=0$.
\end{itemize}

Suppose that $\sH$ is equipped with such a Hermitean quaternionic scalar product. Then we can define the \emph{quaternionic norm $\|\cdot\|:\sH \lra \bR^+$ of $\sH$} by setting
\beq \label{qnorm}
\|u\| := \sqrt{\langle u| u\rangle} \quad \text{if $u \in \sH$}. 
\eeq
The function $\|\cdot\|$ is a genuine norm over $\sH$, viewed as a real vector space, and the above defined scalar product $\b \cdot | \cdot \k$ fulfills the standard Cauchy--Schwarz inequality.

\begin{proposition} \label{propnorm}
The Hermitean quaternionic scalar product $\b \cdot | \cdot \k$ satisfies the Cauchy--Schwarz inequality:
\beq \label{CS}
|\b u|v \k |^2 \leq \b u | u \k \, \b v | v \k \quad \mbox{if }u,v \in \sH.
\eeq
Moreover, the map $\sH \ni u \mapsto \|u\|\in \bR^+$ defined in $(\ref{qnorm})$ has the following properties:
\begin{itemize}
 \item $\|uq\| = \|u\| \:|q|$ if $u \in \sH$ and $q\in \bH$.
 \item $\|u+v\| \leq \|u\| + \|v\|$ if $u,v\in \sH$.
 \item If $\|u\|=0$ for some $u \in \sH$, then $u=0$.
 \item If $u,v \in \sH$, then the following polarization identity holds:
\begin{align}
4 \langle u| v\rangle =& \|u+v\|^2-\|u-v\|^2+\left(\|ui+v\|^2-\|ui-v\|^2 \right)i+ \nonumber\\
\label{polarization}
&+\left(\|uj+v\|^2-\|uj-v\|^2\right)j+\left(\|uk+v\|^2-\|uk-v\|^2\right)k.
\end{align}
\end{itemize}
\end{proposition}

We explicitly present the proof of the above statement just to give the flavour of the procedure taking the non--commutativity of the scalars into account. The result is evident if $v=0$. Assume $v \neq 0$. Concerning (\ref{CS}), we start from the inequality
\[
0 \leq \langle up - vq| up - vq\rangle = \overline{p}\langle u|u\rangle p +
\overline{q}\langle v|v\rangle q- \overline{p}\langle u|v\rangle q 
- \overline{q}\langle v|u\rangle p.
\] 
Choosing $p=\b v|v \k$ and  $q= \b v|u \k$, we obtain:
\[
0 \leq \b v|v \k \left( \b u|u \k   \b v|v \k   - \b u|v \k  \b v|u \k  \right).
\]
Since $\b v|v \k >0$ and $\b u|v \k \b v|u \k  =\overline{\b v|u \k} \b v|u \k=|\b v|u \k|^2=|\b u|v \k|^2$, (\ref{CS}) follows immediately. Let us pass to the norm properties. With regards homogeneity, exploiting the fact that $\langle u| u \rangle$ is real and thus it commutes with all quaternions, we obtain: $\|uq\|^2 = \b uq | uq \k = \overline{q}\b u| u \k q = \overline{q}q \b u | u \k=|q|^2 \|u\|^2$. Triangular inequality follows from (\ref{CS}). Indeed, bearing in mind that $|\mr{Re}(q)| \leq |q|$ for every $q \in \bH$, it holds:
\begin{align*}
\|u+v\|^2 =& \b u+v| u+v \k=\|u\|^2+\|v\|^2+2 \, \mr{Re}(\b u | v \k) \leq \\
\leq & \|u\|^2 + \|v\|^2 + 2 |\langle  u|v\rangle|
\leq \\
\leq & \|u\|^2 + \|v\|^2 + 2 \| u\| \|v\|=(\|u\|+\|v\|)^2.
\end{align*}
The polarization identity can be proved by direct inspection.

From the positivity and the triangular inequality, 
it follows that:
\beq \label{dist}
d(u,v):=\|u-v\| \quad \text{if } u,v \in \sH
\eeq
defines a (translationally invariant) distance. One easily verifies that, in view of the given definitions, the operations of sum of vectors and right multiplication of vectors and quaternions are continuous with respect to the topology induced by $d$. Similarly, the scalar product is jointly continuous as a map from $\sH \times \sH$ to $\bH$.

As soon as we are equipped with a metric topology, it makes sense to state the following definition.
\begin{definition}
The quaternionic pre--Hilbert space $\sH$ is said to be a \emph{quaternionic Hilbert space} if it is complete with respect to its natural distance $d$.
\end{definition}

\textit{In what follows, given a quaternionic Hilbert space $\sH$, we will implicitly assume that $\sH \neq \{0\}$}.

A function $f:\sH \to \bH$ is \emph{right $\bH$--linear} if $f(u+v)=f(u)+f(v)$ and $f(uq)=f(u)q$ for every $u,v \in \bH$ and $q \in \bH$. Let $\sH'$ be the \emph{topological dual space of $\sH$} consisting of all continuous right $\bH$--linear functions on $\sH$. Equip $\sH'$ with its natural structure of left $\bH$--module induced by the left multiplication by $\bH$: if $f \in \sH'$ and $q \in \bH$, then $qf$ is the element of $\sH'$ sending $u \in \sH$ into $qf(u) \in \bH$.

The following result is an immediate consequence of quaternionic Hahn--Banach theorem (see \cite[\S 2.10]{BDS} and \cite[\S 4.10]{libroverde}).

\begin{lemma} \label{lem:Hahn-Banach}
Let $\sH$ be a quaternionic Hilbert space and let $u \in \sH$. If $f(u)=0$ for every $f \in \sH'$, then $u=0$. 
\end{lemma}

It is possible to re--cast the definition of Hilbert basis even in the quaternionic--Hilbert--space context. Given a subset $A$ of $\sH$, we define: 
\[
A^\perp:=\{v \in \sH \, | \, \b v | u \k =0 \ \forall u \in A\}.
\]
Moreover, $<A>$ denotes the right $\bH$--linear subspace of $\sH$ consisting of all finite right $\bH$--linear combinations of elements of $A$.

Let $I$ be a non--empty set and let $I \ni i \mapsto a_i \in \bR^+$ be a function on $I$. As usual, we can define $\sum_{i \in I}a_i$ as the following element of $\R^+ \cup \{+\infty\}$:
\[
\sum_{i \in I} a_i:=\sup\left\{\sum_{i \in J} a_i \, \bigg| \, J \mbox{ is a non--empty finite subset of }I\right\}.
\]
It is clear that, if $\sum_{i \in I} a_i < +\infty$, then the set of all $i \in I$ such that $a_i \neq 0$ is at most countable.

Given a quaternionic Hilbert space $\sH$ and a map $I \ni i \mapsto u_i \in \sH$, one can say that the series $\sum_{i \in I}u_i$ converges absolutely if $\sum_{i \in I}\|u_i\|<+\infty$. If this happens, then only a finite or countable number of $u_i$ is nonzero and the series $\sum_{i \in I}u_i$ converges to a unique element of $\sH$, independently from the ordering of the $u_i$'s. 

The next three results can be established following the proofs of their corresponding complex versions (see e.g.~\cite{Moretti,RudinARC}).

\begin{proposition} \label{propoHB}
Let $\sH$ be a quaternionic Hilbert space and let 
$N$ be a subset of $\sH$ such that, for $z,z' \in N$, $\b z|z' \k=0$ if $z \neq z'$ and
$\b z| z \k=1$. Then conditions $(\mr{a})$--$(\mr{e})$ listed below are pairwise equivalent.
\begin{itemize}
\item[$(\mr{a})$] For every $u,v \in \sH$, the series $\sum_{z \in N} \b u|z \k \b z| v \k$ converges absolutely and it holds:
\[
\b u|v \k=\sum_{z \in N} \b u|z \k \b z| v \k.
\]

\item[$(\mr{b})$] For every $u \in \sH$, it holds:
\[
\|u\|^2=\sum_{z \in N} |\b z| u \k|^2.
\]

\item[$(\mr{c})$] $N^\perp = \{0\}$.

\item[$(\mr{d})$] $<N>$ is dense in $\sH$.
\end{itemize}
\end{proposition}

\begin{proposition}
Every quaternionic Hilbert space $\sH$ admits a subset $N$, called \emph{Hilbert basis of $\sH$}, such that, for $z,z' \in N$, $\b z|z' \k=0$ if $z \neq z'$ and $\b z | z \k=1$, and $N$ satisfies  equivalent conditions $(\mr{a})$--$(\mr{d})$ stated in the preceding proposition. Two such sets have the same cardinality. 

Furthermore, if $N$ is a Hilbert basis of $\sH$, then every $u \in \sH$ can be uniquely decomposed as follows:
\[
u=\sum_{z\in N} z \b z| u \k,
\]
where the series $\sum_{z\in N} z \b z| u \k$ converges absolutely in $\sH$.
\end{proposition}

\begin{theorem}\label{teoremaperp}
If $\sH$ is a quaternionic Hilbert space and $\emptyset \neq A\subset \sH$, then it holds:
\[A^{\perp} = <A>^\perp = \overline{<A>}^\perp = \overline{<A>^\perp}\:, \quad
\overline{\langle A\rangle} = (A^\perp)^\perp \ \text{ and }\ 
A^\perp \oplus <A> = \sH\:,\]
where the bar denotes the topological closure and  the symbol $\oplus$ denotes the ortho\-gonal direct sum.
\end{theorem}

Taking the previously introduced results into account, the representation Riesz' theorem extends to the quaternionic case. This result can be proved as in the complex case (see e.g.~\cite{RudinFA}).

\begin{theorem}[Quaternionic representation Riesz' theorem]\label{quat_Riesz}
If $\sH$ is a quaternionic Hilbert space, the map
\[
\sH \ni v \mapsto \langle v\, |\, \cdot\, \rangle \in \sH'
\]
is well--posed and defines a conjugate--$\bH$--linear isomorphism.
\end{theorem}

\subsection{Operators} \label{subsec:operators}
 
First of all, we present the general definition of what we mean by a right $\bH$--linear operator.
 
\begin{definition}
Let $\sH$ be a quaternionic Hilbert space.  A \emph{right $\bH$--linear operator} is a map $T: D(T) \lra \sH$ such that:
\[
T(ua+vb)=(Tu)a+(Tv)b \quad \mbox{if }u,v \in D(T) \mbox{ and }a,b \in \bH,
\]
where the \emph{domain} $D(T)$ of $T$ is a (not necessarily closed) right $\bH$--linear subspace of $\sH$. We define the \emph{range $\Ran(T)$ of $T$} by setting $\Ran(T):=\{Tu \in \sH \, | \, u \in D(T)\}$.
\end{definition}
 
\emph{In what follows, by the term ``operator'', we mean a ``right $\bH$--linear operator''. Similarly, by a ``subspace'', we mean a ``right $\bH$--linear subspace''.}

Let $T: D(T) \lra \sH$ and $S: D(S) \lra \sH$ be operators. As usual, we write $T\subset S$ if $D(T) \subset D(S)$ and $S \rr_{D(T)}=T$. In this case, $S$ is said to be an extension of $T$. We define the \emph{natural domains} of the sum $T+S$ and of the composition $TS$ by setting $D(T+S):= D(T) \cap D(S)$ and $D(TS):=\{x \in D(S) \, | \, Sx \in D(T)\}$. Here we use the symbols $TS$ and $Sx$ in place of $T \circ S$ and $S(x)$, respectively.

The operator $T$ is said to be \emph{closed} if the graph $\cG(T):=D(T) \oplus \Ran(T)$ of $T$ is closed in $\sH \times \sH$, equipped with the product topology. Finally, $T$ is called \emph{closable} if it admits closed operator extensions. In this case, the \emph{closure $\overline{T}$ of $T$} is the smallest closed extension.
 
We have the following elementary, though pivotal, result that permits the introduction of  the notion of bounded operator. The proof is the same as for complex Hilbert spaces (see e.g.~\cite{Moretti,RudinARC}).
 
\begin{theorem}
If $\sH$ is a quaternionic Hilbert space, then an operator $T: D(T) \lra \sH$ is continuous if and only if is bounded; that is, there exists $K \geq 0$ such that
\[
\|Tu\| \leq K \|u\| \quad \mbox{if } u \in D(T).
\]
Furthermore, a bounded operator is closed.
\end{theorem}

As in the complex case, if $T: D(T) \lra \sH$ is any operator, one define $\|T\|$ by setting
\beq \label{QN}
\|T\|:=\sup_{u \in D(T) \setminus \{0\}} \frac{\|Tu\|}{\|u\|}=\inf\{K \in \bR \:|\: \|Tu\| \leq K \|u\| \ \forall u \in D(T)\}.
\eeq
Denote by $\gB(\sH)$ the set of all bounded (right $\bH$--linear) operators of $\sH$: 
\[
\gB(\sH):=\{T:\sH \lra \sH \, \text{ operator} \, | \, \|T\|<+\infty\}.
\]
It is immediate to verify that, if $T$ and $S$ are operators in $\gB(\sH)$, then the same is true for $T+S$ and $TS$, and it holds:
\beq \label{eq:sum-prod}
\|T+S\| \leq \|T\|+\|S\| \quad \mbox{and} \quad \|TS\| \leq \|T\| \, \|S\|.
\eeq

The reader observes that $\gB(\sH)$ has a natural structure of real algebra, in which the sum is the usual pointwise sum, the product is the composition and the real scalar multiplication $\gB(\sH) \times \bR \ni (T,r) \mapsto Tr \in \gB(\sH)$ is defined by setting
\beq \label{eq:rT}
(Tr)(u):=T(u)r.
\eeq
In Section \ref{subsec:two-sided}, we will extend this real algebra structure to a quaternionic two--sided Banach unital $C^*$-algebra structure. The reader observes that definition (\ref{eq:rT}) can be repeated for every operator $T:D(T) \lra \sH$.

It is worth introducing here a notion, which will be useful later. As usual, we denote by $\bR[X]$ the ring of real polynomials in the indeterminate $X$. For conve\-nience, we write the polynomials in $\bR[X]$ \textit{with coefficients on the right}. Given $P(X)=\sum_{h=0}^dX^hr_h$ in $\bR[X]$, we define the operator $P(T) \in \gB(\sH)$ as follows:
\beq \label{eq:P(T)}
P(T):=\sum_{h=0}^dT^hr_h,
\eeq 
where $T^0$ is considered to be equal to the identity operator $\1:\sH \lra \sH$ of $\sH$.

The norm of $\gB(\sH)$ allows us to define a metric $D:\gB(\sH)\times\gB(\sH) \lra \R^+$ on $\gB(\sH)$ as follows:
\beq \label{distance}
D(T,T'):=\|T-T'\| \quad \text{if $T,T' \in \gB(\sH)$}.
\eeq

\begin{proposition} \label{propcompleteness}
Let $\sH$ be a quaternionic Hilbert space. Equip $\gB(\sH)$ with the metric $D$. The following assertions hold:
\begin{itemize}
\item[$(\mr{a})$] $\gB(\sH)$ is a complete metric space. In particular, it is a Baire space.
\item[$(\mr{b})$]
As maps $\gB(\sH) \times \gB(\sH) \mapsto \gB(\sH)$, the sum and the composition of opera\-tors are continuous. The same is true for the real scalar multiplication.
\item[$(\mr{c})$] The subset of $\gB(\sH)$ consisting of elements admitting two--sided inverse in $\gB(\sH)$ is open in $\gB(\sH)$. 
\item[$(\mr{d})$]  The \emph{uniform boundedness principle} holds: Given any subset $F$ of $\gB(\sH)$, if $\sup_{T \in F} |Tu|<+\infty$ for every $u \in \sH$, then $\sup_{T \in F}\|T\|<+\infty$.
\item[$(\mr{e})$] The \emph{open map theorem} holds: If $T\in \gB(\sH)$ is surjective, then $T$ is open. In particular, if $T$ is bijective, then $T^{-1} \in \gB(\sH)$.
\item[$(\mr{f})$] The \emph{closed graph theorem} holds: If $T:\sH \lra \sH$ is closed, then $T \in \gB(\sH)$.
\end{itemize}
\end{proposition} 
\begin{proof} The proofs of $(\mr{a})$, $(\mr{b})$ are elementary.  Point $(\mr{c})$ is proved as in \cite[Theorem 10.12]{RudinFA}. Points $(\mr{d})$, $(\mr{e})$ and $(\mr{f})$ follow from the Baire theorem exactly as in complex Banach space theory: Lemma 2.2.3, Theorem 2.2.4 and Corollary 2.2.5 of \cite{Analysisnow} prove $(\mr{e})$; Theorem 2.2.9 of \cite{Analysisnow} leads to $(\mr{d})$ and Theorem 2.2.7 
proves $(\mr{f})$.
\end{proof}

The notion of adjoint operator is the same as for complex Hilbert spaces.

\begin{definition}
Let $\sH$ be a quaternionic Hilbert space and let $T:D(T) \lra \sH$ be an operator with dense domain. The adjoint $T^*: D(T^*) \lra \sH$ of $T$ is the unique operator with the following properties (the definition being well posed since $D(T)$ is dense):
\[
D(T^*):= \{u \in \sH \, | \, \exists w_u \in \sH \mbox{ with } \b w_u|v \k=\b u | T v \k \ \forall v \in D(T)\}.
\]
and
\beq \label{defagg}
\b T^* u | v \k=\b u | T v \k \quad \forall v \in D(T), \forall u \in D(T^*).
\eeq
\end{definition}

It is worth noting that an immediate consequence of such a definition is that the operator $T^*$ is always closed. 
Moreover, if $T \in \gB(\sH)$, then requirement (\ref{defagg}) alone automatically determines $T^*$ as an element of $\gB(\sH)$ in view of quaternionic representation Riesz' theorem (see Theorem \ref{quat_Riesz} above).

As usual, an operator $T:D(T) \lra \sH$ with dense domain is said to be:
\begin{itemize}
 \item \emph{symmetric} if $T \subset T^*$,
 \item \emph{anti symmetric} if $T \subset -T^*$,
 \item \emph{self--adjoint} if $T=T^*$,
 \item \emph{essentially self--adjoint} if it is closable and $\overline{T}$ is self--adjoint,
 \item \emph{anti self--adjoint} if $T=-T^*$,
 \item \emph{positive}, and we write $T \geq 0$, if $\b u | Tu \k \in \R^+$ for every $u \in D(T)$,
 \item \emph{normal} if $T \in \gB(\sH)$ and $TT^*=T^*T$,
 \item \emph{unitary} if $D(T)=\sH$ and $TT^*=T^*T=\1$.
\end{itemize}

\begin{remark} \label{propun}
$(\mr{1})$ By definition, if $U$ is a unitary operator of $\sH$, then $\b Uu | Uv \k=\b u | v \k$ for every $u,v \in \sH$ and $U$ is isometric; that is, $\|Uu\|=\|u\|$ for every $u \in \sH$. In particular, $U$ belongs to $\gB(\sH)$. As for complex Hilbert spaces, an operator $U:\sH \lra \sH$ is unitary if and only if it is isometric and surjective. 

$(\mr{2})$ If $N, N' \subset \sH$ are Hilbert bases, then there exists a unitary operator $U \in \gB(\sH)$ such that $U(N)=N'$. Moreover, if $V \in \gB(\sH)$ is unitary, then $\{V z \in \sH \,| \, z \in N\}$ is a Hilbert basis of $\sH$.
\end{remark}

Now we state quaternionic versions of three well--known results for complex Hilbert spaces. The former is an immediate consequence of Theorem \ref{teoremaperp} and of identity (\ref{defagg}).

\begin{proposition}\label{lemma0}
Let $\sH$ be a quaternionic Hilbert space and let $T:D(T) \lra \sH$ be an operator with domain dense in $\sH$. We have:
\[
\mi{Ran}(T)^\perp =\mi{Ker}(T^*)
\quad \mbox{and} \quad
\mi{Ker}(T) \subset \mi{Ran}(T^*)^\perp.
\]

Furthermore, if $D(T^*)$ is dense in $\sH$ and $T$ is closed, then we can replace ``$\subset$'' with ``$=$''. In particular, this is true when $T \in \gB(\sH)$.
\end{proposition}

\begin{theorem} \label{teoopagg}
Let $\sH$ be a quaternionic Hilbert space and let $T: D(T) \lra \sH$ be an operator with dense domain. The following facts hold.
\begin{itemize}
 \item[$(\mr{a})$] $T^*$ is closed and, if $U \in \gB(\sH \oplus \sH)$ is the unitary operator sending $(x,y)$ into $(-y,x)$, then the following orthogonal decomposition holds:
\beq \label{decgraf}
\sH \oplus \sH = \overline{\cG(T)} \oplus U(\cG(T^*)),
\eeq
where the bar denotes the closure in $\sH \oplus \sH$.
 \item[$(\mr{b})$] $T$ is closable if and only if $D(T^*)$ is dense in $\sH$ and $\overline{T}=(T^*)^*$.
 \item[$(\mr{c})$] If $T$ is closed, injective and has image dense in $\sH$, then the same is true for $T^*$ and for $T^{-1}:\mi{Ran}(T) \lra D(T)$. Moreover, under these hypotheses, it holds: $(T^*)^{-1}= (T^{-1})^*$.
 \item[$(\mr{d})$] If $T$ is closed, then $D(T^*T)$ is dense in $\sH$ and $T^*T$ is self--adjoint.
 \item[$(\mr{e})$] \emph{(Helling--Toeplitz Theorem)} If $T$ is self--adjoint or anti self--adjoint and $D(T)=\sH$, then $T\in \gB(\sH)$.
 \end{itemize}
\end{theorem}

\begin{proof} The proof are the same as for complex Hilbert spaces, since they do not depend on the fact that the Hilbert space 
is defined over $\bC$ or $\bH$. Points $(\mr{a})$ and $(\mr{b})$ can be proved exactly as Theorem 5.1.5 in \cite{Analysisnow}. One can prove $(\mr{c})$ exactly as Proposition 5.1.7 in \cite{Analysisnow} and $(\mr{d})$ as point $(\mr{i})$ in Theorem 5.1.9 in \cite{Analysisnow}.
Point $(\mr{e})$ follows from $(\mr{a})$ taking into account the closed graph theorem.
 \end{proof}

\begin{remark} \label{remarkop}
Exactly as in the case of complex Hilbert space (see e.g.~\cite{Moretti,RudinFA}), one can prove that, if $S:D(S) \lra \sH$ and $T:D(T) \lra \sH$ are operators with $D(T)$ and $D(S)$ dense in $\sH$, then it hold:
\begin{itemize}
 \item[(i)]  $S^* \subset T^*$ if $T \subset S$.
 \item[(ii)] If $T \in \gB(\sH)$, then  $T^* \in \gB(\sH)$, $\|T^*\|=\|T\|$ and $\|T^*T\|=\|T\|^2$.
 \item[(iii)] $T \subset (T^*)^*$ if $D(T^*)$ is dense in $\sH$, and $T=(T^*)^*$ if $T \in \gB(\sH)$.
 \item[(iv)] $T^*+S^* \subset (T + S)^*$ if $D(T+S)$ is dense in $\sH$, and $(T+S)^*=T^*+S^*$ if $T \in \gB(\sH)$.
 \item[(v)] $S^*T^* \subset (TS)^*$ if $D(TS)$ is dense in $\sH$, and $(TS)^*=S^*T^*$ if $T \in \gB(\sH)$.
 \item[(vi)] If $T$ is self--adjoint, $S$ is symmetric and $T \subset S$, then $S=T$.
 \item[(vii)] $T$ is essentially self--adjoint if and only if $D(T^*)$ is dense in $\sH$ and $T^*$ is self--adjoint. In that case, $T^*=\overline{T}$ and $\overline{T}$ is the only self--adjoint extension of~$T$.
 \item[$(\mr{viii})$] Let $T \in \gB(\sH)$. If $T$ is bijective and $T^{-1} \in \gB(\sH)$, then  $T^*(T^{-1})^*=\1=(T^{-1})^*T^*$. In this way, $T$ is bijective and $T^{-1} \in \gB(\sH)$ if and only if $T^*$ is bijective and $(T^*)^{-1} \in \gB(\sH)$. Moreover, in that situation, it holds: $(T^*)^{-1}=(T^{-1})^*$.
 \item[$(\mr{ix})$] $(Tr)^*=T^*r$ if $r \in \bR$.
\end{itemize}
\end{remark}

The following last proposition deserves an explicit proof as it is different from that 
in complex Hilbert spaces.
 
\begin{proposition} \label{Lemmadiag}
Let $\sH$ be a quaternionic Hilbert space and let $T:D(T) \lra \sH$ be an operator with dense domain. The following assertions hold.
\begin{itemize}
 \item[$(\mr{a})$] If $T$ is closed (for example, if $T \in \gB(\sH)$) and there exists a dense subspace $L$ of $D(T)$ such that $\b u | Tu \k=0$ for every $u \in L$, then $T=0$.
 \item[$(\mr{b})$] If $\b u | Tu \k \in \R$ for every $u \in D(T)$ (for example, if $T \geq 0$), then $T$ is symmetric and hence it is self--adjoint in the case in which $D(T)=\sH$.
\end{itemize}
\end{proposition}

\begin{proof}
$(\mr{a})$ Since $D(T)$ is dense in $\sH$, $L$ is dense in $\sH$ as well. Fix $u,v \in L$. For every $q \in \bH$, $u+vq$ belongs to $L$ and hence $0=\b u + vq|T(u + vq)\k=\b u | Tv \k q +\bar q \b v | Tu \k$. In particular, taking $q=1$, we get: $\b u | Tv \k + \b v | Tu \k=0$, from which we deduce that
\[
\b u | Tv \k q -\bar q \b u | Tv \k=0 \quad \mbox{for every }q \in \bH.
\]
Defining $a:=\b u | Tv \k=a_0+\jmath a_1$ for some $a_0,a_1 \in \bR$ and $\jmath \in \bS$, we get: $0=a\jmath + \jmath a=(a_0+\jmath a_1)\jmath+\jmath (a_0+\jmath a_1)=-2a_1+2\jmath a_0$. Therefore we have that $a_0=a_1=0$ and hence $\b u | Tv \k=0$.
Choosing a sequence $L \supset \{u_n\}_{n \in \bN} \to Tv$, we infer that $\|Tv\|^2=0$ (and hence $Tv=0$) for every $v \in L$. Finally, if $x \in D(T)$ and $L \supset \{x_n\}_{n \in \bN} \to x$, then we have that $Tx_n=0$ for every $n \in \bN$, so $\{(x_n, Tx_n)\}_{n \in \bN} \to (x,0)$. Since $T$ is closed, it follows that $Tx=0$.

$(\mr{b})$ Let us follow the strategy used in the proof of point $(\mr{a})$. First, observe that $\b x | Tx \k=\overline{\b x | Tx \k}=\b Tx | x \k$ for every $x \in D(T)$. Fix $u,v \in D(T)$. For every $q \in \bH$, we have:
\begin{align*}
0=& \b u+vq | T(u+vq) \k-\b T(u+vq) | u+vq \k= \\
=& (\b u | Tv \k-\b Tu | v \k)q+\bar q(\b v | Tu \k-\b Tv | u \k).
\end{align*}
Define $b:=\b u | Tv \k-\b Tu | v \k$. Choosing $q=1$, we obtain that $bq-\bar q b=0$ for every $q \in \bH$. Proceeding as above, we infer that $b=0$ or, equivalently, that $\b u | Tv \k=\b Tu | v \k$ for every $u,v \in D(T)$. This proves the symmetry of $T$.
\end{proof}

\subsection{Square root and polar decomposition of operators}

The theorem of existence of the square root of positive bounded operators works exactly as 
in the case of complex Hilbert spaces. The reason is that the proof exploits the convergence of sequences of real polynomials of operators in the strong operator topology (cf.\ \cite{Emch} and \cite[Propositions~3.2.11 and 3.2.12, Theorem~3.2.17]{Analysisnow}).

\begin{theorem} \label{radT}
Let $\sH$ be a quaternionic Hilbert space and let $T \in \gB(\sH)$. If $T \geq 0$, then there exists a unique operator in $\gB(\sH)$, indicated by $\sqrt{T}$, such that $\sqrt{T} \geq 0$ and $\sqrt{T}\sqrt{T}=T$. Furthermore, it turns out that $\sqrt{T}$ commutes with every operator which commutes with $T$.
\end{theorem}

Even for quaternionic Hilbert spaces, one has the polar decomposition theorem. Before presenting our quaternionic version of such a theorem, we need some preparations.

\begin{proposition} \label{lemma3}
Let $\sH$ be a quaternionic Hilbert space and let $T \in \gB(\sH)$ be a normal operator. Then we have that $\|Tu\|=\|T^*u\|$ for every $u \in \sH$ and it holds:
\[
\mi{Ker}(T)=\mi{Ker}(T^*)
\quad \mbox{and} \quad
\overline{\mi{Ran}(T)}=\overline{\mi{Ran}(T^*)}.
\]
In particular, $\sH$ orthogonally decomposes as:
\beq \label{decH}
\sH=\mi{Ker}(T) \oplus \overline{\mi{Ran}(T)}.
\eeq
\end{proposition}
\begin{proof} Given any $u \in \sH$, we have:
\[
\|Tu\|^2=\b Tu | Tu \k=\b u | T^*Tu \k=\b u | TT^*u \k=\b T^*u | T^*u \k=\|T^*u\|^2.
\]
This implies immediately that $\mi{Ker}(T)=\mi{Ker}(T^*)$. By combining the latter equality with Proposition \ref{lemma0}, we obtain immediately the equality $\overline{\mi{Ran}(T)}=\overline{\mi{Ran}(T^*)}$ and the orthogonal decomposition (\ref{decH}). 
\end{proof}

Let $T \in \gB(\sH)$. The reader observes that $T^*T \geq 0$. Indeed, it holds:
\[
\b u | T^*Tu \k=\b Tu | Tu \k=\|Tu\|^2 \geq 0
\quad \mbox{if }u \in \sH.
\]
As usual, we define $|T| \in \gB(\sH)$ by setting
\[
|T|:= \sqrt{T^*T}.
\]
By point $(\mr{b})$ of Proposition \ref{Lemmadiag}, the operator $|T|$ is self--adjoint. Moreover, it has the same kernel of $T$:
\beq \label{norm}
\||T|(u)\|^2=\b |T|(u) | |T|(u) \k=\b u | |T|^2(u) \k=\b u | T^*Tu \k=\|Tu\|^2
\eeq
for every $u \in \sH$. In this way, if $T$ is normal, one also has that
\beq \label{treeq}
\mi{Ker}(T)=\mi{Ker}(T^*)=\mi{Ker}(|T|)
\quad \mbox{and} \quad
\overline{\mi{Ran}(T)}=\overline{\mi{Ran}(T^*)}= \overline{\mi{Ran}(|T|)}.
\eeq

We are in position to state and prove our quaternionic version of the polar decomposition theorem. To the best of our knowledge, such a theorem has been mentioned and used several times in the mathematical physics literature, without an explicit proof.

\begin{theorem}\label{polar-polar} 
Let $\sH$ be a quaternionic Hilbert space and let $T \in \gB(\sH)$ be an operator. Then there exist, and are unique, two operators $W$ and $P$ in $\gB(\sH)$ such that: 
\begin{itemize}
 \item[$(\mr{i})$] $T=WP$,
 \item[$(\mr{ii})$] $P \geq 0$, 
 \item[$(\mr{iii})$] $\mi{Ker}(P) \subset \mi{Ker}(W)$,
 \item[$(\mr{iv})$] $W$ is isometric on $\mi{Ker}(P)^\perp$; that is, $\|Wu\|=\|u\|$ for every $u \in \mi{Ker}(P)^\perp$.
\end{itemize}
Furthermore, $W$ and $P$ have the following additional properties:
\begin{itemize}
 \item[$(\mr{a})$]
 $P=|T|$.
 \item[$(b)$] If $T$ is normal, then $W$ defines a unitary operator in $\gB(\overline{\mi{Ran}(T)})$.
 \item[$(\mr{c})$] If $T$ is normal, then $W$ commutes with $|T|$ and with all the operators in $\gB(\sH)$ commuting with both $T$
and $T^*$.
 \item[$(\mr{d})$] $W$ is self--adjoint if $T$ is.
 \item[$(\mr{e})$] $W$ is anti self--adjoint if $T$ is.
\end{itemize}
\end{theorem}
\begin{proof}
Firstly, we show that, if there exists a decomposition $(\mr{i})$ of $T$ with properties $(\mr{ii})$, $(\mr{iii})$ and $(\mr{iv})$, then it is unique. Thanks to $(\mr{ii})$ and Proposition \ref{Lemmadiag},$(\mr{b})$, $P$ is self--adjoint and hence $\mi{Ker}(P)^\perp=\overline{\mi{Ran}(P)}$ and $T^*T=PW^*WP$. By using $(\mr{iv})$ and the polarization identity (see (\ref{polarization})), we know that $\b WPu | WPv \k=\b Pu | Pv \k$ or, equivalently, $0=\b u | (PW^*WP-P^2)v \k=\b u | (T^*T-P^2)v \k$ for every $u,v \in \sH$. This is equivalent to say that $T^*T=P^2$. Theorem \ref{radT} ensures that $P$ coincides with $|T|$ and hence $(\mr{a})$ holds. In particular, $P$ is unique. Let us show the uniqueness of~$W$. Since $\sH=\mi{Ker}(P) \oplus \mi{Ker}(P)^\perp=\mi{Ker}(P) \oplus \overline{\mi{Ran}(P)}$ and
$W$ vanishes on $\mi{Ker}(P)$, to determine it, it is enough fixing it on $\mi{Ker}(P)^\perp = \overline{Ran(P)}$. Actually, this is done  by the requirement $T= WP$ itself, that fixes $W$ on $Ran(P)$ and thus on the whole $\overline{Ran(P)}$, because $W$ is continuous.

Let us prove the existence of a polar decomposition. Let $P:=|T|$. Define $W$ as follows. If $v \in \mi{Ker}(P)$, then $Wv:=0$. If $v \in \mi{Ran}(P)$, then $Wv:=Tu$, where $u$ is any element of $\sH$ such that $v=Pu$. Since $\mi{Ker}(P)=\mi{Ker}(T)$ (see (\ref{norm})), the latter definition is well posed. Using (\ref{norm}) again, we infer that $W$ is isometric on $\mi{Ran}(P)$: $\|Wv\|=\|Tu\|=\||T|u\|=\|v\|$ for every $v \in \mi{Ran}(P)$. By continuity, $W$ can be uniquely defined on the whole $\overline{\mi{Ran}(P)}=\mi{Ker}(P)^\perp$. Evidently, $W$ and $P$ satisfy properties $(\mr{i})$--$(\mr{iv})$.

In the remainder of the proof, we suppose that $T$ is normal.

Let us prove $(\mr{b})$. The normality of $T$ ensures that $\overline{\mi{Ran}(T)}=\overline{\mi{Ran}(P)}$ (see (\ref{treeq})). Since $T=WP$, it is evident that $W(\mi{Ran}(P))=\mi{Ran}(T)$. By continuity, we have that $W(\overline{\mi{Ran}(T)})=W(\overline{\mi{Ran}(P)}) \subset \overline{\mi{Ran}(T)}$. In this way, we have that $\mi{Ran}(T) \subset W(\overline{\mi{Ran}(T)}) \subset \overline{\mi{Ran}(T)}$. On the other hand, by $(\mr{iv})$, $W$ is isometric on $\mi{Ker}(P)^\perp=\overline{\mi{Ran}(T)}$ and hence $W(\overline{\mi{Ran}(T)})$ is closed in $\sH$. It follows that $W(\overline{\mi{Ran}(T)})=\overline{\mi{Ran}(T)}$. The restriction of $W$ from $\overline{\mi{Ran}(T)}$ into itself turns out to be a surjective isometric operator or, equivalently, a unitary operator. 

By $(\mr{iv})$ and $(\mr{a})$, we have that $W^*W=\1$ on $\mi{Ker}(|T|)^\perp=\overline{\mi{Ran}(|T|)}$. 
In this way, using the equalities $T=W|T|$ and 
$T^*T=TT^*$, we infer that $|T|^2=W|T|^2W^*$. Applying $W^*$ on the left, it arises: $W^*|T|^2=|T|^2W^*$ and, taking the adjoint,
$|T|^2W=W|T|^2$. Since $|T|=\sqrt{|T|^2}$ commutes with all of the operators commuting with $|T|^2$, we get that $(|T|W-W|T|)|T|=0$ and hence $|T|W=W|T|$ on $\overline{\mi{Ran}(|T|)}$.    
Since $W$ vanishes on $\mi{Ker}(|T|)$, it follows that $|T|W=W|T|$ on the whole $\sH$.

Suppose now that $A \in \gB(\sH)$ commute with $T$ and $T^*$. Since $A$ commutes with $T$, it follows that $A$ preserves the decomposition $\sH=\mi{Ker}(T) \oplus \overline{\mi{Ran}(T)}$. As we know, the operator $|T|$ commutes with all operators commuting with $T$ and $T^*$. In particular, it commutes with $A$. In this way, the equality $AW|T|=W|T|A$ (recall that $T=W|T|$) is equivalent to the following one: $(AW-WA)|T|=0$. It follows that $A$ commutes with $W$ on $\overline{\mi{Ran}(T)}$. Since $W$ vanishes on $\mi{Ker}(T)$, $A$ trivially commutes with $W$ on $\mi{Ker}(T)$. This proves that $A$ commutes with $W$ and hence~$(\mr{c})$.

Let us prove $(\mr{d})$. Since $|T|$ is self--adjoint and commutes with $W$, $|T|$ commutes also with $W^*$. In this way, assuming $T$ self--adjoint, we have that $0=T-T^*=(W-W^*)|T|$ and hence $W=W^*$ on $\mi{Ker}(T)^\perp=\overline{\mi{Ran}(T)}$. Since $W$ preserves the decomposition $\sH = Ker(T) \oplus \overline{Ran(T)}$ and vanishes on $ Ker(T)$, one easily gets that $W^*$ vanishes on $Ker(T)$ as well. We infer that $W^*=W$, as desired.

It remains to show $(\mr{d})$. Exactly as we have done to prove $(\mr{c})$, if $T$ is anti self--adjoint, one proves that $0=T+T^*=(W+W^*)|T|$. Thus $W=-W^*$ on 
$\mi{Ker}(T)^\perp$. Since $W=0=-W^*$ on $\mi{Ker}(T)$, we have that $W^*=-W$.
\end{proof}

\begin{remark} \label{remarkbastardo} 
What is important to stress here is that, for a normal operator $T$, the three operators appearing in the decomposition $T= W|T|$ separately respect decomposition (\ref{decH}). In other words, $T$, $W$ and $P$ separately map $\mi{Ker}(T)$ and $\overline{\mi{Ran}(T)}$ into themselves. This circumstance will turn out useful in several proofs along all this work.
\end{remark}

\begin{remark} \label{rem}
The reader observes that, it being $\sH=\mi{Ker}(P) \oplus \mi{Ker}(P)^\perp$, preceding condition $(\mr{iii})$ can be replaced by the following one: $(\mr{iii}^{\pr})$ $\mi{Ker}(P)=\mi{Ker}(W)$. 
\end{remark}


\section{Left multiplications, imaginary units and complex subspaces}
In this part, we introduce a notion that is proper of quaternionic modules, since it arises from the non--commutativity of the algebra of quaternions. It is the notion of left scalar multiplication, which makes both $\sH$ and $\gB(\sH)$  quaternionic two--side Banach modules. As a byproduct, the notion of (operatorial) imaginary unit is also presented. This notion plays a crucial r\^ole in spectral theory over quaternionic Hilbert spaces, giving rise to a nice relationship with the theory in complex Hilbert spaces. This fact will be evident shortly, especially in view of Proposition \ref{propestensione}.

\subsection{Left scalar multiplications} \label{subsec:leftprod}
Let us show that it is possible to equip  $\sH$ with a left multiplication with quaternions. It will be a highly non--canonical operation relying upon a choice of a preferred Hilbert basis. So, pick out a Hilbert basis $N$ of $\sH$ and define the \emph{left scalar multiplication of $\sH$ induced by $N$} as the map $\bH \times \sH \ni (q,u) \mapsto qu \in \sH$ given by 
\beq \label{leftp}
q u:=\sum_{z \in N} z q \b z|u \k \quad \text{if $u \in \sH$ and $q \in \bH$.}
\eeq

\begin{proposition}\label{propprod} 
The left product defined in (\ref{leftp}) satisfies the following properties.
\begin{itemize}
\item[$(\mr{a})$] $q(u+v) = qu + qv$ and $q(up)= (qu)p$ for every $u \in \sH$ and $q,p \in \bH$.

\item[$(\mr{b})$] $\|q u\|=|q|\: \|u\|$ for every $u \in \sH$ and $q \in \bH$.

\item[$(\mr{c})$] $q(q'u) = (qq')u$ for every $u \in \sH$ and $q,q' \in \bH$.
 
\item[ $(\mr{d})$] $\langle \overline{q}u| v\rangle = \langle u| qv\rangle$ for every $u,v \in \sH$ and $q\in \bH$.

\item[$(\mr{e})$] $r u = ur$ for every $u \in \sH$ and $r \in \bR$.

\item[$(\mr{f})$] $qz=zq$ for every $z\in N$ and $q\in\bH$.

\end{itemize}

Consequently, for every $q \in \bH$, the map $L_q:\sH \lra \sH$, sending $u$ into $qu$, is an element of $\gB(\sH)$. Moreover, the  map $\LL_N:\bH \lra \gB(\sH)$, defined by setting
\[
\LL_N(q):=L_q,
\]
is a norm--preserving real algebra homomorphism, with the additional properties:
\beq \label{ur}
L_ru=ur \quad \text{if $r \in \bR$ and $u \in \sH$}
\eeq
and
\beq \label{lq}
(L_q)^*=L_{\overline{q}} \quad \mbox{if $q \in \bH$}.
\eeq

Finally, if $N'$ is another Hilbert basis of $\sH$ and $U \in \gB(\sH)$ is any unitary operator such that $U(N)=N'$, then it holds: 
\begin{itemize}
 \item[$\mr(f)$] $\LL_{N'}(q)=U\LL_N(q)U^{-1}$ for every $q \in \bH$.
 \item[$(\mr{g})$]
$\LL_N=\LL_{N'}$ if and only if $\b z | z' \k \in \bR$ for every $z \in N$ and $z' \in N'$.
\end{itemize}
\end{proposition}

\begin{proof}
All statements straightforwardly follows from Proposition \ref{propoHB}. However, the proof of (g) deserves some comments. Bearing in mind that $z=\sum_{z' \in N'} z'\b z' | z \k$, the equality $\LL_N=\LL_{N'}$ holds if and only if
\[
\sum_{z' \in N'} z' \b z' | z \k q=zq=\LL_N(q)(z)=\LL_{N'}(q)(z)= \sum_{z'\in N'} z'q\b z' | z \k
\]
for every $z \in N$ and $q \in \bH$. This is equivalent to say that, fixed any $z \in N$ and $z' \in N'$, $\b z'|z \k q = q \b z'|z \k$ for every $q \in \bH$, which in turn is equivalent to require that $\b z'|z \k \in \bR$.
\end{proof}

\begin{remark} \label{rem:ur=ru}
$(1)$ All possible different notions of left scalar multiplication have to coincide when multiplying with real quaternions in view of point $(\mr{e})$ of Proposition \ref{leftp}, because $ur$ does not depend on the choice of $N$.

$(2)$ Thanks to Proposition \ref{leftp}, given a left scalar multiplication $\bH \times \sH \ni (q,u) \mapsto L_q(u)=qu \in \sH$, we know that the map $\bH \ni q \mapsto L_q \in \gB(\sH)$ is a norm--preserving real algebra homomorphism satisfying (\ref{ur}) and (\ref{lq}). Conversely, it can be shown that every norm--preserving real algebra homomorphism from $\bH$ to $\gB(\sH)$ satisfying (\ref{ur}) and (\ref{lq}) is a left scalar multiplication induced by some Hilbert basis of $\sH$ (see \cite{GhMoPe2} for the proof).
\end{remark}

\textit{In what follows, by the symbol $\bH \ni q \mapsto L_q$, we mean the left scalar multiplication $L_qu:=\sum_{z \in N} z q \b z|u \k$ of $\sH$ induced by some fixed Hilbert basis $N$ of $\sH$ itself}.


\subsection{Quaternionic two-sided Banach $\boldsymbol{C^*}$-algebra structure on $\boldsymbol{\gB(\sH)}$} \label{subsec:two-sided}

In order to avoid any possible misunderstanding, we specify the meaning we give to the notion of quaternionic two--sided Banach $C^*$--algebra with unity. Such a notion coincides with the one used in \cite{libroverde} (see also \cite{BDS}).

We call \textit{quaternionic two--sided vector space} a non--empty subset $V$ equipped with a \textit{sum} $V \times V \ni (u,v) \mapsto u+v \in V$, with a \textit{left scalar multiplication} $\bH \times V \ni (q,u) \mapsto qu \in V$ and with a \textit{right scalar multiplication} $V \times \bH \ni (u,q) \mapsto uq$ such that $V$ is an abelian group with respect to the sum and it holds:
\begin{itemize}
 \item[$(v1)$] $q(u+v)=qu+qv \;$ and $\; (u+v)q=uq+vq$,
 \item[$(v2)$] $(q+p)u=qu+pu \;$ and $\; u(q+p)=uq+up$,
 \item[$(v3)$] $q(pu)=(qp)u \;$ and $\; (up)q=u(pq)$,
 \item[$(v4)$] $(qu)p=q(up)$,
 \item[$(v5)$] $u=1u=u1$,
 \item[$(v6)$] $ru=ur$
\end{itemize}
for every $u,v \in V$, $q,p \in \bH$ and $r \in \bR$. Such a quaternionic two--sided vector space $V$ is said to be a \textit{quaternionic two--sided algebra} if, in addition, it is equipped with a \textit{product} $V \times V \ni (u,v) \mapsto uv \in V$ such that
\begin{itemize}
 \item[$(a1)$] $u(vw)=(uv)w$,
 \item[$(a2)$] $u(v+w)=uv+uw \;$ and $\; (u+v)w=uw+vw$,
 \item[$(a3)$] $q(uv)=(qu)v \;$ and $\; (uv)q=u(vq)$
\end{itemize}
for every $u,v,w \in V$ and $q \in \bH$. Moreover, we say that $V$ is \textit{with unity}, or \textit{unital}, if there exists an element $\old$ of $V$, the \textit{unity} of $V$, such that
\begin{itemize}
 \item[$(a4)$] $u=\old u=u\old$
\end{itemize}
for every $u \in V$. Suppose that $V$ is a quaternionic two--sided algebra with unity. If $uv=vu$ for every $u,v \in V$, then $V$ is called \textit{commutative}. A map $^*:V \lra V$ is called \textit{$^*$--involution} if it holds:
\begin{itemize}
 \item[$(*1)$] $(u^*)^*=u$,
 \item[$(*2)$] $(u+v)^*=u^*+v^*$, $\; (qu)^*=u^*\overline{q} \;$ and $\; (uq)^*=\overline{q}u^*$,
 \item[$(*3)$] $(uv)^*=v^*u^*$
\end{itemize}
for every $u,v \in V$ and $q \in \bH$. Equipping $V$ with a $^*$--involution, we obtain a \textit{quaternionic two--sided $^*$--algebra with unity}. It is easy to verify that $\old^*=\old$ and, if $u$ is invertible in $V$, then $(u^{-1})^*=(u^*)^{-1}$. Finally, the quaternionic two--sided $^*$--algebra $V$ with unity is called a \textit{quaternionic two--sided Banach $C^*$--algebra with unity} if, in addition, $V$ is equipped with a function $\|\cdot\|:V \lra \bR^+$ such that
\begin{itemize}
 \item[$(n1)$] $\|u\|=0$ if and only if $u=0$,
 \item[$(n2)$] $\|u+v\| \leq \|u\|+\|v\|$,
 \item[$(n3)$] $\|qu\|=|q| \, \|u\|=\|uq\|$,
 \item[$(n4)$] $\|uv\| \leq \|u\| \, \|v\|$,
 \item[$(n5)$] $\|\old\|=1$,
 \item[$(n*)$] $\|u^*u\|=\|u\|^2$
\end{itemize}
for every $u,v \in V$ and $q \in \bH$, and the metric space obtained equipping $V$ with the distance $V \times V \ni (u,v) \mapsto \|u-v\| \in \bR^+$ is complete. We say that $\|\cdot\|$ is a \textit{quaternionic Banach $C^*$--norm} on $V$.

Suppose that $V$ is a quaternionic two--sided Banach $C^*$--algebras with unity $\old$. As usual, a subset $V'$ of $V$ is a \textit{quaternionic two--sided Banach unital $C^*$--subalgebra} of $V$ if $\old \in V'$ and the restrictions to $V'$ of the operations of $V$ define on $V'$ a structure of quaternionic two--sided Banach $C^*$--algebras with unity $\old$. 

A map $\phi:V \lra W$ between quaternionic two--sided Banach $C^*$--algebras with unity is called \textit{$^*$--homomorphism} if it holds:
\begin{itemize}
 \item[$(h1)$] $\phi(u+v)=\phi(u)+\psi(v), \;$ $\phi(qu)=q\phi(u) \;$ and $\; \phi(uq)=\phi(u)q$,
 \item[$(h2)$] $\phi(u^*)=\phi(u)^*$,
 \item[$(h3)$] $\phi(\old)=\old$
\end{itemize}
for every $u,v \in V$ and $q \in \bH$. 

\begin{remark} \label{rem:complex-C^*}
Considering $\bR$ (resp.\ $\bC_\jmath$ for some fixed $\jmath \in \bS$) instead of $\bH$ as set of scalars, one defines a real (resp.\ $\bC_\jmath$) two--sided Banach unital $C^*$--algebra. Thanks to $(v6)$, the notion of real two--sided Banach unital $C^*$--algebra is equivalent to the usual one of real Banach unital $C^*$--algebra. Similarly, complex Banach unital $C^*$--algebras correspond to two--sided $\bC_\jmath$--Banach unital $C^*$--algebras $V$ having the following property in place of $(v6)$:
\begin{itemize}
 \item[$(v6')$] $cu=uc$ for every $u \in V$ and $c \in \bC_\jmath$.
\end{itemize}

The notion of real Banach unital $C^*$--subalgebra of a real Banach unital $C^*$--algebra can be defined in the usual way. Since a quaternionic two--sided Banach unital $C^*$--algebra $V$ is also a real (two--sided) Banach unital $C^*$--algebra, we can speak about real Banach unital $C^*$--subalgebras of $V$. Evidently, if $(v6')$ holds, then one can speak about $\bC_\jmath$--Banach unital $C^*$--subalgebras of $V$ as well.
\end{remark}

Consider a quaternionic Hilbert space $\sH$ and fix a left scalar multiplication $\bH \ni q \mapsto L_q$ of $\sH$. Given a (right $\bH$--linear) operator $T:D(T) \lra \sH$ and $q \in \bH$, we define the map $qT:D(T) \lra \sH$ by setting
\beq \label{eq:qT}
(qT)u:=q(Tu).
\eeq
It is immediate to verify that $qT$ is again a (right $\bH$--linear) operator. Simi\-larly, if $L_q(D(T)) \subset D(T)$, one can define the (right $\bH$--linear) operator $Tq:D(T) \lra \sH$ as follows:
\beq \label{eq:Tq}
(Tq)(u):=T(qu).
\eeq
Observe that, if $q \in \bR$, then the inclusion $L_q(D(T)) \subset D(T)$ is always verified.

The given notions of left and right scalar multiplications $qT$ and $Tq$ between quaternions and operators are nothing but usual compositions $L_qT$ and $TL_q$, respectively. In general, $qT$ is different from $Tq$. However, if $r \in \bR$, the operators $rT$ and $Tr$ defined here are equal and coincide with the operator $Tr$ defined in (\ref{eq:rT}).

By a direct inspection, one can easily prove that, if $D(T)$ is dense in $\sH$, then $(qT)^*=T^*\overline{q}$ and $(Tq)^*=\overline{q} \, T^*$ if $L_q(D(T))\subset D(T)$.

Suppose now that $T \in \gB(\sH)$. By Proposition \ref{propprod}$(\mr{b})$, it follows immediately that
\beq \label{|qT|}
\|qT\|=|q|\, \|T\|.
\eeq
Moreover, it holds also:
\beq \label{|Tq|}
\|Tq\|=|q| \, \|T\|.
\eeq
Indeed, if $q=0$, (\ref{|Tq|}) is evident. If $q \neq 0$, then we have:
\[
\|Tq\|=\sup_{u\neq 0} \frac{\|T(qu)\|}{\|u\|}=  \sup_{v\neq 0} \frac{\|T(v)\|}{\|q^{-1}v\|} = 
 \frac{1}{|q^{-1}|}\sup_{v\neq 0} \frac{\|T(v)\|}{\|v\|} = |q| \|T\|.
\]
Equalities (\ref{|qT|}) and (\ref{|Tq|}) imply that, if $T \in \gB(\sH)$, then $qT$ and $Tq$ belong to $\gB(\sH)$.

It is now easy to establish the next result.

\begin{theorem} \label{thm:two-sided}
Let $\sH$ be a quaternionic Hilbert space equipped with a left scalar multiplication. Then the set $\gB(\sH)$, equipped with the pointwise sum, with the left and right scalar multiplications defined in (\ref{eq:qT}) and (\ref{eq:Tq}), with the composition as product, with the adjunction $T \mapsto T^*$ as $^*$--involution and with the norm defined in (\ref{QN}), is a quaternionic two--sided Banach $C^*$--algebra with unity $\1$. 
\end{theorem}

\begin{remark} \label{r0}
Independently from the choice of a left scalar multiplication of $\sH$, $\gB(\sH)$ is always a \emph{real Banach $C^*$--algebra with unity $\1$}. It suffices to consider the right scalar multiplication (\ref{eq:rT}), the adjun\-ction $T \mapsto T^*$ as $^*$--involution and the norm defined in (\ref{QN}).
\end{remark}


\subsection{Imaginary units and complex subspaces}
Consider a quaternionic Hilbert space $\sH$ equipped with a left scalar multiplication $\bH \ni q \mapsto L_q$. For short, we write $L_qu=qu$. For every imaginary unit $\imath \in \bS$, the operator $J:= L_\imath$ is anti self--adjoint and unitary; that is, it holds:
\[
J^*=-J \quad \text{and} \quad J^*J=\1.
\]
The proof straightforwardly arises from Proposition~\ref{propprod}. We intend to establish the converse statement: if an operator $J \in \gB(\sH)$ is anti--self adjoint and unitary, then $J=L'_\imath$ for some left scalar multiplication $\bH \ni q \mapsto L'_q$ of $\sH$.

To prove the statement, we need a preliminary definition known from the litera\-ture \cite{Emch}, which will turn out to be useful several times in this work later.

\begin{definition}\label{defHJ}
Let $J \in \gB(\sH)$ be an anti self--adjoint and unitary operator and let $\imath\in \bS$. Recall that $\bC_\imath$ denotes the real subalgebra of $\bH$ generated by $\imath$; that is, $\C_\imath:=\{\alpha+\imath\beta \in \bH \, | \, \alpha,\beta \in \bR\}$. Define the \emph{complex subspaces $\sH^{J\imath}_+$ and $\sH^{J\imath}_-$ of $\sH$ associated with $J$ and $\imath$} by setting
\[
\sH^{J\imath}_\pm:=\{u \in \sH \, | \, Ju=\pm u \imath\}.
\]
\end{definition}

\begin{remark} \label{rC}
$\sH^{J\imath}_\pm$ are closed subsets of $\sH$, because $u \mapsto Ju$ and $u \mapsto \pm u\imath$ are continuous. However, they are not (right $\bH$--linear) subspaces of $\sH$. The names of $\sH^{J\imath}_+$ and $\sH^{J\imath}_-$ are justified by point $(\mr{e})$ of Proposition \ref{propJ} below.
\end{remark}

The aim of this section is to prove the following result.

\begin{proposition}\label{propJ} 
Let $\sH$ be a quaternionic Hilbert space, let $J \in \gB(\sH)$ be an anti self--adjoint and unitary operator, and let $\imath \in \bS$. Then the following facts hold.
\begin{itemize}
 \item[ $(\mr{a})$] There exists a left scalar multiplication $\bH \ni q \mapsto L_q$ of $\sH$ such that $J = L_\imath$.
 \item[$(\mr{b})$] If $\bH \ni q \mapsto L_q$ and $\bH \ni q \mapsto L'_q$ are  left scalar multiplications of $\sH$ such that $L_\imath=L'_\imath$, then $L_{q}=L'_{q}$ for every $q \in \bC_\imath$.
 \item[$(\mr{c})$] Let $\bH \ni q \mapsto L_q$ and $\bH \ni q \mapsto L'_q$ be left scalar  multiplications of $\sH$ induced by Hilbert bases $N$ and $N'$, respectively. Then $L_\imath=L'_\imath$ if and only if $\b z|z' \k \in \bC_\imath $ for every $z \in N$ and $z' \in N'$.
 \item[$(\mr{d})$] $\sH^{J\imath}_+ \neq \{0\}$ and $\sH^{J\imath}_- \neq \{0\}$.
 \item[$(\mr{e})$] Identify $\bC_\imath$ with $\bC$ in the natural way. Then $\sH^{J\imath}_+$ is a complex Hilbert space with respect to the structure induced by $\sH$: its sum is the sum of $\sH$, its complex scalar multiplication is the  right scalar multiplication of $\sH$ restricted to $\bC_\imath$ and its $\bC_{\imath}$--valued Hermitean scalar product coincides with the restriction of the one of~$\sH$. For short, we say that $\sH^{J\imath}_+$ is a \emph{$\bC_\imath$--Hilbert space}. An analogous statement holds for $\sH^{J\imath}_-$.

\item[$(\mr{f})$] If $N$ is a Hilbert basis of the $\bC_\imath$--Hilbert space $\sH^{J\imath}_+$, then $N$ is also a Hilbert basis of $\sH$ and it holds:
\[
J=\sum_{z\in N} z \imath \b z \, | \, \cdot \, \k.
\]
An analogous statement holds for $\sH^{J\imath}_-$.
\end{itemize}
\end{proposition}
 
In order to prove the proposition, we need two lemmata. The first deals with points $(\mr{d})$ and $(\mr{e})$.

\begin{lemma} \label{lemmalemme}
Let $\sH$ be a quaternionic Hilbert space, let $J \in \gB(\sH)$ be an anti self--adjoint and unitary operator and let $\imath \in \bS$. Then $\sH^{J\imath}_+ \neq \{0\}$, the restriction of the Hermitean scalar product $\b \cdot \, | \, \cdot \k$ to $\sH^{J\imath}_+$ is $\bC_\imath$--valued and $\sH^{J\imath}_+$ turns out to be a $\bC_\imath$--Hilbert space with respect to the structure induced by $\sH$.

An analogous statement holds for $\sH^{J\imath}_-$.
\end{lemma}

\begin{proof}
We prove the thesis for $\sH_+^{J\imath}$, the other case being essentially identical. Firstly, we show that there exists $u \neq 0$ such that $u-Ju\imath \neq 0$. Choose a non--null element $u$ of $\sH$. If $u-Ju\imath \neq 0$, we are done. Suppose $u=Ju\imath$. Let $\jmath \in \bS$ such that $\imath\jmath=-\jmath\imath$. We have that $u\jmath=Ju\imath\jmath=-Ju\jmath\imath$. Since $u \neq 0$, we infer that $u\jmath \neq Ju\jmath\imath$. Replacing $u$ with $u\jmath$ if necessary, we may suppose that $u-Ju\imath \neq 0$, as desired. Then $J(u-Ju\imath)=Ju+u\imath=(u-Ju\imath)\imath$; that is, $Ju'=u'\imath$ for $u':=u-Ju\imath \neq 0$. Therefore 
$\sH_+^{J\imath} \neq \{0\}$. Next we notice that, if $v,w \in \sH_+^{J\imath}$ and $r,r' \in \bR$, one has: 
\begin{align*}
\b v|w \k (r+r'\imath)&=\b v|w (r+r'\imath)\k=\b v | wr+wr'\imath\k=\\
&=\b v | w \k r+\b v | Jw \k r' = 
r\langle v | w \rangle  +  r'\langle -Jv | w \rangle \\
&=
r\langle v | w \rangle -  r'\langle v\imath | w \rangle 
=\langle v (r-r'\imath)| w \rangle
= (r+r'\imath) \langle v | w \rangle\:.
\end{align*}
Therefore, $\langle u|v\rangle$ commutes with every element of $\bC_\imath$. This implies that $\langle u|v\rangle\in\bC_\imath$ and that $\langle\cdot | \cdot \rangle$ reduces to a standard Hermitean scalar product on $\sH_+^{J\imath}$, viewed as a complex pre--Hilbert space. Finally, $\sH_+^{J\imath}$ is complete, since $\sH$ is such and $\sH_+^{J\imath}$ is closed.
\end{proof}

\begin{lemma} \label{lemmasomma}
As a $\bC_\imath$--Hilbert space, $\sH$ admits the following direct sum decomposition:
\[
\sH=\sH_+^{J\imath} \oplus \sH_-^{J\imath}.
\]
Moreover, if $N$ is a Hilbert basis of the $\bC_\imath$--Hilbert space $\sH_+^{J\imath}$, then it holds:
\begin{itemize}
\item[$(\mr{a})$] If $\jmath \in \mathbb S$ is an imaginary unit with
$\imath \jmath =- \jmath \imath$ (i.e.\ orthogonal to $\imath$), $N\jmath := \{z\jmath\:|\: z\in N\}$ is a Hilbert basis of the $\bC_\imath$--Hilbert space $\sH_-^{J\imath}$.
\item[$(\mr{b})$] $N$ is a Hilbert basis of $\sH$.
\end{itemize}
\end{lemma}

\begin{proof}
The first statement holds because $\sH_+^{J\imath}\cap \sH_-^{J\imath}= \{0\}$ and, defining $x_\pm := \frac{1}{2}(x \mp Jx\imath)$ for $x \in \sH$, one has $x= x_++x_-$ and $x_\pm \in \sH_\pm^{J\imath}$ as it follows by a direct inspection.

Let us prove $(\mr{a})$. As it is easy to verify, the map $\sH_+^{J\imath} \ni y \mapsto y\jmath \in \sH_-^{J\imath}$ is a well--defined isometric $\bC_\imath$--anti-linear isomorphism and hence the set $N\jmath$ is orthonormal, because $N$ is such. Choose $u \in \sH^{J\imath}_-$. Since $u \jmath \in \sH^{J\imath}_+$, we have that $u\jmath=\sum_{z \in N} z \b z|u \jmath \k$. From this equality, it follows that
\[
u=\sum_{z \in N} z \b z | u \k=\sum_{z \in N} z\jmath \b z\jmath| u \k.
\]
This proves that $N\jmath$ is a Hilbert basis of $\sH_-^{J\imath}$.

It remains to show $(\mr{b})$. By point $(\mr{a})$, one has:
\[x = \sum_{z\in N} z\langle z| x_+ \rangle +  \sum_{z\in N} z\jmath \langle z\jmath| x_- \rangle  = \sum_{z\in N} z \langle z| x_+ +x_- \rangle\:.\]
So $<N>$ is dense in $\sH$ and $\langle z|z'\rangle = \delta_{zz'}$ for every $z,z' \in N$. Thus $N$ is a Hilbert basis of $\sH$. 
\end{proof}

\begin{proof}[{Proof of Proposition \ref{propJ}}]
Points $(\mr{d})$ and $(\mr{e})$ were proved in Lemma \ref{lemmalemme}.

Let us prove $(\mr{f})$. Let $N$ be a Hilbert basis of $H_+^{J\imath}$. By the definition of $H_+^{J\imath}$, it holds that $Jz=z\imath$ for each $z \in N$. Thanks to point $(\mr{b})$ of Lemma \ref{lemmasomma}, $N$ is also a Hilbert basis of $\sH$. In this way, for every $u \in \sH$, it holds:
\[
Ju=J\sum_{z \in N} z \b z|u \k=\sum_{z \in N} Jz \b z|u \k
=\sum_{z \in N} z\imath \b z|u \k.
\]
Point $(\mr{a})$ follows immediately. In fact, if $\bH \ni q \mapsto L_q$ is the left scalar multiplication induced by $N$, then $J=L_\imath$.

Let now $\bH \ni q \mapsto L_q$ and $\bH \ni q \mapsto L'_q$ be scalar left multiplications of $\sH$ such that $L_\imath=L'_\imath$. Since $L_1=\1=L'_1$, for each $\alpha,\beta \in \bR$, it follows:
\[
L_{\alpha+\beta\imath}=L_1\alpha+L_\imath \beta=L'_1\alpha+L'_\imath \beta=L'_{\alpha +\beta\imath}.
\]
This shows $(\mr{b})$. Point $(\mr{c})$ can be proved similarly to $(\mr{g})$ in Proposition \ref{propprod}.
\end{proof}


\subsection{Extension of $\boldsymbol{J}$ to a full left scalar multiplication of $\boldsymbol{\sH}$}

Finally, we present a technical, but quite useful, proposition we shall employ later several times.
This proposition generalizes, and formulates into a modern language, some results presented in Section 3 of \cite{Emch}.

\begin{proposition}\label{propestensione}
For every $\bC_\imath$--linear operator $T: D(T) \lra \sH^{J\imath}_+$ with $D(T) \subset \sH^{J\imath}_+$, there exists a unique right $\bH$--linear operator $\tilde{T}:D(\tilde T) \lra \sH$ of $T$ with $D(\tilde{T}) \subset \sH$ such that $J(D(\tilde T)) \subset D(\tilde T)$, $D(\tilde{T}) \cap \sH^{J\imath}_+ =D(T)$ and $\tilde T(u)=T(u)$ for every $u \in D(T)$. The following additional facts hold.
\begin{itemize}
 \item[$(\mr{a})$] If $T\in \gB(\sH^{J\imath}_+)$, then $\tilde{T} \in \gB(\sH)$ and $\|\tilde T\|=\|T\|$. 
 \item[$(\mr{b})$]  $J\tilde{T}=\tilde{T}J$.
 \item[$(\mr{c})$] Let $V:D(V) \lra \sH$ be a right $\bH$--linear operator with $D(V) \subset \sH$. Then $V$ is equal to $\tilde U$ for some $\bC_\imath$--linear operator $U : D(V)\cap \sH^{J\imath}_+ \lra \sH^{J\imath}_+$ if and only if $J(D(V)) \subset D(V)$ and $JV=VJ$.
 \item[$(\mr{d})$] If $D(T)$ is dense in $\sH^{J\imath}_+$, then $D(\tilde T)$ is dense in $\sH$ and $(\tilde T)^*=\widetilde{T^*}$.
\end{itemize}

Furthermore, given a $\bC_\imath$--linear operator $S:D(S) \lra \sH^{J\imath}_+$ with $D(S) \subset \sH^{J\imath}_+$, we have:
\begin{itemize}
 \item[$(\mr{e})$] $\widetilde{ST}= \tilde{S}\tilde{T}$. 
 \item[$(\mr{f})$] If $S$ is a right $($resp. left$)$ inverse of $T$, then $\tilde S$ is a right $($resp. left$)$ inverse of $\tilde T$. 
\end{itemize}

An analogous statement holds for $\sH^{J\imath}_-$.
\end{proposition}
\begin{proof}
Choose $\jmath \in \bS$ in such a way that $\{1,\imath,\jmath,\imath\jmath\}$ is a basis of $\bH$ and $\imath\jmath=-\jmath\imath$.

We begin showing that, if $\tilde{T}:D(\tilde{T}) \lra \sH$ is a right $\bH$--linear extension of $T$ with $J(D(\tilde{T})) \subset D(\tilde{T})$ and $D(\tilde{T}) \cap \sH^{J\imath}_+=D(T)$, then $\tilde{T}$ is uniquely determined by~$T$.

Denote by $\Phi:\sH \lra \sH$ the $\bC_\imath$--anti--linear isomorphism of $\sH$ sending $x$ in $x\jmath$. It is immediate to verify that $\Phi(\sH^{J\imath}_\pm)=\sH^{J\imath}_\mp$. Since $D(\tilde{T})$ is a right $\bH$--linear subspace of $\sH$, we have that $\Phi(D(\tilde{T}))=D(\tilde{T})$. Applying $\Phi$ to both members of the equality $D(\tilde{T}) \cap \sH^{J\imath}_+=D(T)$, we obtain $D(\tilde{T}) \cap \sH^{J\imath}_-=\Phi(D(T))$. Let $x \in D(\tilde{T})$. Define $a:=\frac{1}{2}(x-Jx\imath)$ and $b:=\frac{1}{2}(x+Jx\imath)$. It is easy to see that $a \in \sH^{J\imath}_+$, $b \in \sH^{J\imath}_-$ and $x=a+b$. Since $J(D(\tilde{T})) \subset D(\tilde{T})$, we infer that $a \in D(\tilde{T}) \cap \sH^{J\imath}_+=D(T)$ and $b \in D(\tilde{T}) \cap \sH^{J\imath}_-=\Phi(D(\tilde{T}))$. This proves that, as a $\bC_\imath$--linear subspace of $\sH$, $D(\tilde{T})$ coincides with $D(T) \oplus \Phi(D(T))$. Bearing in mind that $b\jmath=\Phi(b) \in \sH^{J\imath}_+$, we have that $\tilde{T}(b)=-\tilde{T}(b\jmath\jmath)=-T(b\jmath)\jmath$ and hence $\tilde T(x)=T(a)-T(b\jmath)\jmath$. It follows that $\tilde{T}$ is uniquely determined by $T$, as desired.

Concerning the existence of the extension of $T$ with the required properties, we are now forced to define $D(\tilde{T}):=D(T) \oplus \Phi(D(T))$ and $\tilde{T}:D(\tilde{T}) \lra \sH$ by setting $\tilde{T}(x):=T(a)-T(b\jmath)\jmath$, where $x \in D(\tilde{T})$, $a \in D(T)$, $b \in \Phi(D(T))$ and $x=a+b$. 

We must verify that $J(D(\tilde{T})) \subset D(\tilde{T})$, $D(\tilde{T})$ is a right $\bH$--linear subspace of $\sH$ and $\tilde{T}$ is right $\bH$--linear. Let $x$, $a$ and $b$ be as above. Observe that $J(D(\tilde{T})) \subset D(\tilde{T})$. Indeed, $Jx=Ja+Jb=a\imath-b\imath$ belongs to $D(\tilde{T})$, because $D(T)$ and $\Phi(D(T))$ are $\bC_\imath$--linear subspaces of $\sH$. Let $q \in \bH$. There exist, and are unique, $z_1,z_2 \in \bC_\imath$ such that $q=z_1+z_2\jmath$. It holds:
\[
xq=(a+b)(z_1+z_2\jmath)=(az_1+bz_2\jmath)+(bz_1+az_2\jmath).
\]
Since $z_1\jmath=\jmath \overline{z}_1$ and $z_2\jmath=\jmath \overline{z}_2$, we get that $az_1+bz_2\jmath \in D(T)$ and $bz_1+az_2\jmath \in \Phi(D(T))$ and then $D(\tilde{T})$ is a right $\bH$--linear subspace of $\sH$. Moreover, we have:
\begin{align*}
\tilde{T}(xq) & =T(az_1+bz_2\jmath)-T((bz_1+az_2\jmath)\jmath)\jmath=\\
& =T(a)z_1+T(bz_2\jmath)-T(bz_1\jmath-az_2)\jmath=\\
& =T(a)z_1+T(bz_2\jmath)-T(bz_1\jmath)\jmath+T(a)z_2\jmath.
\end{align*}

On the other hand,  we obtain the right $\bH$--linearity of $\tilde{T}$ as follows:
\begin{align*}
\tilde{T}(x)q & =(T(a)-T(b\jmath)\jmath)(z_1+z_2\jmath)=\\
& =T(a)z_1-T(b\jmath)\jmath z_1+T(a)z_2\jmath-T(b\jmath)\jmath z_2\jmath=\\
& =T(a)z_1-T(b\jmath)\overline{z}_1\jmath+T(a)z_2\jmath-T(b\jmath)\overline{z}_2\jmath\jmath=\\
& =T(a)z_1-T(b\jmath\overline{z}_1)\jmath+T(a)z_2\jmath+T(b\jmath\overline{z}_2)=\\
& =T(a)z_1-T(bz_1\jmath)\jmath+T(a)z_2\jmath+T(bz_2\jmath)=\tilde{T}(xq).
\end{align*}

Let us prove the remaining points.

$(\mr{a})$ Suppose that $T \in \gB(\sH^{J\imath}_+)$. Then $D(\tilde{T})=\sH^{J\imath}_+ \oplus \Phi(\sH^{J\imath}_+)=\sH^{J\imath}_+ \oplus \sH^{J\imath}_-=\sH$. Let $x$, $a$ and $b$ be as above and let $c \in \sH^{J\imath}_+$ such that $b=c\jmath$. Since $\jmath z=\overline{z}\jmath$ for every $z \in \bC_\imath$, we obtain:
\begin{align*}
\b a|b \k+\b b|a \k & =\b a|c\jmath \k+\b c\jmath|a \k=\b a|c \k \jmath-\jmath \b c|a \k=\\
&=\b a|c \k \jmath-\overline{\b c|a \k}\jmath=\b a|c \k \jmath-\b a|c \k \jmath=0
\end{align*}
and hence $\|x\|^2=\|a+b\|^2=\|a\|^2+\|b\|^2$. By the same argument, we have that $\|\tilde{T}(x)\|^2=\|T(a)\|^2+\|T(b\jmath)\jmath\|^2$. In particular, it holds:
\[
\|\tilde{T}(x)\|^2 \leq \|T\|^2(\|a\|^2+\|b\|^2)=\|T\|^2\|x\|^2.
\]
This shows that $\|\tilde{T}\| \leq \|T\|$. On the other hand, we know that $\sH^{J\imath}_+ \neq \{0\}$ and hence $\tilde{T}(y)=T(y)$ for some $y \in \sH^{J\imath}_+ \setminus \{0\}$. It follows that $\|\tilde{T}\|=\|T\|$. 

$(\mr{b})$ Let $x=a+b \in \sH$, where $a \in \sH^{J\imath}_+$ and $b \in \sH^{J\imath}_-$. Observe that $Jx=a\imath-b\imath \in D(\tilde{T})=D(T) \oplus \Phi(D(\tilde{T}))$ if and only if $a\imath \in D(T)$ and $b\imath \in \Phi(D(T))$; that is, $x \in D(\tilde{T})$. These considerations imply that $J(D(\tilde{T}))=D(\tilde{T})$. It follows that $D(J\tilde{T})=D(\tilde{T})=D(\tilde{T}J)$. Suppose now $x=a+b \in D(\tilde{T})$. It holds:
\[
\tilde{T}Jx=\tilde{T}(a\imath-b\imath)=T(a\imath)+T(b\imath\jmath)\jmath=T(a)\imath-T(b\jmath)\imath\jmath
\]
and
\[
J\tilde{T}x=J(T(a)-T(b\jmath)\jmath)=T(a)\imath+T(b\jmath)\jmath\imath=T(a)\imath-T(b\jmath)\imath\jmath=\tilde{T}Jx.
\]

$(\mr{c})$ If $V=\tilde{U}$ for some $\bC_\imath$--linear operator $U : D(V)\cap \sH^{J\imath}_+ \lra \sH^{J\imath}_+$, then we just know that $J(D(V)) \subset D(V)$ and $JV=VJ$. Let us prove the converse implication. Given $x \in D(V) \cap \sH^{J\imath}_+$, it holds: $J(Vx)=V(Jx)=V(x\imath)=(Vx)\imath$. In other words, we have that $V(D(V) \cap \sH^{J\imath}_+) \subset \sH^{J\imath}_+$. Define the $\bC_\imath$--linear operator $U:D(V) \cap \sH^{J\imath}_+ \lra \sH^{J\imath}_+$ by setting $Uu:=Vu$. Since $J(D(V)) \subset D(V)$, $D(U)=D(V) \cap \sH^{J\imath}_+$ and $V$ is a right $\bH$--linear extension of $U$, by the uniqueness of such an extension, we infer that $V=\tilde{U}$.

$(\mr{d})$ Since $\sH=\sH^{J\imath}_+ \oplus \sH^{J\imath}_-$ and $D(\tilde{T})=D(T) \oplus \Phi(D(T))$, if $D(T)$ is dense in $\sH^{J\imath}_+$, then $D(\tilde{T})$ is automatically dense in $\sH$. Taking the adjoint in both members of the equality $\tilde{T}J=J\tilde{T}$, using $J^*=-J$ and bearing in mind point $(\mr{v})$ of Remark \ref{remarkop}, we infer that $J\tilde{T}^*=\tilde{T}^*J$, which includes $J(D(\tilde{T}^*)) \subset D(\tilde{T}^*)$. Here $\tilde{T}^*$ denotes $(\tilde{T})^*$. Thanks to $(\mr{c})$, there exists $U:D(\tilde{T}^*) \cap \sH^{J\imath}_+ \lra \sH^{J\imath}_+$ such that $\tilde{T}^*=\tilde{U}$. It remains to prove that $U=T^*$ or, equivalently, that $\tilde{T}^*x=T^*x$ for every $x \in D(\tilde{T}^*) \cap \sH^{J\imath}_+$. This is quite easy to see. Indeed, if $x \in D(\tilde{T}^*) \cap \sH^{J\imath}_+$, then $\tilde{T}^*x \in \sH^{J\imath}_+$, because $\tilde{T}^*$ commutes with $J$. Moreover, by definition of adjoint, it holds: $\b \tilde{T}^*x|v \k=\b x|\tilde{T}v \k$ for every $v \in D(\tilde{T})$. In particular, it follows that $\b \tilde{T}^*x|v \k=\b x|Tv \k$ for every $v \in D(T)$, which is equivalent to say that $\tilde{T}^*x=T^*x$, as desired. 

$(\mr{e})$ Let $x \in D(\widetilde{ST})$. As we have just proved, $D(\widetilde{ST})=D(ST) \oplus \Phi(D(ST))$ so there exist, and are unique, $a \in D(ST)$ and $b \in \Phi(D(ST))$ such that $x=a+b$. It follows that $\{T(a),T(b\jmath)\} \subset D(S)$ and hence $\tilde{T}(x)=T(a)-T(b\jmath)\jmath \in D(S) \oplus \Phi(D(S))=D(\tilde{S})$. In other words, $x$ belongs to $D(\widetilde{ST})$. This implies that $D(\widetilde{ST}) \subset D(\tilde{S}\tilde{T})$. Proceeding similarly, one can prove the converse inclusion. The equality $D(\widetilde{ST})=D(\tilde{S}\tilde{T})$, just proved, ensures that $J(D(\tilde{S}\tilde{T}))=J(D(\widetilde{ST})) \subset D(\widetilde{ST})=D(\tilde{S}\tilde{T})$ and $D(\tilde{S}\tilde{T}) \cap \sH^{J\imath}_+=D(\widetilde{ST}) \cap \sH^{J\imath}_+=D(ST)$. In this way,  being $\tilde{S}\tilde{T}u=STu$ for every $u \in D(ST)$, thanks to the uniqueness of the extension, we infer that $\tilde{S}\tilde{T}=\widetilde{ST}$.

$(\mr{f})$ This point follows immediately from $(\mr{e})$.
\end{proof}


\section{Resolvent and spectrum}
It is not so obvious how to extend the definitions of spectrum and resolvent in quaternionic Hilbert spaces. Let us focus on the simpler situation, where we are looking for eigenvalues of a bounded right $\bH$--linear operator $T$ on a quaternionic Hilbert space $\sH$. Without fixing any left scalar multiplication of $\sH$, the equation determining the eigenvalues reads as follows:
\[
Tu=uq.
\] 
Here a drawback arises: if $q \in \bH \setminus \bR$ is fixed, the map $u \mapsto uq$ is not right $\bH$--linear, differently from $u \mapsto Tu$. Consequently, the eigenspace of $q$ cannot be a right $\bH$--linear subspace. Indeed, if $\lambda \neq 0$, $u\lambda$ is an eigenvector of $\lambda^{-1} q \lambda$ instead of $q$ itself. As a tentative way out, one could decide to deal with quaternionic Hilbert spaces equipped with a left scalar multiplication interpreting the eigenvalues equation in terms of such a left scalar multiplication:
\[
Tu=qu.
\]
Now both sides are right $\bH$--linear. However, this approach is not suitable for physical applications \cite{Adler}, where self--adjoint operators should have real spectrum. As an elementary example, consider the finite dimensional quaternionic Hilbert space constructed over $\sH=\bH \oplus \bH$ equipped with the standard scalar product:
\[
\b (r,s)|(u,v) \k:=\overline{r}u+\overline{s}v \quad \text{if $u,v,r,s \in \bH$.}
\]
The right $\bH$--linear operator represented by the matrix
\[
T=
\begin{bmatrix}
0 & i \\
-i & 0
\end{bmatrix},
\]
where the multiplication is that on the left, is self--adjoint. (See \cite{Loring} and references therin for  several results in the  finite-dimensional case).
However, for every quaternion of the form $q=a+cj+dk$ with $a,c,d \in \bR$ and $a^2+c^2+d^2=1$, there is an element $w_q \in \sH\setminus\{0\}$ with $Tw_q=qw_q$. Although $T=T^*$, it admits non--real eigenvalues. Consequently, in the rest of the paper, we come back to the former approach keeping the constraint, found above, concerning the fact that each eigenvalue $q$ brings a whole conjugation class of the quaternions, the \emph{eigensphere}: \[\cS_q=\{\lambda^{-1} q \lambda \in \bH \,|\, \lambda \in \bH \setminus \{0\}\}.\]
As a matter of fact, we adopt the viewpoint introduced in \cite{libroverde}, that is invariant under conjugation in the sense just pointed out.


\subsection{Spherical resolvent and spectrum and their elementary properties}
We are in a position to present  one of the most important notion treated in this work, that of spectrum of a quaternionic operator. 
First of all, given a (right $\bH$--linear) operator $T:D(T) \lra \sH$ and $q \in \bH$, we define the associated operator $\Delta_q(T):D(T^2) \lra \sH$ by setting:
\[
\Delta_q(T):=T^2-T(q+\overline{q})+\1 |q|^2.
\]
The fundamental suggestion by \cite{libroverde} is to systematically replace $T-\1 q$ for $\Delta_q(T)$ in the definition of resolvent set, resolvent operator and spectrum. The following definitions are a re-adaptation of the corresponding ones stated in \cite{libroverde}. A more accurate comparison  appears immediately after the definition.

\begin{definition} \label{def_spectrum}
Let $\sH$ be a quaternionic Hilbert space and let $T:D(T) \lra  \sH$ be an operator. The \emph{spherical resolvent set of $T$} is the set $\srho(T) \subset \bH$ of the quaternions $q$ such that the three following conditions hold true:
\begin{itemize}
 \item[$(\mr{a})$] $\mi{Ker}(\Delta_q(T))= \{0\}$.  
 \item[$(\mr{b})$] $\mi{Ran}(\Delta_q(T))$ is dense in $\sH$. 
 \item[$(\mr{c})$] $\Delta_q(T)^{-1}:\mi{Ran}(\Delta_q(T)) \lra D(T^2)$ is bounded.
\end{itemize}

The \emph{spherical spectrum $\ssp(T)$ of $T$} is defined by setting $\ssp(T):=\bH \setminus \rho_S(T)$. It decomposes into three disjoint subsets as follows:
\begin{itemize}
 \item[$(\mr{i})$] the \emph{spherical point spectrum of $T$}:
  \[
  \sigma_{\mi{pS}}(T):=\{q \in \bH \,|\, \mi{Ker}(\Delta_q(T)) \neq \{0\} \};
  \]
 \item[$(\mr{ii})$] the \emph{spherical residual spectrum of $T$}:
  \[
  \sigma_{\mi{rS}}(T):=\left\{q \in \bH \left| \, \mi{Ker}(\Delta_q(T))=\{0\}, \, \overline{\mi{Ran}(\Delta_q(T))} \neq \sH \right.\right\};
 \] 
\item[(iii)] the \emph{spherical continuous spectrum of $T$}:
 \[
 \sigma_{\mi{cS}}(T):=\left\{q \in \bH \left| \, \mi{Ker}(\Delta_q(T))= \{0\},\, \overline{\mi{Ran}(\Delta_q(T))}=\sH, \, \Delta_q(T)^{-1} \not \in \gB(\sH) \right.\right\}.
 \]
\end{itemize}

The \emph{spherical spectral radius of $T$} is defined as the following element $r_S(T)$ of $\bR^+ \cup \{+\infty\}$:
\[
r_S(T):=\sup\big\{|q| \in \bR^+ \big| \, q \in  \ssp(T)\big\}.
\]

If $Tu=uq$ for some $q \in \bH$ and $u \in \sH \setminus \{0\}$, then $u$ is called \emph{eigenvector of $T$ with eigenvalue $q$}. 
\end{definition}

\begin{remark} \label{spectrum}
$(1)$ In \cite{libroverde}, a different, but equivalent, definition of spherical resolvent set $\rho_S(T)$ is adopted for $T \in \gB(\sH)$. Therein no decomposition onto the various parts of the spectrum has been introduced.
In \cite{libroverde}, the following definition (Def. 4.8.1) is stated (straightforwardly adapting it to the quaternionic Hilbert space case, since \cite{libroverde} studies the case of a quaternionic two--sided Banach modules):
\beq
\rho_S(T)=\{q \in \bH \,|\, \Delta_q(T):\sH \lra \sH \mbox{ is bijective and }\Delta_q(T)^{-1} \in \gB(\sH)\}. \label{defrislibroverde}
\eeq
Actually, that definition is completely equivalent to our Definition \ref{def_spectrum} so that 
(\ref{defrislibroverde}) is an identity assuming our definition of $\rho_S(T)$.
Indeed, if $\Delta_q(T)$ is bijective and its inverse belongs to $\gB(\sH)$, then conditions $(\mr{a})$, $(\mr{b})$ and $(\mr{c})$ are evidently verified. Conversely, suppose that conditions $(\mr{a})$, $(\mr{b})$ and $(\mr{c})$ hold true. Firstly, observe that $(\mr{c})$ implies that $\mi{Ran}(\Delta_q(T))$ is closed in $\sH$. Let $\mi{Ran}(\Delta_q(T)) \supset \{y_n\}_{n \in \bN} \to y \in \sH$ and, for each $n \in \bN$, let $x_n:=\Delta_q(T)^{-1}y_n \in \sH=D(\Delta_q(T))$. The sequence $\{x_n\}_{n \in \bN}$ of $\sH$ is a Cauchy sequence. Indeed, it holds:
\[
\|x_n-x_m\| \leq \|\Delta_q(T)^{-1}\| \, \|y_n-y_m\|.
\]
It follows that $\{x_n\}_{n \in \bN} \to x$ for some $x \in \sH$. Since $\Delta_q(T)$ is continuous, we infer that $\Delta_q(T)x = y$ and hence $\mi{Ran}(\Delta_q(T))$ is closed in $\sH$. Thanks to $(\mr{a})$ and $(\mr{b})$, $\Delta_q(T)$ turns out to be bijective with $\Delta_q(T)^{-1} \in \gB(\sH)$.

In the case of an unbounded operator defined on a subspace of $\sH$, we stick to Definition \ref{def_spectrum} since, on the one hand, the domains and boundedness of operators are irrelevant in that definition. On the other hand, our definition gives rise to statements that are 
straightforwardly extensions of the corresponding propositions in the  theory in complex Hilbert spaces (see, in particular, (b) and (c) in Theorem \ref{teospectra} below). This is true in spite of the fact that definitions are substantially different, as they rely upon the second order polynomial
$\Delta_q(T)$, rather than the first--order polynomial $T-\1 q$.
 
It is worth noticing that, referring to unbounded operators, \cite{libroverde} proposes a formally different definition (see Definition 4.15.2 in \cite{libroverde}), whose equivalence with ours, through a straightforward  re--adaptation to the case of a quaternionic Hilbert space, would deserve investigation. It will be done elsewhere.

$(2)$ For every $q \in \bH$, define $E_q(T):=\{u \in \sH \,|\, Tu=uq\}$. It is worth stressing that $E_q(T)$ is a right $\bH$--linear subspace of $\sH$ if and only if either $q \in \bR$ or $q \not\in \bR$ and $E_q(T)=\{0\}$. Indeed, if $q \in \bH \setminus \bR$, $u \in E_q(T) \setminus \{0\}$ and $p$ is a quaternion with $qp \neq pq$, then $up \not\in E_q(T)$. 

$(3)$ If $p$ and $q$ are conjugated quaternions, then  $\Delta_p(T)=\Delta_q(T)$. Indeed, $p+\overline{p}=q+\overline{q}$ and $|p|=|q|$ (see point $(4)$ of Remark \ref{remin}).  We infer that $\srho(T)$ and $\ssp(T)$ are circular. The same is true for $\sigma_{\mi{pS}}$, $\sigma_{\mi{rS}}$ and $\sigma_{\mi{cS}}$.

Recall that a subset $A$ of $\bH$ is called \emph{circular} if it is equal to $\OO_\K$ for some subset $\K$ of $\bC$. This is equivalent to say that $A$ satisfies one of the following two equivalent conditions:
\begin{itemize}
 \item If $\alpha+\imath\beta \in A$ for some $\alpha,\beta \in \bR$ and $\imath \in \bS$, then $\alpha+\jmath\beta \in A$ for every $\jmath \in \bS$.
 \item If $q \in A$ and $p \in \bH$ is conjugated to $q$, then $p \in A$. Equivalently, if $q \in A$, then $\bS_q \subset A$.
\end{itemize}

$(4)$ The spherical spectrum, or at least its intersection with a complex plane $\bC_\jmath$, has already appeared in the literature, in the most general setting of real $^*$--algebras. In \cite{Kulkarni}, it is referred to as Kaplansky's definition \cite{Kaplansky} of the spectrum.
\end{remark}

More generally, the spherical resolvent and the spherical spectrum can be defined for bounded right $\bH$--linear operators on quaternionic two--sided Banach modules in a form similar to that introduced above (see \cite{libroverde}). Several properties of resolvents and of spectra of bounded opertors on complex Banach or Hilbert  spaces remain valid in that general context. Nevertheless, their proofs are by no means trivial (see \cite{libroverde} again). Here we recall some of these properties in the quaternionic Hilbert setting.

\begin{theorem} \label{teopropspectrum}
Let $\sH$ be a quaternionic Hilbert space and let $T \in \gB(\sH)$ be an operator. The following assertions hold.
\begin{itemize}
 \item[$(\mr{a})$] Let $q \in \bH$ with $|q|>\|T\|$ and, for each $n \in \bN$, let $a_n$ be the real number defined by setting $a_n:=|q|^{-2n-2}\sum_{h=0}^nq^h\overline{q}^{n-h}$. Then the series $\sum_{n \in \bN}T^na_n$ converges absolutely in $\gB(\sH)$ (with respect to the operator norm $\| \cdot \|$) to $\Delta_q(T)^{-1}$. In particular, it holds:
\[
r_S(T) \leq \|T\|.
\]
 \item[$(\mr{b})$] $\ssp(T)$ is a non--empty compact subset of $\bH$.
 \item[$(\mr{c})$] Let $P \in \bR[X]$ and let $P(T)$ be the corresponding operator in $\gB(\sH)$ defined in (\ref{eq:P(T)}). Then, if $T$ is self--adjoint, the following \emph{spectral map property} holds:
\[
\ssp(P(T))=P(\ssp(T)).
\]
 \item[$(\mr{d})$] For every $n \in \bN$, we have that $\ssp(T^n) =(\ssp(T))^n:=\{q^n \in \bH \,|\, q \in \ssp(T)\}$.
 \item[$(\mr{e})$] \emph{Gelfand's spectral radius formula} holds:
\beq \label{GF}
r_S(T)=\lim_{n \to +\infty} \|T^n\|^{1/n}.
\eeq
In particular, if $T$ is normal, then:
\beq \label{raggiospett} 
r_S(T)=\|T\|.
\eeq
\end{itemize}
\end{theorem}
\begin{proof}
This result is a consequence of Theorems 4.7.4, 4.7.5, 4.8.11, 4.12.5, 4.12.6 and 4.13.2 of \cite{libroverde}. Actually, these theorems of \cite{libroverde} deals with the more general case of quaternionic two--sided Banach modules. In this way, in order to be able to apply it in our context, it suffices to fix a left scalar multiplication $\bH \ni q \mapsto L_q$ of $\sH$ and to equip $\gB(\sH)$ with its natural quaternionic two--sided Banach module structure described in Section \ref{subsec:leftprod}. For the convenience of the reader, we give a proof of the theorem. The following proof of point $(\mr{e})$ is a detailed version of the corresponding one of \cite{libroverde}. With regards to $(\mr{d})$, we simply give the exact reference in \cite{libroverde}. The proofs of the other points are different, and more direct.

$(\mr{a})$ Let $q \in \bH$ with $|q|>\|T\|$. By direct computations, it is immediate to verify that $a_0=|q|^{-2}$, $-(q+\overline{q})a_0+|q|^2a_1=0$ and $a_{n-2}-(q+\overline{q})a_{n-1}+|q|^2a_n=0$ for every $n \geq 2$. Moreover, we have that $|a_n| \leq (n+1)|q|^{-n-2}$ for every $n \in \bN$ and hence $\limsup_{n \to +\infty}\|T^na_n\|^{1/n} \leq \|T\| \, |q|^{-1}<1$. It follows that the series $S:=\sum_{n \in \bN}T^na_n$ converges absolutely in $\gB(\sH)$. Furthermore, it holds:
\begin{align*}
\Delta_q(T)S =& S \Delta_q(T)=\1|q^2|a_0+T(-(q+\overline{q})a_0+|q|^2a_1)+\\
&+\sum_{n \geq 2}T^n(a_{n-2}-(q+\overline{q})a_{n-1}+|q|^2a_n)=\1.
\end{align*}
This proves that $S$ is the inverse of $\Delta_q(T)$ in $\gB(\sH)$ and hence $r_S(T) \leq \|T\|$.

$(\mr{b})$ Firstly, we show that $\ssp(T)$ is compact. Bearing in mind the  just proved inequality $r_S(T) \leq \|T\|$, it suffices to see that $\ssp(T)$ is closed in $\bH$ or, equiva\-lently, that $\srho(T)$ is open in $\bH$. By point $(\mr{c})$ of Proposition \ref{propcompleteness}, the subset $\mscr{I}$ of $\gB(\sH)$ consisting of all isomorphisms is open. Since the map $\Theta:\bH \lra \gB(\sH)$, sending $q$ into $\Delta_q(T)$, is continuous, it follows that $\srho(T)=\Theta^{-1}(\mscr{I})$ is open in $\bH$.

Now fix a left scalar multiplication of~$\sH$. Let $\{1,i,j,k\}$ be the standard orthonormal basis of $\sH$. Identify $\bC$ with $\bC_i$ in the natural way. Define the map $\psi:\bC \cap \srho(T) \lra \gB(\sH)$ by setting, for  $z=\alpha+i\beta$ with $\alpha,\beta \in \bR$,
\[
\psi(z):=\Delta_z(T)^{-1}(T-\1\overline{z})=\Delta_{\alpha+i\beta}(T)^{-1}(T-\1(\alpha-i\beta)).
\] 

Choose $z \in \bC \cap \srho(T)$ with $|z|>\|T\|$ and define the sequence $\{a_n\}_{n \in \bN} \subset \bR$ as in the statement of point $(\mr{b})$ with $q=z$. It is easy to verify that $a_n\overline{z}-a_{n-1}=z^{-n-1}$ if $n \geq 1$. In this way, we infer that
\begin{align} \label{eq:psi}
\psi(z)&=\textstyle \left(\sum_{n \in \bN}T^na_n\right)(T-\1\overline{z})= \nonumber \\
&\textstyle =-\1 z^{-1}-\sum_{n \geq 1}T^n(a_n\overline{z}-a_{n-1})=-\sum_{n \in \bN}T^nz^{-n-1},
\end{align}
where the series $\sum_{n \in \bN}T^nz^{-n-1}$ converges absolutely in $\gB(\sH)$. Let $r\ge r_S(T)$  and let $B_r$ be the open ball of $\bC$ centered at $0$ with radius $r$. Since the map $\psi$ is real analytic with respect to $\alpha$ and $\beta$, and holomorphic for $|z|>\|T\|$, it is holomorphic on an open neighborhood $U$ of $\bC \setminus B_r$. This means that $\partial \psi/\partial \alpha$ and $\partial \psi/\partial \beta$ exist, are continuous and satisfy:
\beq \label{eq:partial}
\dd{\psi}{\alpha}+\dd{\psi}{\beta} \, i=0 \quad \mbox{in }\gB(\sH) \mbox{ if }z \in U.
\eeq

Let $F:\gB(\sH) \lra \bH$ be a continuous right $\bH$--linear functional on $\gB(\sH)$; that is, an element of $\gB(\sH)'$, and let $f_0,f_1,f_2,f_3:\bC \cap \srho(T) \lra \bR$ and $\ell,m:\bC \cap \srho(T) \lra \bC$ be the functions of class $\mscr{C}^1$ such that
\[
F \circ \psi=f_0+if_1+jf_2+kf_3=\ell+km
\]
and hence $\ell=f_0+if_1$ and $m=f_3+if_2$. Thanks to (\ref{eq:partial}), we infer that
\begin{align*}
0 &=F\left(\dd{\psi}{\alpha}+\dd{\psi}{\beta} \, i\right)=\dd{(F \circ \psi)}{\alpha}+\dd{(F \circ \psi)}{\beta} \, i=\\
&=\left(\dd{f_0}{\alpha}-\dd{f_1}{\beta}\right)+i\left(\dd{f_1}{\alpha}+\dd{f_0}{\beta}\right)+j\left(\dd{f_2}{\alpha}+\dd{f_3}{\beta}\right)+k\left(\dd{f_3}{\alpha}-\dd{f_2}{\beta}\right),
\end{align*}
which is equivalent to say that $\ell$ and $m$ are holomorphic.

Let us complete the proof of $(\mr{b})$ by showing that $\ssp(T) \neq \emptyset$. By (\ref{eq:psi}), we have that
\beq \label{eq:|psi|}
(|\ell(z)|^2+|m(z)|^2)^{1/2}=|(F \circ \psi)(z)| \leq K\|F\| \, |z|^{-1} \quad \mbox{if }|z| \geq 1+\|T\|,
\eeq
where $\|F\|$ is the operator norm of $F$ and $K:=\sum_{n \in \bN}\|T\|^n(1+\|T\|)^{-n} \in \bR^+$.

 Suppose that $\ssp(T)=\emptyset$. Thanks to (\ref{eq:|psi|}), $\ell$ and $m$ turn out to be bounded entire holomorphic functions with $\lim_{|z| \to +\infty}|\ell(z)|=0=\lim_{|z| \to +\infty}|m(z)|$. Liouville's theorem ensures that $\ell$ and $m$ are null, and hence $F \circ \psi$ is null for every $F \in \gB(H)'$. Lemma  \ref{lem:Hahn-Banach} implies that $\psi(z)=0$ for every $z \in \bC$. It follows that $T=\1 \overline{z}$ for every $z \in \bC$, which is impossible.  

$(\mr{c})$ If $P$ is constant, then $(\mr{c})$ is trivial. Suppose that $P$ has positive degree.

Let us prove that $P(\ssp(T)) \subset \ssp(P(T))$. Let $q \in \ssp(T)$. Since $T$ is assumed to be self--adjoint, as we have just said, Theorem~\ref{teospectra} below, which is completely independent from the proposition we are proving, implies that $q \in \R$. It follows that $P(q) \in \R$ and $\Delta_{P(q)}(P(T))=(P(T)-\1 P(q))^2$. The polynomial $P(X)-P(q)$ in $\R[X]$ vanishes at $X=q$ and hence there exists $\theta \in \R[X]$ such that $P(X)-P(q)=(X-q)\theta(X)$. In particular, we infer that
\begin{equation} \label{eq:sub}
\Delta_{P(q)}(P(T))=(T-\1 q)^2(\theta(T))^2=\Delta_q(T)(\theta(T))^2=(\theta(T))^2\Delta_q(T).
\end{equation}
It follows that $P(q) \in \ssp(P(T))$; that is, $\Delta_{P(q)}(P(T))$ is not inverti\-ble in $\gB(\sH)$. Otherwise, thanks to (\ref{eq:sub}), $\Delta_q(T)$ would be invertible in $\gB(\sH)$ as well, contradicting the fact that $q \in \ssp(T)$.

It remains to show that $\ssp(P(T)) \subset P(\ssp(T))$. Let $q \in \ssp(P(T))$. Since $P$ has real coefficients, the operator $P(T)$ is self--adjoint. Therefore, $q \in \R$ and $\Delta_q(P(T))=(P(T)-\1 q)^2$. Fix $\imath \in \mathbb{S}$. By the Fundamental Theorem of Algebra, there exist $\alpha_0 \in \R \setminus \{0\}$, $\alpha_1,\ldots,\alpha_h \in \R$ and $\alpha_{h+1},\ldots,\alpha_k \in \C_\imath \setminus \R$ for some $h,k \in \N$ with $h \leq k$ such that
\[
P(X)-q=\alpha_0\prod_{\ell=1}^h(X-\alpha_{\ell})\prod_{\ell=h+1}^k\Delta_{\alpha_{\ell}}(X).
\]
In particular, it holds:
\[
\Delta_q(P(T))=(P(T)-\1 q)^2=\alpha_0^2\prod_{\ell=1}^h\Delta_{\alpha_{\ell}}(T)\prod_{\ell=h+1}^k(\Delta_{\alpha_{\ell}}(T))^2.
\]
If $\Delta_{\alpha_{\ell}}(T)$ were invertible in $\gB(\sH)$ for every $\ell \in \{1,\ldots,k\}$, then $\Delta_q(P(T))$ would be invertible in $\gB(\sH)$ as well, contradicting the fact that $q \in \ssp(P(T))$. It follows that there exists $\ell \in \{1,\ldots,k\}$ (or better $\ell \in \{1,\ldots,h\}$) such that $\alpha_{\ell} \in \ssp(T)$. Since $P(\alpha_{\ell})-q=0$, $q=P(\alpha_{\ell})$ belongs to $P(\sigma_S(T))$, as desired.

$(\mr{d})$ This point is proved in Theorem 4.12.5 of \cite{libroverde}.

$(\mr{e})$ Let us follows the proof of Theorem 4.12.6 of \cite{libroverde}. By combining $(\mr{d})$ with the last part of $(\mr{a})$, we infer that $r_S(T^n)=r_S(T)^n \leq \|T^n\|$ for every $n \in \bN$ and hence $r_S(T) \leq \liminf_{n \to +\infty}\|T^n\|^{1/n}$. In order to prove (\ref{GF}), it suffices to show that
\beq \label{eq:limsup}
\limsup_{n \to +\infty}\|T^n\|^{1/n} \leq r_S(T).
\eeq
Let us prove this inequality. Consider again the maps $\psi:\bC \cap \srho(T) \lra \gB(\sH)$, $F \in \gB(\sH)'$ and $\ell,m:\bC \cap \srho(T) \lra \bC$ introduced in the proof of point $(\mr{b})$. Let $r>r_S(T)$, let $B_r$ be the open ball of $\bC$ centered at $0$ with radius $r$ and let $\bar{B}_r$ be its closure in $\bC$. Since $\ell$ and $m$ are holomorphic on an open neighborhood of $\bC \setminus B_r$, it holds:
\begin{align*}
(F \circ \psi)(z)&=\ell(z)+km(z)=-\frac{1}{2\pi i}\int_{\partial B_r}\frac{\ell(\xi)}{\xi-z} \, d\xi-k\frac{1}{2\pi i}\int_{\partial B_r}\frac{m(\xi)}{\xi-z} \, d\xi=\\
&=-\frac{1}{2\pi}\int_{\partial B_r}(\ell(\xi)+km(\xi))(\xi-z)^{-1}i^{-1} \, d\xi=\\
&=-\frac{1}{2\pi}\int_{\partial B_r}(F \circ \psi)(\xi)(\xi-z)^{-1}i^{-1} \, d\xi=\\
&=F \left(-\frac{1}{2\pi}\int_{\partial B_r}\psi(\xi)(\xi-z)^{-1}i^{-1} \, d\xi\right) \quad \mbox{if }z \in \bC \setminus \bar{B}_r,
\end{align*}
where the integral $\int_{\partial B_r}\psi(\xi)(\xi-z)^{-1}i^{-1} \, d\xi$ can be defined as an operator norm limit of Riemann sums. Lemma \ref{lem:Hahn-Banach} implies the following Cauchy formula for $\psi$:
\[
\psi(z)=-\frac{1}{2\pi}\int_{\partial B_r}\psi(\xi)(\xi-z)^{-1}i^{-1} \, d\xi \quad \mbox{if }z \in \bC \setminus \bar{B}_r.
\]
Fix $z \in \bC \setminus \bar{B}_r$. Since $(\xi-z)^{-1}=z^{-1}(1-\frac{\xi}{z})^{-1}=-\sum_{n \in \bN}\xi^nz^{-n-1}$ if $\xi \in \partial B_r$ and the series $-\sum_{n \in \bN}\xi^nz^{-n-1}$ converges absolutely for $\xi \in \partial B_r$, $\psi$ can be expanded into Laurent series on $\bC \setminus \bar{B}_r$ as follows:
\begin{align*}
\psi(z) &=\sum_{n \in \bN}\frac{1}{2\pi}\int_{\partial B_r}\psi(\xi)\xi^nz^{-n-1}i^{-1} \, d\xi=\sum_{n \in \bN}\left(\frac{1}{2\pi}\int_{\partial B_r}\psi(\xi)\xi^ni^{-1} \, d\xi\right)z^{-n-1}
\end{align*}
Comparing the latter equality with (\ref{eq:psi}) and bearing in mind the uniqueness of the Laurent series expansion, we infer at once that $T^n=\frac{1}{2\pi}\int_{\partial B_r}\psi(\xi)\xi^ni^{-1} \, d\xi$ for every $n \in \bN$ and the series $\sum_{n \in \bN}T^nz^{-n-1}$ converges absolutely for every $z \in \bC$ with $|z|>r_S(T)$. In particular, $\lim_{n \to +\infty}\|T^n\| \, |z|^{-n-1}=0$ and hence the sequence $\{\|T^n\| \, |z|^{-n}\}_{n \in \bN} \subset \bR$ is bounded. It follows that $\limsup_{n \to +\infty}\|T^n\|^{1/n} \leq |z|$ for every $z \in \bC$ with $|z|>r_S(T)$. This proves (\ref{eq:limsup}) and hence (\ref{GF}).

If $T$ is normal, then it holds:
\[
\|T^2\|^2 = \|(T^2)^*  T^2\| = \|(T^*)^2  T^2\| 
= \|(T^*T)^*  (T^*T)\| 
= \|T^*T\|^2 = (\|T\|^2)^2
\]
and hence $\|T^2\|=\|T\|^2$. Proceeding by induction, one easily shows that $\|T^{2^k}\|^{1/2^k}=\|T\|$ for every $k \in \bN$. Now (\ref{raggiospett}) follows from (\ref{GF}).
\end{proof}

\begin{corollary} \label{cor:finite}
Let $\sH$ be a quaternionic Hilbert space, let $T \in \gB(\sH)$ be a self--adjoint operator and let $P \in \bR[X]$ such that $\ssp(T)$ is finite and $P$ vanishes on $\ssp(T)$. Then $P(T)$ is the null operator in $\gB(\sH)$.
\end{corollary}
\begin{proof}
Since $T$ is self--adjoint, it is immediate to verify that $P(T)$ is self--adjoint as well. By point $(\mr{c})$ of Theorem \ref{teopropspectrum}, we infer that $\ssp(P(T))=\{0\}$ and hence (\ref{raggiospett}) implies that $P(T)=0$, as desired.
\end{proof}


\subsection{Some spectral properties of operators on quaternionic Hilbert spaces}\label{secspecprop}
 Regardless different definitions with respect to the complex Hilbert space case, the notions of spherical spectrum and resolvent set enjoy some properties which are quite similar to those for complex Hilbert spaces. We go to illustrate them together with some other features that, conversely, are proper to the quaternionic Hilbert space case.

First of all, quite remarkably,  it turns out that $\sigma_{\mi{pS}}(T)$ coincides with the set of eigenvalues of $T$.

\begin{proposition}\label{propsigmap}
Let $\sH$ be a quaternionic Hilbert space and let $T:D(T) \lra \sH$ be an operator. Then $\sigma_{pS}(T)$ coincides with the set of all eigenvalues of $T$.
\end{proposition}

\begin{proof} Let $u$ be an eigenvector of $T$ with eigenvalue $q \in \bH$. By the definition of $\Delta_q(T)$, we have that
\[
\Delta_q(T)u=T(Tu-uq)-(Tu-uq)\overline{q}=0
\]
and hence $q \in \sigma_{pS}(T)$, because $u \neq 0$. Conversely, if $q \in \sigma_{\mi{pS}}(T)$, then there is $v \in D(T^2)$ such that
\[
0=\Delta_q(T)v=T(Tv-vq)-(Tv-vq)\overline{q}.
\]
If $v':=Tv-vq=0$, then $q$ is an eigenvalue of $T$. Otherwise $v'$ is an eigenvector of $T$ with eigenvalue $\overline{q}$. Since $\overline{q}=sqs^{-1}$ for some $s \neq 0$, $T(v's)=(v's)q$ and $v's \neq 0$, we infer that $q$ is an eigenvalue of $T$ as well.
\end{proof}

\begin{remark}
In our context, the subspace $\mi{Ker}(\Delta_q(T))$ has the role of an ``eigenspace''. In particular, it holds: $\mi{Ker}(\Delta_q(T)) \neq \{0\}$ if and only if $\cS_q$ is an eigensphere of $T$.
\end{remark}

There is an important difference from the standard complex Hilbert space: if $T \in \gB(\sH)$, then $T$ and $T^*$ have always the same spherical spectrum and, as we shall prove later, $T$ and $T^*$ are even unitarily equivalent, whenever $T$ is normal.

\begin{proposition}\label{propspecTT}
Let $\sH$ be a quaternionic Hilbert space and let $T \in \gB(\sH)$ be an operator. Then $\srho(T)= \srho(T^*)$ and $\ssp(T)=\ssp(T^*)$.
\end{proposition}
\begin{proof} Since $\Delta_q(T)^*=\Delta_q(T^*)$ for every $q \in \bH$, point $(\mr{viii})$ of Remark \ref{remarkop} and point $(1)$ of Remark \ref{spectrum} immediately imply that $\srho(T)= \srho(T^*)$ and hence $\ssp(T)=\ssp(T^*)$.
 \end{proof}

We are now in a position to establish an important result concerning the spectrum of normal, self--adjoint, anti self--adjoint and unitary operators. That statement, at first glance, sounds quite weird, since it declares that the point spectrum of an anti self--adjoint and unitary operator $T$ on $\sH$ completely fills a sphere. In this way, $\sigma_{pS}(T)$ has to be uncountable in all cases,
even if the quaternionic Hilbert space $\sH$ is separable. However, differently from the complex Hilbert space case,  there is no contradiction now. Indeed, in quaternionic Hilbert spaces, eigenvectors of two different eigenvalues are not mutually orthogonal in general, unless both the eigenvalues are real.
 
\begin{theorem}\label{teospectra} 
Let $\sH$ be a quaternionic Hilbert space and let 
$T: D(T) \lra \sH$ be an operator with dense domain. The following assertions hold.
\begin{itemize}
\item[$(\mr{a})$] If $T\in \gB(\sH)$ is normal, then we have that
\begin{itemize}
\item[(i)] $\sigma_{\mi{pS}}(T)=\sigma_{\mi{pS}}(T^*)$,
\item[(ii)] $\sigma_{\mi{rS}}(T)=\sigma_{\mi{rS}}(T^*)=\emptyset$,
\item[(iii)] $\sigma_{\mi{cS}}(T)=\sigma_{\mi{cS}}(T^*)$.
\end{itemize}
\item[$(\mr{b})$] If  $T$ is self--adjoint (not necessarily in $\gB(\sH)$), then $\ssp(T) \subset \bR$ and $\sigma_{\mi{rS}}(T)$ is empty. 
\item[$(\mr{c})$] If $T$ is anti self--adjoint (not necessarily in $\gB(\sH)$), then $\ssp(T) \subset \mr{Im}(\bH)$ and $\sigma_{\mi{rS}}(T)$ is empty. 
\item[$(\mr{d})$] If $T \in \gB(\sH)$ is unitary, then $\ssp(T) \subset \{q \in \bH \,|\, |q|=1\}$.
 \item[$(\mr{e})$] If $T\in \gB(\sH)$ is anti self--adjoint and unitary, then $\ssp(T)=\sigma_{pS}(T) = \bS$.
\end{itemize}
\end{theorem}
\begin{proof} $(\mr{a})$ Since $\Delta_q(T)^*=\Delta_q(T^*)$ and $T$ is assumed to be normal, it follows immediately that $\Delta_q(T)$ is normal as well. Therefore, thanks to Proposition \ref{lemma3}, we know that $\mi{Ker}(\Delta_q(T))= \mi{Ker}(\Delta_q(T^*))$. This equality implies at once that 
$\sigma_{\mi{pS}}(T)=\sigma_{\mi{pS}}(T^*)$. Let us prove that $\sigma_{\mi{rS}}(T)=\emptyset$. Suppose on the contrary that $\sigma_{\mi{rS}}(T)$ contains a quaternion $q$. It would follows that
\[
\{0\}=\mi{Ker}(\Delta_q(T))=\mi{Ker}(\Delta_q(T^*))=\overline{\mi{Ran}(\Delta_q(T))}^\perp \neq \{0\},
\]
which is a contradiction. Similarly, one can prove that $\sigma_{\mi{rS}}(T^*)=\emptyset$. Since $\ssp(T)=\ssp(T^*)$ by Proposition \ref{propspecTT}, $\sigma_{\mi{pS}}(T)=\sigma_{\mi{pS}}(T^*)$, $\sigma_{\mi{rS}}(T)=\sigma_{\mi{rS}}(T^*)$ and the three components of the spherical spectrum are pairwise disjoint, we infer that $\sigma_{\mi{cS}}(T)=\sigma_{\mi{cS}}(T^*)$ as well.

$(\mr{b})$ Consider $q=r+\nu \in \bH$ with $r \in \bR$ and $\nu \in \mr{Im}(\bH) \setminus \{0\}$. We intend to show that $q \in \srho(T)$, which is equivalent to $\ssp(T) \subset \bR$. One has:
\[
\Delta_q(T)=T^2-T2r+\1 r^2+\1 |\nu|^2= 
(T-\1 r)^2+\1 |\nu|^2.
\]
Since $T$ is self--adjoint, $T-\1 r$ is self--adjoint as well. In particular, $T-\1 r$ is closed with dense domain. Define the operator $S_r:D(T^2) \lra \sH$ by setting $S_r:=(T-\1 r)^2$. Applying point $(\mr{d})$ of Theorem \ref{teoopagg}, we infer that $D(S_r)=D(T^2)=D(\Delta_q(T))$ is dense in $\sH$ and $S_r$ is self--adjoint. Consequently, $\Delta_q(T)=(T-\1 r)^2+\1 |\nu|^2$ is self--adjoint as well. Since $T-\1 r$ and $S_r$ are self--adjoint, if $u \in D(T^2)$, then we have that 
\[
\b u| S_r u \k=\b (T-\1 r)u | (T-\1 r)u \k  \geq 0,
\]
\[
\|\Delta_q(T)u\|^2=\b S_ru+u|\nu|^2 \, | \,  S_ru+u|\nu|^2\k=\|S_ru\|^2+2|\nu|^2 \b u \, | \, S_r u \k+|\nu|^4 \|u\|^2
\]
and hence we can write that
\beq \label{diseq}
\|\Delta_q(T)u\| \geq  |\nu|^2 \|u\|.
\eeq
The latter inequality implies at once that $\mi{Ker}(\Delta_q(T))=\{0\}$ and 
$\Delta_q(T)^{-1}:\mi{Ran}(\Delta_q(T)) \lra D(T^2)$ is bounded. Bearing in mind Proposition \ref{lemma0} and the fact that $\Delta_q(T)$ is self--adjoint, we observe:
\begin{align*}
\overline{\mi{Ran}(\Delta_q(T))}&=(\mi{Ran}(\Delta_q(T))^\perp)^\perp=
\mi{Ker}(\Delta_q(T)^*)^\perp=\\
&=\mi{Ker}(\Delta_q(T))^\perp = \{0\}^\perp= \sH.
\end{align*}
This proves that $q \in \srho(T)$ and hence that $\ssp(T) \subset \bR$, as desired. 

The proof of the fact that $\sigma_{\mi{rS}}(T)=\emptyset$ is similar to the one given above. If $p \in \sigma_{\mi{rS}}(T)$, then we obtain the following contradiction:
\[
\{0\}=\mi{Ker}(\Delta_p(T))=\mi{Ker}(\Delta_p(T)^*)=\overline{\mi{Ran}(\Delta_p(T))}^\perp \neq \{0\}.
\]

$(\mr{c})$ Let $\lambda=r+\nu \in \bH$ with $r \in \bR \setminus \{0\}$ and $\nu \in \mr{Im}(\bH)$. We must prove that $\lambda \in \srho(T)$. First, we show that
\beq \label{stimaantiself}
\|\Delta_{\lambda}(T) u\| \geq r^2 \|u\| \quad \text{if }u \in D(T^2),
\eeq
which implies that $\mi{Ker}(\Delta_\lambda(T))= \{0\}$ and $\Delta_\lambda(T)^{-1}:\mi{Ran}(\Delta_\lambda(T)) \lra D(\Delta_\lambda(T))$ is bounded. Let $u \in D(T^2)$. Since $T$ is anti self--adjoint, we have that $\b T^2u|Tu\k+\b Tu|T^2u\k=0=\b Tu|u\k+\b u|Tu\k$ and $\b T^2u|u\k+\b u|T^2u\k=-2\|Tu\|^2$. In particular, it holds:
\[
\| \Delta_{\lambda}(T) u\|^2 =\|T^2u\|^2+(r^2 +|\nu|^2)^2\|u\|^2+2(r^2-|\nu|^2)\|Tu\|^2.
\]
If $r^2-|\nu|^2 \geq 0$, then (\ref{stimaantiself}) is evident. Suppose that $r^2-|\nu|^2<0$. Since $\|Tu\|^2=-\b u|T^2 u\k \leq \|u\|\, \|T^2u\|$, we have that $2(r^2-|\nu|^2)\|Tu\|^2 \geq 2(r^2-|\nu|^2)\|u\|\,\|T^2u\|$. It follows that
\begin{align*}
\|\Delta_{\lambda}(T) u\|^2 &\geq \|T^2u\|^2 +(r^2 +|\nu|^2)^2\|u\|^2+2(r^2-|\nu|^2)\|u\|\,\|T^2u\|=\\
&=(\|T^2u\|-|\nu|^2\|u\|)^2+2r^2\|u\|\,\|T^2u\|+ (r^4+2r^2|\nu|^2)\|u\|^2 \geq \\
&\geq r^4\|u\|^2.
\end{align*}
Inequality (\ref{stimaantiself}) is proved. Now we show that $D(T^2)$ is dense in $\sH$ and $\Delta_\lambda(T)^*=\Delta_{-r+\nu}(T)$. Applying point $(\mr{d})$ of Theorem \ref{teoopagg} to $T$, we obtain that $D(T^2)$ is dense in $\sH$ and $T^2$ is self--adjoint. Evidently, the operator $T^2+\1 r^2:D(T^2) \lra \sH$ is self--adjoint as well. In this way, bearing in mind point $(\mr{iv})$ of Remark \ref{remarkop}, we infer that
\begin{align*}
T^2+\1 r^2&=(T-\1 r)^2+T2r=((T-\1 r)^2+T2r)^* \supset((T-\1 r)^2)^*+T^*2r=\\
&=((T-\1 r)^2)^*-T2r
\end{align*}
and hence $(T+\1 r)^2=(T-\1 r)^2+T4r \supset ((T-\1 r)^2)^*$. By point $(\mr{v})$ of Remark \ref{remarkop}, we have also that
$((T-\1 r)^2)^* \supset (T^*-\1 r)^2=(T+\1 r)^2$. It follows that $((T-\1 r)^2)^*=(T+\1 r)^2$ and hence
\[
\Delta_\lambda(T)^*=((T-\1 r)^2+\1 |\nu|^2)^*=(T+\1 r)^2+\1 |\nu|^2=\Delta_{-r+\nu}(T),
\]
as desired. Thanks to (\ref{stimaantiself}), we know that $\|\Delta_{-r+\nu}(T)u\| \geq r^2\|u\|$ for every $u \in \sH$. In particular, we infer that $\mi{Ker}(\Delta_{-r+\nu}(T))=\{0\}$. In this way, Proposition \ref{lemma0} implies that 
\begin{align*}
\overline{\mi{Ran}(\Delta_\lambda(T))} &= (\mi{Ran}(\Delta_\lambda(T))^\perp)^\perp=
\mi{Ker}(\Delta_\lambda(T)^*)^\perp=\\
&=\mi{Ker}(\Delta_{-r+\nu}(T))^\perp=\{0\}^\perp= \sH.
\end{align*}
We have just established that $\lambda \in \srho(T)$ and hence that $\ssp(T) \subset \mr{Im}(\bH)$. The equality $\sigma_{\mi{rS}}(T)=\emptyset$ can be proved as in the proof of $(\mr{a})$.

$(\mr{d})$ Since $\Delta_0(T)=T^2$, it is obvious that $\Delta_0(T)$ is bijective and its inverse coincides with $(T^*)^2 \in \gB(\sH)$. It follows that $0 \in \srho(T)$. If $|q|>1=\|T\|$, then point $(\mr{a})$ of Theorem~\ref{teopropspectrum} ensures that $q \in \srho(T)$. Let $0<|q|<1$. We have:
\begin{align*}
\Delta_q(T)&=T^2-T(q+\overline{q})+\1 |q|^2=\\ &=((T^*)^2-T^*(\overline{q}^{-1}+ q^{-1})+\1 |q^{-1}|^{2})T^2|q|^2=\\
&=\Delta_{q^{-1}}(T^*)T^2|q|^2.
\end{align*}
The operator $T^*$ is unitary and $|q^{-1}|= |q|^{-1}>1$. In this situation, we know that $\Delta_{q^{-1}}(T^*)$ is bijective and has bounded inverse. It follows that $\Delta_q(T)$ is bijective and has bounded inverse as well. This proves that $q \in \srho(T)$. We have just established that, if $q \in \srho(T)$, then $|q| \neq 1$, completing the proof of $(\mr{d})$.

$(\mr{e})$ The fact that $\sigma_S(J)\subset \bS$ is an  immediate consequence of $(\mr{c})$, $(\mr{d})$. Moreover, due to (d) in Proposition \ref{propJ} and
Proposition \ref{propsigmap}, every $\imath \in \bS$ belongs to $\sigma_{pS}(T)$. Since $\sigma_{pS}(T) \subset \sigma_S(T)$ the thesis holds.
\end{proof}


\section{Real measurable functional calculus for self--adjoint operators and slice nature of normal operators} \label{sec:J}

Let $\sH$ be a quaternionic Hilbert space and let $T \in \gB(\sH)$ be a normal operator. This part focuses on the problem concerning the existence of self--adjoint operators $A,B \in \gB(\sH)$ and of an anti self--adjoint and unitary operator $J \in \gB(\sH)$ such that $T$ decomposes as
\[
T = A+JB
\]
and $J$ commutes with $T$ and $T^*$. We shall prove not only that operators $A$, $B$ and $J$ exist (Theorem \ref{teoext}), but even that there exists a left scalar multiplication $\bH \ni q \mapsto L_q$ of $\sH$ such that $L_\imath = J$ for some $\imath \in \bS$ and $L_q$ commutes with $A$ and $B$ for every $q \in \bH$ (Theorem \ref{newtheorem}). As a by--product, we will give a proof of the known fact \cite{visw} that normal operators are unitarily similar to their adjoint operators. Furthermore, in Proposition \ref{propinterssigma} and Corollary \ref{corollaruCplus}, we will also establish the relation between the spherical spectrum of $T$ and the standard spectrum of the restriction of $T$ to the complex Hilbert subspaces $\sH_\pm^{J\imath}$ of $\sH$ defined in Definition \ref{defHJ}. The way we shall follow to prove the mentioned results relies upon some tools of measurable functional calculus for self--adjoint operators 
on quaternionic Hilbert spaces, which are interesting on their own right.

\subsection{The operator $\boldsymbol{J_0}$} \label{subsec:J_0}

It happens that, for a fixed normal operator $T \in \gB(\sH)$, there is an anti self--adjoint operator $J_0$, isometric on $\overline{\mi{Ran}(T-T^*)}=\mi{Ker}(T-T^*)^\perp$ and vanishing on $\mi{Ker}(T-T^*)$, uniquely determined by $T$, that commutes with $T$ and $T^*$, and induces an apparently familiar decomposition of $T$ into a complex combination of self--adjoint operators:
\beq \label{decA0}
T=(T+T^*)\frac{1}{2}+J_0 |T-T^*|\frac{1}{2}.
\eeq
In the special case in which $\mi{Ker}(T-T^*)=\{0\}$, if we fix an imaginary unit $\imath$ of $\bH$ and a left scalar multiplication $\bH \ni q \mapsto L_q$ of $\sH$ with $L_\imath=J_0$ (see Proposition~\ref{propJ}), then we can also write:
\beq \label{decA}
T=\frac{1}{2}(T+T^*) + J_0 \frac{1}{2\imath} (T-T^*).
\eeq
Notice that, if $J_0$ did not commute with $T-T^*$, then the operator $\frac{1}{2\imath}(T-T^*)$ could not be self--adjoint. So commutativity plays a crucial r\^ole here.

\begin{theorem} \label{teobastardo2} 
Let $\sH$ be a quaternionic Hilbert space and let $T \in \gB(\sH)$ be an operator. Then there exists, and is unique, an operator $J_0 \in \gB(\sH)$ such that:
\begin{itemize}
 \item[$(\mr{i})$] $J_0$ is anti self--adjoint,
 \item[$(\mr{ii})$] decomposition (\ref{decA0}) holds true,
 \item[$(\mr{iii})$] $\mi{Ker}(T-T^*)=\mi{Ker}(J_0)$,
 \item[$(\mr{iv})$] $J_0J_0^*=\1$ on $\mi{Ker}(T-T^*)^\perp$,
 \item[$(\mr{v})$] $J_0$ commutes with $|T-T^*|$.
\end{itemize}

Moreover, $J_0(\mi{Ker}(T-T^*)^\perp) \subset \mi{Ker}(T-T^*)^\perp$ and, if $T$ is normal, then $J_0$ commutes also with all the operators in $\gB(\sH)$ commuting with both $T$ and $T^*$.
\end{theorem}
\begin{proof}
By applying Theorem~\ref{polar-polar} (see also Remark \ref{rem}) to the anti self--adjoint operator $T-T^*$, we obtain an anti self--adjoint operator $W \in \gB(\sH)$ such that $T-T^*=W|T-T^*|$, $\mi{Ker}(T-T^*)=\mi{Ker}(|T-T^*|) \subset \mi{Ker}(W)$ 
 and $WW^*=\1$ on $\mi{Ker}(T-T^*)^\perp$. Moreover, $W$ commutes with $|T-T^*|$ and with all the operators in $\gB(\sH)$ commuting with $T-T^*$ (and with $(T-T^*)^*=-(T-T^*)$). In particular, if $T$ is normal, then $W$ commutes also with $T$ and $T^*$. Evidently, $J_0:=W$ has the desired properties.

Let us show that such an operator $J_0$ is unique. Let $J_1 \in \gB(\sH)$ be an operator satisfying conditions $(\mr{i})$--$(\mr{v})$ (with $J_0$ replaced by $J_1$). By combining $(\mr{i})$, $(\mr{ii})$ and $(\mr{v})$ with the fact that $|T-T^*|$ is self--adjoint, we infer that
\[
T^*=(T+T^*)\frac{1}{2}-J_1|T-T^*|\frac{1}{2}.
\]
It follows that $T-T^*=J_1|T-T^*|$. Define $P:=|T-T^*|$. Bearing in mind $(\mr{iii})$, $(\mr{iv})$, the positivity of $P$ and the uniqueness of the quaternionic polar decomposition of $T-T^*$ (see Theorem~\ref{polar-polar}), we infer at once that $J_0=J_1$.
\end{proof}

As an immediate consequence, we obtain: 

\begin{corollary}
Let $\sH$ be a quaternionic Hilbert space, let $T \in \gB(\sH)$ be a normal operator, let $J_0 \in \gB(\sH)$ be the operator with the properties 
listed in the statement of Theorem \ref{teobastardo2} and let $\imath \in \bS$. Suppose that $\mi{Ker}(T-T^*)=\{0\}$. Then $J_0$ is unitary. Moreover, if $\bH \ni q \mapsto L_q$ is a left scalar multiplication of $\sH$ such that $L_\imath=J_0$, then both operators $\frac{1}{2}(T+T^*)$ and $\frac{1}{2\imath}(T-T^*)$ are self--adjoint and decomposition (\ref{decA}) holds.
\end{corollary}


\subsection{Continuous and measurable real functions of a self-adjoint operator}

\textit{Throughout this section, given a quaternionic Hilbert space $\sH$, we consider $\gB(\sH)$ as a real Banach algebra with unity $\1$}; that is, here $\gB(\sH)$ denotes the set of (bounded right $\bH$--linear) operators on $\sH$, equipped with the pointwise sum, with the real scalar multiplication defined in $(\ref{eq:rT})$, with the composition as product and with the norm defined in $(\ref{QN})$.

Let $\K$ be a non--empty subset of $\bR$. We denote by $\CC(\K,\bR)$ the commutative real Banach unital algebra of continuous real--valued functions defined on~$\K$. As usual, the algebra operations are the natural ones defined pointwisely, the unity $1_\K$ is the function constantly equal to $1$ and the norm is the supremum one~$\|\cdot\|_\infty$.

It is worth recalling the notion of real polynomial function.

\begin{definition} \label{def:polynomial-functions}
A function $p:\K \lra \bR$ is said to be a \emph{real polynomial function} if there exists a polynomial $P \in \bR[X]$ such that $p(t)=P(t)$ for each $t \in \K$.
\end{definition}

\begin{remark} \label{rem:poly-funct}
Since the zero set of a non--null real polynomial is finite and every finite subset is the zero set of a non--null polynomial, we infer that a real polynomial function on $\K$ is induced by a unique polynomial in $\bR[X]$ if and only if $\K$ is infinite. In other words, the homomorphism from $\bR[X]$ to $\CC(\K,\bR)$, sending $P$ into the real polynomial function induced by $P$ itself, is injective if and only if $\K$ is infinite.
\end{remark}

Before stating the next result, we remind the reader that the spherical spectrum of a self--adjoint bounded operator is a non--empty compact subset of $\bR$ (see Theorems \ref{teopropspectrum}$(\mr{b})$ and \ref{teospectra}$(\mr{b})$). 

\begin{theorem}\label{teoA1}
Let $\sH$ be a quaternionic Hilbert space and let $T \in \gB(\sH)$ be a self--adjoint operator. Then there exists, and is unique, a continuous homomorphism 
\[
\Phi_T:\CC(\ssp(T),\bR) \ni f \mapsto f(T) \in  \gB(\sH)
\]
of real Banach unital algebras such that:
\begin{itemize}
 \item[$(\mr{i})$] $\Phi_T$ is unity--preserving; that is, $\Phi_T(1_{\ssp(T)})=\1$.
 \item[$(\mr{ii})$] $\Phi_T(\mi{id})=T$, where $\mi{id}:\ssp(T) \hookrightarrow \bR$ denotes the inclusion map.
\end{itemize}
The following further facts hold true.
\begin{itemize}
\item[$(\mr{a})$] The operator $f(T)$ is self--adjoint for every $f \in \CC(\ssp(T),\bR)$.
\item[$(\mr{b})$] $\Phi_T$ is isometric; that is, $\|f(T)\|=\|f\|_\infty$ for every $f \in \CC(\ssp(T),\bR)$.
\item[$(\mr{c})$] $\Phi_T$ is positive;
that is, $f(T) \geq 0$ if $f \in \CC(\ssp(T),\bR)$ and $f(t) \geq 0$ for every $t \in \ssp(T)$.
\item[$(\mr{d})$] For every $f \in \CC(\ssp(T),\bR)$, $f(T)$ commutes with every element of $\gB(\sH)$ that commutes with $T$.
\end{itemize}
\end{theorem}
\begin{proof}
The present proof is organized into two steps. 

\textit{Step I}. Suppose that $\ssp(T)$ is infinite. Let us prove the uniqueness of $\Phi_T$. First, observe that, thanks to $(\mr{i})$ and $(\mr{ii})$, if $p:\ssp(T) \lra \bR$ is a real polynomial function induced by the (unique) polynomial $P \in \bR[X]$, then $\Phi_T(p)=P(T)$, where $P(T)$ is the operator in $\gB(\sH)$ defined in (\ref{eq:P(T)}). Let $\Psi: \CC(\ssp(T),\bR) \lra \gB(\sH)$ be another unity--preserving continuous homomorphism with $\Psi(\mi{id})=T$. It follows that the map $\Psi-\Phi_T$ is continuous and vanishes on every real polynomial function defined on $\ssp(T)$. By the Weierstrass approximation theorem, the set of all real polynomial functions on $\ssp(T)$ is dense in $\CC(\ssp(T),\bR)$ and hence $\Psi=\Phi_T$.

Let us construct $\Phi_T$. For every real polynomial function $p:\ssp(T) \lra \bR$, define $\Phi_T(p):=P(T)$, where $P$ is the unique polynomial in $\bR[X]$ inducing $p$. Since $T$ is self--adjoint, the operator $P(T)$ is also self--adjoint and hence points $(\mr{c})$ and $(\mr{e})$ of Theorem \ref{teopropspectrum} imply that
\[
\|\Phi_T(p)\|=\sup\{|P(r)| \in \bR^+ \,|\, r \in \ssp(T)\} =\|p\|_\infty.
\]
By continuity, $\Phi_T$ extends uniquely to an isometric homomorphism on the whole $\mscr{C}(\ssp(T),\bR)$, which satisfies $(\mr{a})$ and $(\mr{b})$. Evidently, by construction, $\Phi_T$ satisfies also $(\mr{i})$ and $(\mr{ii})$. Let us prove $(\mr{c})$. Given a function $f \in \mscr{C}(\ssp(T),\bR)$ with $f(t) \geq 0$ for every $t \in \ssp(T)$, we can apply $\Phi_T$ to $\sqrt{f} \in \mscr{C}(\ssp(T),\bR)$, obtaining the operator $\sqrt{f}(T) \in \gB(\sH)$. Since $\sqrt{f}(T)\sqrt{f}(T)=f(T)$ and $\sqrt{f}(T)$ is self--adjoint, it holds:
\[
\b u|f(T)u \k=\b u|\sqrt{f}(T)\sqrt{f}(T)u \k  
=\b \sqrt{f}(T) u|\sqrt{f}(T)u \k \in \bR^+ \quad \mbox{if }u \in \sH,
\]
as desired. Point $(\mr{d})$ is evident if $f$ is a real polynomial function. By continuity and by the Stone--Weierstrass approximation theorem, $(\mr{d})$ turns out to be true for every $f \in \mscr{C}(\ssp(T),\bR)$.

\textit{Step II.} Suppose now that $\ssp(T)$ is finite. The uniqueness of $\Phi_T$ is easy to see. Indeed, if $f:\ssp(T) \lra \bR$ is a function (which is always continuous in this case), then there exists $P \in \bR[X]$ such that $f(t)=P(t)$ for each $t \in \ssp(T)$. In this way, bearing in mind $(\mr{i})$ and $(\mr{ii})$, we have that $\Phi_T(f)$ must be equal to $P(T)$. Now, Corollary \ref{cor:finite} ensures that such a definition of $\Phi_T$ is consistent, because it does not depend on the choice of $P$, but only on $f$.

Points $(\mr{a})$ and $(\mr{d})$ are evident. Points $(\mr{b})$ and $(\mr{c})$ can be proved as in Step I.
\end{proof}

Next step consists in extending the notion of function $f(T)$ of a self--adjoint operator $T$ to the case in which $f:\ssp(T) \lra \bR$ is Borel--measurable and bounded.

In the following, we consider the non--empty subset $\K$ of $\bR$ to be equipped with the relative euclidean topology and we denote by $\mscr{B}(\K)$ the Borel $\sigma$--algebra of $\K$. Moreover, we indicate by $M(\K,\bR)$ the commutative real Banach unital algebra of bounded Borel--measurable functions from $\K$ to $\bR$. As before, the algebra operations are the natural ones defined pointwisely, the unity is the function $1_\K$ and the norm is the supremum one $\|\cdot\|_\infty$, and \textit{not} the essential--supremum one, since no measure has been adopted on $\K$.

\begin{theorem}\label{teoA2}
Let $\sH$ be a quaternionic Hilbert space and let $T \in \gB(\sH)$ be a self--adjoint operator. Then there exists, and is unique, a continuous homomorphism 
\[
\widehat{\Phi}_T:M(\ssp(T),\bR) \ni f \mapsto f(T) \in \gB(\sH)
\]
of commutative real Banach unital algebras such that:
\begin{itemize}
 \item[$(\mr{i})$] $\widehat{\Phi}_T$ is unity--preserving and $\widehat{\Phi}_T(\mi{id})=T$.
 \item[$(\mr{ii})$] $f(T)$ is self--adjoint for every $f \in M(\ssp(T),\bR)$.
 \item[$(\mr{iii})$] If a sequence $\{f_n\}_{n \in \bN}$ in $M(\ssp(T),\bR)$ is bounded and converges pointwisely to some $f \in  M(\ssp(T),\bR)$, then
$\{\widehat{\Phi}_T(f_n)\}_{n \in \bN} \to \widehat{\Phi}_T(f)$ in the weak opera\-tor topology; that is, $\{F(\widehat{\Phi}_T(f_n)(u))\}_{n \in \bN} \to F(\widehat{\Phi}_T(f)(u))$ in $\bH$ for every $u \in \sH$ and $F \in \sH'$.
\end{itemize}
The following further facts hold.
\begin{itemize}
 \item[$(\mr{a})$] $\widehat{\Phi}_T$ extends $\Phi_T$; that is, $\widehat{\Phi}_T=\Phi_T$ on $\mscr{C}(\ssp(T),\bR)$.
 \item[$(\mr{b})$] $\widehat{\Phi}_T$ is norm--decreasing; that is,
$\|f(T)\| \leq \|f\|_\infty$ for every $f \in M(\ssp(T),\bR)$.
 \item[$(\mr{c})$] $\widehat{\Phi}_T$ is positive;
that is, $f(T) \geq 0$ if $f \in M(\ssp(T),\bR)$ and $f(t) \geq 0$ for every $t \in \ssp(T)$.
 \item[$(\mr{d})$] For every $u \in \sH$, define the function $\mu_u:\mscr{B}(\ssp(T)) \lra \R^+$ by setting
\[
\mu_u(E):=\b u|\widehat{\Phi}_T(\chi_E)u\k,
\]
where $\chi_E$ is the characteristic function of the Borel subset $E$ of $\ssp(T)$. Then each $\mu_u$ is a finite positive $\sigma$--additive regular Borel measure on $\ssp(T)$ and, for every $f \in M(\ssp(T),\bR)$, it holds:
\beq \label{eq:f^2}
\|f(T)u\|^2=\int_{\ssp(T)}f^2 \, d\mu_u.
\eeq
 \item[$(\mr{e})$] For every $f \in M(\ssp(T),\bR)$, $f(T)$ commutes with every element of $\gB(\sH)$ that commutes with $T$.
 \item[$(\mr{f})$] If a sequence $\{f_n\}_{n \in \bN}$ in $M(\ssp(T),\bR)$ is bounded and converges pointwisely to some $f \in M(\ssp(T),\bR)$, then
$\{\widehat{\Phi}_T(f_n)\}_{n \in \bN} \to \widehat{\Phi}_T(f)$ in the strong opera\-tor topology; that is, $\{\widehat{\Phi}_T(f_n)(u)\}_{n \in \bN} \to \widehat{\Phi}_T(f)(u)$ in $\sH$ for every $u \in \sH$.
\end{itemize}
\end{theorem}
\begin{proof}
We begin with the \textit{existence issue} proving {\em en passant} the validity of points $(\mr{a})$--$(\mr{f})$.
We subdivide this part of the proof into three steps.

\textit{Step I.}
Fix $u \in \sH$. Bearing in mind point $(\mr{c})$ of Theorem \ref{teoA1}, we can define the map
\[
\mscr{C}(\ssp(T),\bC) \ni f \mapsto \b u|\mr{Re}(f)(T) u \k+i \b u|\mr{Im}(f)(T)u \k \in \bC,
\]
where $\mscr{C}(\ssp(T),\bC)$ denotes the usual complex Banach space of $\bC$--valued continuous functions on $\ssp(T)$. That map defines a complex linear functional on $\mscr{C}(\ssp(T),\bC)$, which is positive by $(\mr{c})$ of Theorem \ref{teoA1}. Riesz' theorem for Borel measures on Hausdorff locally compact spaces (see Theo\-rem~2.14 of \cite{RudinARC}) implies that there is a unique positive $\sigma$--additive regular Borel measure $\mu_u:\ssp(T) \lra \bR^+$ such that
\[
\b u|\mr{Re}(f)(T) u \k+i \b u|\mr{Im}(f)(T)u \k=\int_{\ssp(T)} f \, d\mu_u \quad \text{if $f \in \mscr{C}(\ssp(T),\bC)$}
\]
and hence
\beq \label{defM0}
\b u|f(T) u \k=\int_{\ssp(T)} f \, d\mu_u \quad \text{if $f \in \mscr{C}(\ssp(T),\bR)$.}
\eeq
As a consequence, we obtain:
\beq \label{M0}
\mu_u(\ssp(T))=\int_{\ssp(T)} 1_{\ssp(T)} \, d\mu_u= \b u| \1 u\k=\|u\|^2.
\eeq

\textit{Step II.} Fix $u,v \in \sH$. Let us prove the following polarization--type identity:
\begin{align} \label{eq:polar}
4 \langle u| f(T)v \rangle=&\langle u+v | f(T) (u+v)\rangle -\langle u-v | f(T) (u-v)\rangle+ \nonumber\\
&+ \left( \langle ui+v | f(T) (ui+v)\rangle -\langle ui-v | f(T) (ui-v)\rangle \right)i+\nonumber\\
&+ \left( \langle uj+v |f(T) (uj+v)\rangle -\langle uj-v | f(T) (uj-v)\rangle \right)j+\\
&+ \left( \langle uk+v |f(T) (uk+v)\rangle  -\langle uk-v | f(T) (uk-v)\rangle \right)k. \nonumber
\end{align}
Let $\alpha,\beta,\gamma,\delta \in \bR$ such that $\langle u| f(T) v \rangle=\alpha+\beta i+\gamma j+\delta k$. By point $(\mr{a})$ of Theorem \ref{teoA1}, we know that $f(T)$ is self--adjoint and hence it holds:
\begin{align*}
\b u+v | f(T) (u+v)\k-\b u-v | f(T) (u-v)\k &=2\b u|f(T)v \k+2\b v|f(T)u \k=\\
&=2\b u|f(T)v \k+2\b f(T)v|u \k=\\
&=2\b u|f(T)v \k+2\overline{\b u|f(T)v \k}=\\
&=4\mr{Re}(\b u|f(T)v \k)=4\alpha.
\end{align*}
In particular, we have:
\begin{align*}
&\b ui+v | f(T)(ui+v)\k-\b ui-v | f(T)(ui-v)\k =4\mr{Re}(-i\b u|f(T)v \k)=4\beta,\\
&\b uj+v | f(T)(uj+v)\k-\b uj-v | f(T)(uj-v)\k =4\mr{Re}(-j\b u|f(T)v \k)=4\gamma,\\
&\b uk+v | f(T)(uk+k)\k-\b uk-v | f(T) (uk-v)\k =4\mr{Re}(-k\b u|f(T)v \k)=4\delta.
\end{align*}
Equality (\ref{eq:polar}) follows immediately.
For every $a \in \{0,1,2,3\}$, define the signed measure $\nu^{(a)}_{u,v}:\mscr{B}(\ssp(T)) \lra \bR$ by setting
\[
\nu^{(a)}_{u,v}(E):=\frac{1}{4}\mu_{ue_a+v}(E)-   \frac{1}{4}\mu_{ue_a-v}(E)\quad \text{if $E \in \mscr{B}(\ssp(T))$.}
\]
For convenience, define $e_0:=1$, $e_1:=i$, $e_2:=j$ and $e_3:=k$. Thanks to (\ref{defM0}) and (\ref{eq:polar}), we infer at once that
\beq \label{defM}
\langle u | f(T) v \rangle = \sum_{a=0}^3 \int_{\ssp(T)} f \, d\nu^{(a)}_{u,v} \, e_a \quad \text{if $f \in \CC(\ssp(T),\bR)$.}
\eeq
If we think of (finite) signed measures as sub--cases of complex measures, then we can exploit the known result (valid in general for complex regular Borel measures in view of the complex measure Riesz' representation theorem \cite[Theorem 6.19]{RudinARC}) that two such measures coincide when the corresponding integrals of real continuous compactly supported functions produce the same results. In that way, by (\ref{defM}), one easily obtain that
\begin{align}
&\nu^{(a)}_{u,v+v'}=\nu^{(a)}_{u,v}+\nu^{(a)}_{u,v'} \quad \text{if $v' \in \sH$ and $a \in \{0,1,2,3\}$},\label{M2} \\
&\sum_{a=0}^3 \nu^{(a)}_{u,vq}e_a=\sum_{a=0}^3 \nu^{(a)}_{u,v}e_aq \quad \text{if $q \in \bH$}.\label{M3}
\end{align}
In addition, since $f(T)$ is self--adjoint, we have that
\begin{align*}
\sum_{a=0}^3\int_{\ssp(T)}f \, d\mu^{(a)}_{v,u} \, e_a
&=\b v | f(T)u\k=\overline{\b f(T)u | v \k}=\overline{\b u | f(T)v \k}= \\
&=\int_{\ssp(T)} f \, d\mu^{(0)}_{u,v}-\sum_{a=0}^3\int_{\ssp(T)}f \, d\mu^{(a)}_{v,u} \, e_a
\end{align*}
and hence
\begin{align} \label{M1}
& \mu^{(0)}_{u,v}=\mu^{(0)}_{v,u} \; \text{ and } \; \mu^{(a)}_{u,v}=-\mu^{(a)}_{v,u} \; \text{ if }a \in \{1,2,3\}.
\end{align}
Moreover, equality (\ref{M0}) implies that
\beq \label{M4}
|\nu^{(a)}_{u,v}(\ssp(T))| \leq \|u\| \, \|v\| \quad \text{if $a \in \{0,1,2,3\}$}.
\eeq
Indeed, it holds:
\begin{align*}
4|\nu^{(a)}_{u,v}(\ssp(T))| & =\|ue_a+v\|^2-\|ue_a-v\|^2=
4\mr{Re}(\b ue_a|v \k)= \\
&=4\mr{Re}(-e_a \b u|v \k) \leq 4|\b u|v \k| \leq 4\|u\| \, \|v\|.
\end{align*}

Viewing $\mu_u$ as a signed measure, one can compare (\ref{defM0}) with (\ref{defM}) using the same uniqueness property established in Riesz' theorem for complex Borel measures \cite[Theorem~6.19]{RudinARC}, obtaining: 
\beq \label{M5}
\nu^{(0)}_{u,u}=\mu_u \quad \text{and} \quad \nu^{(a)}_{u,u}=0 \; \text{ if $a \in \{1,2,3\}$}.
\eeq 

\textit{Step III.} Let $f \in M(\ssp(T),\bR)$. Define the quadratic form
\[
\sH \times \sH \ni (u,v) \mapsto Q_f(u,v):= \sum_{a=0}^3\int_{\ssp(T)} f \, d\nu^{(a)}_{u,v} \, e_a \in \bH.
\]
In view of (\ref{M2})--(\ref{M4}), that quadratic form is right $\bH$--linear in the second component and quaternionic hermitian; that is, $Q_f(u,vq+v'q')=Q_f(u,v)q+Q_f(u,v')q'$ and $\overline{Q_f(u,v)}=Q_f(v,u)$ if $u,v,v' \in \sH$ and $q,q' \in \bH$. Moreover, $Q_f$ is bounded:
\[
|Q_f(u,v)| \leq 4\|f\|_\infty \|u\|\,\|v\|.
\]
Quaternionic Riesz' theorem (see Theorem \ref{quat_Riesz})
 immediately implies that there exists a unique operator in $\gB(\sH)$, we shall indicate by $f(T)$ again, such that
\beq \label{eq:Q_f}
\b u|f(T) v \k=Q_f(u,v) \quad \text{if $u,v \in \sH$.}
\eeq
That operator is self--adjoint; that is, $(\mr{ii})$ is verified. Indeed, it holds:
\[
\b u|f(T)v \k=Q_f(u,v)=\overline{Q_f(v,u)}=\overline{\b v|f(T)u\k}=\b f(T)u|v \k \quad \text{if }u,v \in \sH.
\]
The map $\widehat{\Phi}_T:M(\ssp(T),\bR) \lra \gB(\sH)$, sending $f$ into $f(T)$, satisfies $(\mr{a})$; that is, it is an extension of $\Phi_T$. This is immediate consequence of (\ref{defM}) and of the definition of $f(T)$. Point $(\mr{a})$ and points $(\mr{i})$ and $(\mr{ii})$ of Theorem \ref{teoA1} implies at once $(\mr{i})$.

Let us prove that $\widehat{\Phi}_T$ is a homomorphism of real algebras. By construction, it is evident that it is $\bR$--linear. Let $x,y \in \sH$. Given $f,g \in M(\ssp(T),\bR)$, we must show that $\widehat{\Phi}_T(fg)=\widehat{\Phi}_T(f)\widehat{\Phi}_T(g)$. Firstly, suppose that $f,g \in \CC(\ssp(T),\bR)$. We have:
\begin{align*}
\sum_{a=0}^3 \int_{\ssp(T)} f \, d\nu^{(a)}_{u,g(T)v} \, e_a & =\b u |\Phi_T(f)\Phi_T(g) v \k=\b u|\Phi_T(fg) v\k=\\
&=\sum_{a=0}^3 \int_{\ssp(T)} f g \, d\nu^{(a)}_{u,v} \, e_a \quad \text{if }u,v \in \sH.
\end{align*}
Once again exploiting the uniqueness of the signed measures, we conclude that $d\nu^{(a)}_{u,g(T)v}=g\, d\nu^{(a)}_{u,v}$ for every $a \in \{0,1,2,3\}$. In this way, if $f \in M(\ssp(T),\bR)$ and $g \in \mscr{C}(\ssp(T),\bR)$, it holds:
\begin{align*}
\sum_{a=0}^3 \int_{\ssp(T)} \, f g \, d\nu^{(a)}_{u,v} \, e_a &=\sum_{a=0}^3 \int_{\ssp(T)} f \, d\nu^{(a)}_{u,g(T)v} \, e_a=\b u|f(T)g(T)v \k=\\
&=\b f(T)u|g(T)v \k=\sum_{a=0}^3 \int_{\ssp(T)} g \, d\nu^{(a)}_{f(T)u,v} \, e_a.
\end{align*}
It follows that $f \, d\nu^{(a)}_{u,v}=d\nu^{(a)}_{f(T)u,v}$ for every $a \in \{0,1,2,3\}$. Finally, if $f,g \in M(\ssp(T),\bR)$, then we have that
\begin{align*}
\b u|\widehat{\Phi}_T(fg)v \k &=\sum_{a=0}^3 \int_{\ssp(T)} \, f g \, d\nu^{(a)}_{u,v} \, e_a=\sum_{a=0}^3 \int_{\ssp(T)} g \, d\nu^{(a)}_{f(T)u,v} \, e_a=\\
&=\b \widehat{\Phi}_T(f)u|\widehat{\Phi}_T(g)v \k=\b u|\widehat{\Phi}_T(f)\widehat{\Phi}_T(g)v \k.
\end{align*}
Since this is true for every $u,v \in \sH$, we infer that $\widehat{\Phi}_T(fg)=\widehat{\Phi}_T(f)\widehat{\Phi}_T(g)$, as desired.

Point $(\mr{c})$ is quite evident. Indeed, if $f \in M(\ssp(T),\bR)$ and $f \geq 0$, then $\sqrt{f} \in M(\ssp(T),\bR)$ and $f=\sqrt{f}\sqrt{f}$. Therefore, it holds:
\[
\b u| f(T) u\k=\b u| \sqrt{f}(T) \sqrt{f}(T) u\k=\b \sqrt{f}(T) u|\sqrt f(T) u\k \geq 0 \quad \text{if }u \in \sH.
\]

Let us prove $(\mr{d})$. First, observe that, if $u \in \sH$ and $E \in \mscr{B}(\ssp(T))$, then (\ref{M5}) and (\ref{eq:Q_f}) imply that
\[
\b u|\widehat{\Phi}_T(\chi_E)u\k=\sum_{a=0}^3\int_{\ssp(T)} \chi_E \, d\nu^{(a)}_{u,u} \, e_a=\int_{\ssp(T)}\chi_E \, d\mu_u=\mu_u(E).
\]
Moreover, bearing in mind (\ref{M0}), we have:
\begin{align*}
\|f(T)u\|^2 &= \b \widehat{\Phi}_T(f)u|\widehat{\Phi}_T(f)u\k=\b u|\widehat{\Phi}_T(f)^2 u\k=\b u|\widehat{\Phi}_T(f^2)u\k=\\
&=\int_{\ssp(T)} f^2 \, d\mu_u
 \leq \|f\|_{\infty}^2 \, \mu_u(\ssp(T))=\|f\|^2_{\infty} \|u\|^2\:,
\end{align*}
so that $\|f(T)\| \leq \|f\|_\infty$. This completes the proof of $(\mr{d})$ and proves $(\mr{b})$.

We pass to prove $(\mr{e})$. Let $f \in M(\ssp(T),\bR)$ and let $S \in \gB(\sH)$ be an operator which commutes with $T$. We must show that $S \, \widehat{\Phi}_T(f)=\widehat{\Phi}_T(f) \, S$. Thanks to point $(\mr{d})$ of Theorem \ref{teoA1}, the latter equality holds if $f \in \CC(\ssp(T),\bR)$. In this way, repeating the argument employed to prove that $\widehat{\Phi}_T$ is a homomorphism of real algebras, one obtains that $\nu_{S^*u,v}^{(a)}=\nu_{u,Sv}^{(a)}$ if $u,v \in \sH$ and $a \in \{0,1,2,3\}$. It follows that
\begin{align*}
\b u|S \, \widehat{\Phi}_T(f) v \k & =\b S^*u| \widehat{\Phi}_T(f) v \rangle=\sum_{a=0}^3 \int_{\ssp(T)} f \, d\nu^{(a)}_{S^*u,v} \, e_a=\\
&=\sum_{a=0}^3\int_{\ssp(T)} f \, d\nu^{(a)}_{u,Sv} \, e_a=\langle u| \widehat{\Phi}_T(f) \, Sv \rangle,
\end{align*}
if $u,v \in \sH$ and hence $S \, \widehat{\Phi}_T(f)=\widehat{\Phi}_T(f) \, S$.

It remains to show $(\mr{f})$ (which implies $(\mr{iii})$). If $\{f_n\}_{n \in \bN}$ is a bounded sequence in $M(\ssp(T),\bR)$, which converges pointwisely to some $f \in  M(\ssp(T),\bR)$, then 
we can apply Lebesgue's dominated convergence theorem to $\{f_n-f\}_{n \in \bN}$. In this way, thanks to (\ref{eq:f^2}), we obtain:
\[
\lim_{n \to +\infty}\| (f_n(T) - f(T)) u\|^2 =
\lim_{n \to +\infty}\int_{\ssp(T)}  (f_n-f)^2  d\mu_u =0,
\]

Now we consider the \textit{uniqueness issue}. Assume that there is another continuous unity--preserving homomorphism $\Psi: M(\ssp(T),\bR) \lra \gB(\sH)$ satisfying $(\mr{i})$--$(\mr{iii})$. By Theorem~\ref{teoA1} and point $(\mr{i})$, we have that
$\Psi(f)=\Phi_T(f)=\widehat{\Phi}_T(f)$ if $f \in \CC(\ssp(T),\bR)$. Let $u \in \sH$. Define the map $\nu^{(\Psi)}_u:\mscr{B}(\ssp(T)) \lra \R^+$ by setting
\[
\nu^{(\Psi)}_u(E):=\b u|\Psi(\chi_E)u\k \quad \text{if }E \in \mscr{B}(\ssp(T)).
\]
Thanks to the fact that $\Psi$ is a homomorphism satisfying $(\mr{ii})$, we have that such a map $\nu^{(\Psi)}_u$ is well--defined. Indeed, if $E \in \mscr{B}(\ssp(T))$, it holds:
\[
\b u|\Psi(\chi_E)u\k=\b u|\Psi(\chi_E\chi_E)u\k=\b u|\Psi(\chi_E)\Psi(\chi_E)u\k=\b \Psi(\chi_E)u|\Psi(\chi_E)u\k \in \bR^+.
\] 
Observe that $\nu^{(\Psi)}_u(\emptyset)=\b u|\Psi(0)u\k=0$. Moreover, if $\{E_n\}_{n \in \bN}$ is a sequence of pairwise disjoint sets in $\mscr{B}(\ssp(T))$, then point $(\mr{iii})$ implies that
\begin{align*}
\textstyle \nu^{(\Psi)}_u(\bigcup_{n \in \bN}E_n) &\textstyle =\left\b u \left| \, \lim_{k \to +\infty}\Psi\left(\sum_{n=0}^k\chi_{E_n}\right)u\right.\right\k=\\
&\textstyle =\lim_{k \to +\infty}\sum_{n=0}^k\left\b u \left| \, \Psi\left(\chi_{E_n}\right)u\right.\right\k=\\
&\textstyle =\sum_{n \in \bN}\left\b u \left| \, \Psi\left(\chi_{E_n}\right)u\right.\right\k=\sum_{n \in \bN}\nu^{(\Psi)}_u(E_n).
\end{align*}
This proves that $\nu^{(\Psi)}_u$ is a positive $\sigma$--additive Borel measure on $\ssp(T)$. Notice that that measure is finite as $ \nu^{(\Psi)}_u(\ssp(T))=\langle u|\Psi(1_{\ssp(T)}) u \rangle=\|u\|^2$. In view of Theorem~2.18 in \cite{RudinARC}, such a measure is regular. By standard approximation of measurable functions with simple functions and by point $(\mr{iii})$, we infer that 
\[
\langle u |\Psi(f) u \rangle = \int_{\ssp(T)} f \, d\nu_u^{(\Psi)} \quad \text{if $f \in M(\ssp(T),\bR)$}.
\]
On the other hand, bearing in mind (\ref{defM0}), if $f \in \CC(\ssp(T),\bR)$, one has
\[
\int_{\ssp(T)} f \, d\mu_u=\langle u |\widehat{\Phi}_T(f) u \rangle=\langle u |\Psi(f) u \rangle=\int_{\ssp(T)} f \, d\nu_u^{(\Psi)}.
\]
 The already exploited uniqueness property implies that $\mu_u=\nu^{(\Psi)}_u$. In this way, for every $f \in M(\ssp(T),\bR)$, it holds:
\[
\langle u | (\Psi(f)  - \widehat{\Phi}_T(f)) u \rangle=\int_{\ssp(T)} f \, d\nu^{(\Psi)}_u- \int_{\ssp(T)} f \, d\mu_u=0.
\]
Applying Proposition~\ref{Lemmadiag}, we have that $\Psi  = \widehat{\Phi}_A$.
\end{proof}


\subsection{Quaternionic $\boldsymbol{L^2}$-representation of measurable real functions of a self-adjoint operator}
In this section, exploiting the functional calculus previously constructed, we establish the key tool in proving the existence of an anti self--adjoint and unitary operator $J$ commuting with a given normal operator $T$ and with its adjoint $T^*$. Indeed, here we prove the existence of such $J$ in the special situation in which $T$ is self--adjoint. The existence of $J$ will be established by constructing a suitable $L^2$--realisation of the quaternionic Hilbert space $\sH$, where $T$ and every measurable real function $f(T)$ act as multiplicative operators.

Let $X$ be a set and let $\mu$ be a positive $\sigma$--additive measure on $X$. Denote by 
 $L^2(X,\bH;\mu)$ the quaternionic Hilbert space of the squared integrable functions $f: X \lra \bH$ with respect to $\mu$. Here $L^2(X,\bH;\mu)$ is considered as a quaternionic right module in which the sum of functions is the standard pointwise sum and the right multiplication by quaternions is defined pointwisely. The scalar product is defined as follows: if $f,g \in L^2(X,\bH;\mu)$ and $f_0,f_1,f_2,f_3,g_0,g_1,g_2,g_3$ are the real--valued functions on $X$ such that $f=\sum_{a=0}^3 f_a e_a$ and $g=\sum_{a=0}^3 g_a e_a$, then we set
\[
(f|g):=\int_X\overline{f}g \, d\mu=\sum_{a,b=0}^3 \int_X f_a g_b \, d\mu \, \overline{e}_ae_b.
\]
Here $\{e_0,e_1,e_2,e_3\}$ is a fixed orthonormal basis of $\bH$. In particular, the norm $\|f\|_{L^2}$ of $f$ in $L^2(X,\bH;\mu)$ is given by $\|f\|_{L^2}=\left(\int_X|f|^2 \, d\mu\right)^{1/2}$ (see e.g.~\cite{BDS}).

\begin{theorem}\label{teoA3}
Let $\sH$ be a quaternionic Hilbert space and let $T \in \gB(\sH)$ be a self--adjoint operator. Then the following assertions hold.
\begin{itemize}
\item[$(\mr{a})$] There exists an orthogonal decomposition of $\sH$ into closed subspaces $\sH =\bigoplus_{\ell \in \Lambda} \sH_{\ell}$ with the following two properties:
\begin{itemize}
\item[$(\mr{i})$] $f(T)(\sH_\ell) \subset \sH_\ell$ for every $f \in M(\ssp(T),\bR)$ and $\ell \in \Lambda$.

\item[(ii)] For each $\ell \in \Lambda$, there is a finite positive regular Borel measure $\mu_\ell$ on $\ssp(T)$ and an isometric isomorphism $U_\ell:L^2(\ssp(T),\bH;\mu_\ell) \lra \sH_\ell$ such that
\[
\left(U^{-1}_\ell f(T)\rr_{\sH_\ell} U_\ell\right)(g)=fg
\]
for every $f \in M(\ssp(T),\bR)$ and $g \in L^2(\ssp(T),\bH;\mu_\ell)$, where $f(T)\rr_{\sH_\ell}$ denotes the restriction of $f(T)$ from $\sH_\ell$ to $\sH_\ell$.
\end{itemize}
\item[$(\mr{b})$] There exists $J \in \gB(\sH)$ such that:
\[
\text{$J^*=-J$, $J^*J=\1$ and $Jf(T)=f(T)J$ for every $f \in M(\ssp(T),\bR)$.}
\]
\end{itemize}
\end{theorem}
\begin{proof}
Let $u \in \sH$. Denote by $H_u$ the vector subspace of $\sH$ consisting of all vectors $x$ having the following property: there exists a non--empty finite fami\-ly $\{E_1,\ldots,E_N\}$ of pairwise disjoint Borel subsets of $\ssp(T)$ and $q_1,\ldots,q_N \in \bH$ such that $x=\sum_{j=1}^N{\chi}_{E_j}(T)uq_j$. Indicate by $\sH_u$ the closure of $H_u$ in $\sH$. Let $\mu_u:\mscr{B}(\ssp(T)) \lra \bR^+$ be the finite positive regular Borel measure of $\ssp(T)$ sending $E$ into $\b u|\chi_E(T) u \k=\|\chi_E(T)u\|^2$ (see point $(\mr{d})$ of Theorem \ref{teoA2}). Define the vector subspace $L^2_u$ of $L^2(\ssp(T),\bH;\mu_u)$ analogous to $H_u$: a function $f:\ssp(T) \lra \bH$ belongs to $L^2_u$ if it can be represented as a finite sum $\sum_{j=1}^M\chi_{F_j}p_j$, where $F_1,\ldots,F_M$ are pairwise disjoint Borel subsets of $\ssp(T)$ and $p_1,\ldots,p_M$ are quaternions.

We subdivide the remainder of this proof into two steps.

\textit{Step I.} First, suppose that $\sH_u=\sH$ for some fixed $u \in \sH$. Let $V:L^2_u \lra \sH$ be the map defined as follows:
\beq \label{tentative}
V(f):=\sum_{j=1}^N\chi_{E_j}(T)uq_j \quad \text{if }f=\sum_{j=1}^N\chi_{E_j}q_j.
\eeq
It is quite easy to verify that $V$ is well--defined; that is, $V(f)$ depends only on $f$, and not on its representation $\sum_{j=1}^N\chi_{E_j}q_j$. Moreover, it is evident that $V$ is right $\bH$--linear. Let us prove that $V$ is isometric. Bearing in mind that $\chi_{E_j}\chi_{E_k}=\chi_{E_j}$ if $j=k$ and $\chi_{E_j}\chi_{E_k}=0$ for $j \neq k$, we infer that
\begin{align*}
\left\|\sum_{j=1}^N\chi_{E_j}q_j \right\|_{L^2}^2 &= \int_{\ssp(T)} \sum_{j,k=1}^N\chi_{E_j}\chi_{E_k} \, d\mu_u \,  \overline{q_j} q_k=\\
&=\int_{\ssp(T)} \sum_{j=1}^N\chi_{E_j} |q_j|^2 \, d\mu_u= \sum_{j=1}^N |q_j|^2  \mu(E_j).  
\end{align*}
Since $\chi_{E_j}(T)\chi_{E_k}(T)=  (\chi_{E_j}\chi_{E_k})(T)$, we have that $\chi_{E_j}(T)\chi_{E_k}(T)=\chi_{E_j}(T)$ if $j=k$ and $\chi_{E_j}(T)\chi_{E_k}(T)=0$ for $j \neq k$. It follows that
\begin{align*}
\left\|\sum_{j=1}^N\chi_{E_j}(T)uq_j \right\|_{\sH}^2 &=
\sum_{j,k=1}^N\left\langle \chi_{E_j}(T)uq_j \,\big|\, \chi_{E_k}(T)uq_k \right\rangle=\\
&=\sum_{j,k=1}^N \overline{q_j}\left\langle u \,\big|\, \chi_{E_j}(T) \chi_{E_k}(T)u \right\rangle q_k=\\
&=\sum_{j=1}^N \overline{q_j}\left\langle u \,\big|\, \chi_{E_j}(T) u \right\rangle q_j
=\sum_{j=1}^N |q_j|^2\left\langle u \,\big|\, \chi_{E_j}(T) u \right\rangle=\\
&=\sum_{j=1}^N |q_j|^2 \mu_u(E_j)=\left\|\sum_{j=1}^N\chi_{E_j}q_j \right\|_{L^2}^2.
\end{align*}
Observe that $L^2_u$ is dense in $L^2(\ssp(T),\bH;\mu_u)$. Indeed, if $f=\sum_{a=0}^3f_ae_a$ is a function in $L^2(\ssp(T),\bH;\mu_u)$, then we can apply Theorem 1.17 of \cite{RudinARC} to each component $f_a$ of $f$, obtaining a sequence of $L^2_u$ which converges to $f$ in $L^2(\ssp(T),\bH;\mu_u)$. Bearing in mind that $V$ is isometric and its range is dense in $\sH$ by hypothesis, we have that $V$ uniquely extends to an isometric isomorphism
$U:L^2(\ssp(T),\bH;\mu_u) \lra \sH$. The construction also implies that
\beq 
\left(U^{-1} \chi_E(T) U\right)(\chi_F)=\chi_E \chi_F \quad \text{if }E,F \in \mscr{B}(\ssp(T)).
\eeq
Indeed, we have:
\begin{align*}
\chi_E(T) (U(\chi_F))=\chi_E(T)\chi_F(T)u=(\chi_E\chi_F)(T)u=U(\chi_E\chi_F).
\end{align*}
By combining Theorem~1.17 in \cite{RudinARC} 
and point $(\mr{f})$ of Theorem~\ref{teoA2}, we obtain at once that
\[
(U^{-1} f(T) U)(\chi_F)= f \, \chi_F \quad \text{if }f \in M(\ssp(T),\bR).
\]
Finally, making use again of Theorem~1.17 of \cite{RudinARC}, we obtain that
\[
(U^{-1} f(T) U)g=fg \; \text{ $\mu_u$--a.e. in $\ssp(T)$}
\]
for every $f \in M(\ssp(T),\bR)$ and for every $g \in L^2(\ssp(T),\bH;\mu_u)$.

Fix $\jmath \in \bS$ and define the operator $Y \in \gB(L^2(\ssp(T),\bH;\mu_u))$ by setting $Yg:=\jmath \, g$. Evidently, $-YY$ is the identity of $L^2(\ssp(T),\bH;\mu_u)$. Furthermore, we have that $Y^*=-Y$. Indeed, it holds:$(f|Yg)=\int_{\ssp(T)}\overline{f}\jmath g \, \mu_u=-\int_{\ssp(T)}\overline{\jmath f} g \, \mu_u=(-Yf|g)$ for every $f,g \in L^2(\ssp(T),\bH;\mu_u)$. Define $J \in \gB(\sH)$ as follows: $J:=UYU^{-1}$. It follows that $JJ=-\1$ and, it being $U$ isometric, $J^*=-J$. Let now $f \in M(\ssp(T),\bR)$. Since $f$ is real--valued, we have that $Yf=\jmath f=f\jmath$ and hence
\begin{align*}
(U^{-1}Jf(T)U)(g) &=(YU^{-1}f(T)U)(g)=\jmath fg=f\jmath g=\\
&=(U^{-1}f(T)UY)(g)=(U^{-1}f(T)JU)(g)
\end{align*}
for every $g \in L^2(\ssp(T),\bH;\mu_u)$. It follows that $J$ commutes with $f(T)$ for every $f \in M(\ssp(T),\bR)$. This proves the theorem if $\sH_u=\sH$ for some $u \in \sH$.

\textit{Step II.} Let us consider the case in which $\sH_u \neq \sH$ for every $u \in \sH$. Since $\sH_0=\{0\}$, we have that $\sH \neq \{0\}$. Choose $u \in \sH \setminus \{0\}$. Observe that $u=\1 u=1_{\ssp(T)}(T)u=\chi_{\ssp(T)}(T)u \in H_u$. It follows that $H_u \neq \{0\}$ and hence $\sH_u \neq \{0\}$.

Fix $f \in M(\ssp(T),\bR)$. Let us show that
\beq \label{eq:H_u}
f(T)(\sH_u) \subset \sH_u.
\eeq
Let $x=\sum_{j=1}^N\chi_{E_j}(T)uq_j$ be an element of $H_u$ and let $\{s_n=\sum_{k=1}^{O_n}\chi_{F_{nk}}r_{nk}\}_{n \in \bN}$ be a bounded sequence of real simple functions on $\ssp(T)$ converging pointwisely to $f$. By point $(\mr{f})$ of Theorem~\ref{teoA2}, it follows that
\begin{align*}
f(T)(x) &=\sum_{j=1}^Nf(T)\chi_{E_j}(T)uq_j=\sum_{j=1}^N(f\chi_{E_j})(T)uq_j=\sum_{j=1}^N(\chi_{E_j} f)(T)uq_j=\\
&=\sum_{j=1}^N\chi_{E_j}(T)f(T)uq_j=\lim_{n \to +\infty} \sum_{j=1}^N\sum_{k=0}^{O_n}\chi_{E_j}(T) \chi_{F_{nk}}(T) uq_jr_{nk}=\\
&=\lim_{n \to +\infty} \sum_{j=1}^N\sum_{k=0}^{O_n}\chi_{E_j \cap F_{nk}}(T)uq_jr_{nk}
\end{align*}
Since each $\sum_{j=1}^N\sum_{k=0}^{O_n}\chi_{E_j \cap F_{nk}}(T)uq_jr_{nk}$ belongs to $H_u$, $f(T)(x)$ belongs to $\sH_u$. Now inclusion (\ref{eq:H_u}) follows by the continuity of $f(T)$. Combining (\ref{eq:H_u}) with the fact  that $f(T)$ is self--adjoint, we infer that $f(T)(\sH_u^\perp) \subset \sH_u^\perp$. In this way, applying Step I to the $\sH_u$'s, we see at once that, in order to complete the proof, it suffices to show the existence of a subset $\{u_\ell\}_{\ell \in \Lambda}$ of $\sH \setminus \{0\}$ such that $H=\bigoplus_{\ell \in \Lambda}\sH_{u_\ell}$ and $\sH_{u_\ell} \perp \sH_{u_{\ell'}}$ for every $\ell,\ell' \in \Lambda$ with $\ell \neq \ell'$. However, this can be done by a standard argument involving Zorn's lemma.
\end{proof}

\begin{remark}\label{remark-many-J}
Let us focus on the proof of $(\mr{b})$. Consider a function $h \in M(\ssp(T),\bR)$ with 
$h(t)^2=1$ for every $t \in \ssp(T)$ and re--define $(Yg)(x):=\jmath h(x) g(x)$ for every $g\in L^2(\sigma_S(T),\bH; \mu_\nu)$. It easily arises that $YY=-\1$ and $Y^*=-Y$. Passing to $\sH_u$, and finally to the whole $\sH$ via Zorn's lemma, one obtains other definitions of $J$ verifying all requirements in $(\mr{b})$. Therefore, $J$ turns out to be highly undetermined by $T$.
\end{remark}

\subsection{The commuting operator $\boldsymbol{J}$}\label{sec:commuting operator}
We are eventually in a position to prove the exi\-stence of an anti self--adjoint and unitary operator $J$, commuting with any given normal operator $T$ and with its adjoint $T^*$.

\begin{theorem}\label{teoext}
Let $\sH$ be a quaternionic Hilbert space and let $T \in \gB(\sH)$ be a normal operator. Then there exists an anti self--adjoint and unitary operator $J \in \gB(\sH)$ such that 
\beq \label{OMEGA}
TJ=JT, \quad T^*J=JT^*
\eeq

and
\beq \label{decA0-bis}
T=(T+T^*)\frac{1}{2}+J|T-T^*|\frac{1}{2}.
\eeq
In particular, $J$ is uniquely determined by $T$ on $\mi{Ker}(T-T^*)^\perp$ and the operators $(T+T^*)\frac{1}{2}$, $|T-T^*|\frac{1}{2}$ and $J$ commute mutually.
\end{theorem}
\begin{proof}
Since $T$ is normal, it is easy to verify that both the operators $T$ and $T^*$ preserve the decomposition $\sH=\mi{Ker}(T-T^*) \oplus \mi{Ker}(T-T^*)^\perp$. Furthermore, $T$ and $T^*$ coincide with the self--adjoint operator $(T+T^*)\frac{1}{2}$ on $\mi{Ker}(T-T^*)$. Denote by $T' \in \gB(\mi{Ker}(T-T^*))$ the self--adjoint operator of $\mi{Ker}(T-T^*)$ obtained restricting $(T+T^*)\frac{1}{2}$. Let $J_0$ be the operator obtained applying Theorem~\ref{teobastardo2} to $T$. Since $J_0$ preserves the above--mentioned decomposition of $\sH$, it is anti self--adjoint,  unitary on $\mi{Ker}(T-T^*)^\perp$ and satisfies (\ref{decA0}), in order to complete the proof, it suffices to find an anti self--adjoint unitary operator $J_0'$ in $\gB(\mi{Ker}(T-T^*))$ commuting with $T'$. Indeed, if one has such an operator $J_0'$, then $J:=J_0' \oplus J_0$ has the desired properties. On the other hand, $J_0'$ is an operator given by point $(\mr{b})$ of Theorem~\ref{teoA3}, applied to $T'$ with $f$ equal to the inclusion map $\ssp(T') \hookrightarrow \bR$.
\end{proof}

\begin{remark}\label{remarkJunded}
If $Ker(T-T^*)$ is not trivial, differently from $J_0$, the operator $J$ cannot commute with every operator commuting with $T-T^*$. Otherwise, $J$ would preserve $\mi{Ker}(T-T^*)$ and the restriction $J'$ of $J$ from $\mi{Ker}(T-T^*)$ to itself would be an anti self--adjoint and unitary operator in $\gB(\mi{Ker}(T-T^*))$, commuting with every element of $\gB(Ker(T-T^*))$. As one can easily prove, there are no operators with these properties in $\gB(Ker(T-T^*))$. Finally, we stress that, as a consequence of Remark \ref{remark-many-J}, $J_0'$ is highly undetermined from $T$ and that indeterminacy affect the definition of $J$ itself. 
\end{remark}

To go on, we need a relevant notion we shall extensively use in the remainder of the paper. Recall that, given a subset $\K$ of $\bC$, we define the \emph{circularization $\OO_\K$ of $\K$ $($in $\bH \, )$} by setting
\beq \label{circularization}
\OO_\K:=\{\alpha+\jmath\beta \in \bH \,|\, \alpha,\beta \in \bR, \alpha+i\beta \in \K, \jmath \in \bS\}.
\eeq

In the following, we denote by $\sigma(B)$ and $\rho(B)$ the standard spectrum and resolvent set
of a bounded operator $B$ of a complex Hilbert space, respectively.

\begin{proposition} \label{propinterssigma} 
Let $\sH$ be a quaternionic Hilbert space, let $T \in \gB(\sH)$ be a normal operator, let $J \in \gB(\sH)$ be an anti self--adjoint and unitary operator satisfying (\ref{OMEGA}), let $\imath \in \bS$ and let $\sH_\pm^{J\imath}$ be the complex subspaces of $\sH$ associated with $J$ and $\imath$ (see Definition~\ref{defHJ}). Then we have that
\begin{itemize}
 \item[$(\mr{a})$] $T(\sH_+^{J\imath}) \subset \sH_+^{J\imath}$ and $T^*(\sH_+^{J\imath}) \subset \sH_+^{J\imath}$.
\end{itemize}
Moreover, if $T\rr_{\sH_+^{J\imath}}$ and $T^*\rr_{\sH_+^{J\imath}}$ denote the $\bC_{\imath}$--complex operators in $\gB(\sH_+^{J\imath})$ obtained restricting respectively $T$ and $T^*$ to $\sH_+^{J\imath}$, then it holds:
\begin{itemize}
 \item[$(\mr{b})$] $(T\rr_{\sH_+^{J\imath}})^*= T^*\rr_{\sH_+^{J\imath}}$.
 \item[$(\mr{c})$] $\sigma(T\rr_{\sH_+^{J\imath}}) \cup \overline{\sigma(T\rr_{\sH_+^{J\imath}})}= \ssp(T)\cap \bC_i$. Here $\sigma(T\rr_{\sH_+^{J\imath}})$ is considered as a subset of $\bC_\imath$ via the natural identification of $\bC$ with $\bC_\imath$ induced by the real vector isomorphism $\bC \ni \alpha+i\beta \mapsto \alpha+\imath\beta \in \bC_\imath$.
 \item[$(\mr{d})$]
$\ssp(T)=\OO_\K$, where $\K:=\sigma(T\rr_{\sH_+^{J\imath}})$.
\end{itemize}

An analogous statement holds for $\sH^{J\imath}_-$.
\end{proposition}
\begin{proof} Recall that $\sH_+^{J\imath}=\{u \in \sH \, | \, Ju=u\imath\}$. Let $u \in \sH_+^{J\imath}$. Since $J$ commute  with both $T$ and $T^*$, we have that $J(Tu)=TJu=(Tu)\imath$ and $J(T^*u)=T^*Ju=(T^*u)\imath$; that is, $Tu,T^*u \in \sH_+^{J\imath}$. Bearing in mind Lemma \ref{lemmalemme}, it follows immediately that $T_+:=T\rr_{\sH_+^{J\imath}}$ and $T^*\rr_{\sH_+^{J\imath}}$ belongs to $\gB(\sH_+^{J\imath})$, and $(\mr{b})$ holds true.  

Let us pass to prove the following equality
\[
\srho(T) \cap \bC_\imath= \rho(T\rr_{\sH_+^{J\imath}}) \cap \overline{\rho(T\rr_{\sH_+^{J\imath}})},
\]
which is equivalent to $(\mr{c})$. Let $q \in \srho(T) \cap \bC_\imath$. Since $J$ commutes with $\Delta_q(T)$ and $\Delta_q(T)$ admits inverse in $\gB(\sH)$, it follows that $J$ commutes also with $\Delta_q(T)^{-1}$ and hence the restriction $\Delta_q(T)_+$ of $\Delta_q(T)$ from $\sH_+^{J\imath}$ into itself is a well--defined invertible operator in $\gB(\sH_+^{J\imath})$. If $\1_+$ denotes the identity operator on $\sH_+^{J\imath}$, then $\Delta_q(T)_+=(T_+-\1_+q)(T_+-\1_+\overline{q})$ in $\gB(\sH_+^{J\imath})$. Since $\Delta_q(T)_+$ is invertible and the operators $\Delta_q(T)_+$, $T_+-\1_+q$ and $T_+-\1_+\overline{q}$ commute with each other, we infer that both $T_+-\1_+q$ and $T_+-\1_+\overline{q}$ are invertible in $\gB(\sH_+^{J\imath})$. In other words, $q$ belongs to $\rho(T\rr_{\sH_+^{J\imath}}) \cap \overline{\rho(T\rr_{\sH_+^{J\imath}})}$. This proves the inclusion $\srho(T) \cap \bC_\imath \subset \rho(T\rr_{\sH_+^{J\imath}}) \cap \overline{\rho(T\rr_{\sH_+^{J\imath}})}$.

Let us prove the converse inclusion. Fix $q \in \rho(T\rr_{\sH_+^{J\imath}}) \cap \overline{\rho(T\rr_{\sH_+^{J\imath}})}$. The operators  $(T\rr_{\sH_+^{J\imath}}-\1_+\overline{q})^{-1}$ and 
$(T\rr_{\sH_+^{J\imath}}-\1_+q)^{-1}$ exist in $\gB(\sH_+^{J\imath})$. Their product is the inverse of 
$\Delta_q(T)_+$. Points $(\mr{a})$ and $(\mr{f})$ of Proposition~\ref{propestensione} immediately imply that $\Delta_q(T)$ is invertible in $\gB(\sH)$, so that $q \in \srho(T)$. Point $(\mr{c})$ is proved.

Finally, $(\mr{d})$ is a straightforward consequence of $(\mr{c})$ and of definition (\ref{circularization}).
\end{proof}

Denote by $\bR[X,Y]$ the ring of real polynomials in the indeterminates $X$ and~$Y$. As in the case of the ring $\bR[X]$ considered in Section \ref{subsec:operators}, we write the polynomials in $\bR[X,Y]$ \textit{with coefficients on the right}. In this way, given $Q \in \bR[X,Y]$, we have that $Q(X,Y)=\sum_{(h,k) \in \mscr{Q}}X^hY^kr_{hk}$ for some non--empty subset $\mscr{Q}$ of $\bN \times \bN$, where the $r_{hk}$'s are real numbers. If $A,B \in \gB(\sH)$, then we define the operator $Q(A,B) \in \gB(\sH)$ by setting
\beq \label{eq:Q(A,B)}
Q(A,B)=\sum_{(h,k) \in \mscr{Q}}A^hB^kr_{hk}.
\eeq

\begin{corollary} \label{cor:well-def-pol}
Let $\sH$ be a quaternionic Hilbert space, let $T \in \gB(\sH)$ be a normal operator, let $J \in \gB(\sH)$ be an anti self--adjoint and unitary operator satisfying (\ref{OMEGA}) and (\ref{decA0-bis}), let $\imath \in \bS$ and let $Q \in \bR[X,Y]$ be a polynomial such that $Q(\alpha,\beta)=0$ if $\alpha+\imath\beta \in \ssp(T)$. Define $A:=(T+T^*)\frac{1}{2}$ and $B:=|T-T^*|\frac{1}{2}$. Then $Q(A,B)$ is the null operator in $\gB(\sH)$.
\end{corollary}
\begin{proof}
Firstly, we observe that the operator $F:=Q(A,B)$ is self--adjoint and commutes with $J$. By applying the spectral radius formula (see (\ref{raggiospett})) and point $(\mr{c})$ of Proposition \ref{propinterssigma} to $F$, we obtain that
\[
\|F\|=\sup\left\{|q| \in \bR^+ \left| \, q \in  \sigma(F|_{\sH^{J\imath}_+}) \right.\right\},
\]
where $F|_{\sH^{J\imath}_+}$ denotes the operator in $\gB(\sH^{J\imath}_+)$ obtained restricting $F$ to $\sH^{J\imath}_+$. Points $(\mr{a})$ and $(\mr{b})$ of Proposition \ref{propinterssigma} imply that the restrictions $A|_{\sH^{J\imath}_+}$ of $A$ to $\sH^{J\imath}_+$ and $B|_{\sH^{J\imath}_+}$ of $B$ to $\sH^{J\imath}_+$ define operators in $\gB(\sH^{J\imath}_+)$. Moreover, we have that $A|_{\sH^{J\imath}_+}=\big(T|_{\sH^{J\imath}_+}+(T|_{\sH^{J\imath}_+})^*\big)\frac{1}{2}$ and $B|_{\sH^{J\imath}_+}=\big(T|_{\sH^{J\imath}_+}-(T|_{\sH^{J\imath}_+})^*\big)\frac{1}{2\imath}$. The latter equality follows from the definition of $\sH^{J\imath}_+$ (see Definition \ref{defHJ}) and from the fact that $B$ is equal to $-J(T-T^*)\frac{1}{2}$. In this way, if $f:\sigma(F|_{\sH^{J\imath}_+}) \lra \bR$ denotes the polynomial function sending $\alpha+\imath\beta$ into $Q(\alpha,\beta)$, then the continuous functional calculus theorem for complex normal operators (see \cite{RudinARC,Analysisnow,Moretti}) ensures that $F|_{\sH^{J\imath}_+}=f(A|_{\sH^{J\imath}_+},B|_{\sH^{J\imath}_+})$ and $\sigma(F|_{\sH^{J\imath}_+})=f(\sigma(T|_{\sH^{J\imath}_+}))$. By hypothesis, we know that $f(\sigma(T|_{\sH^{J\imath}_+}))=\{0\}$. It follows that $\sigma(F|_{\sH^{J\imath}_+})=\{0\}$ and hence $\|F\|=0$; that is, $F=0$, as desired.
\end{proof}

In the following define $\bC_\imath^+ := \{ q \in \bH \:|\: q= \alpha + \imath \beta\:, \alpha, \beta \in \bR\:, \beta \geq 0\}$ and, similarly,
$\bC_\imath^- := \{ q \in \bH \:|\: q= \alpha + \imath \beta\:, \alpha, \beta \in \bR\:, \beta \leq 0\}$. We have:

\begin{corollary}\label{corollaruCplus}
Let $\sH$ be a quaternionic Hilbert space, let $T \in \gB(\sH)$ be a normal operator, let $J \in \gB(\sH)$ be an anti self--adjoint and unitary operator satisfying (\ref{OMEGA}) and (\ref{decA0-bis}), and let $\imath \in \bS$. Then it holds:
\beq \label{eq:T}
\sigma(T|_{H^{J\imath}_+})=\ssp(T) \cap \bC_\imath^+,
\; \;
\sigma(T|_{H^{J\imath}_-})=\ssp(T) \cap \bC_\imath^-
\eeq
and hence
\beq \label{eq:TT}
\sigma(T_{H^{J\imath}_+})=\overline{\sigma(T|_{H^{J\imath}_-})}.
\eeq
\end{corollary}
\begin{proof}
Let $A:= (T+T^*)\frac{1}{2}$ and $B:=|T-T^*|\frac{1}{2}$,
so that $T=A+JB$. Denote by $T_\pm$, $A_\pm$ and $B_\pm$ the operators in $\gB(\sH_\pm^{J\imath})$ obtained restricting $T$, $A$ and $B$ to $\sH_\pm^{J\imath}$, respectively. Observe that $T_\pm=A_\pm \pm B_\pm \imath$ and $B_\pm\geq 0$, because $B$ is so. Therefore, the spectral theorem for self--adjoint operators in complex Hilbert spaces \cite{RudinFA,Moretti} implies that $\sigma(B_\pm) \subset \bR^+$. Let $f_\pm:\sigma(T_\pm) \lra \bC_\imath$ be the continuous $\bC_\imath$--complex  function sending $\alpha+\imath\beta$ into $\beta$. It follows that $B_\pm= f_\pm(\pm T_\pm)$, where the right--hand side is defined in the standard sense by means of continuous functional calculus for  normal operators in complex Hilbert spaces \cite{RudinFA,Moretti}. Moreover, using again the spectral theory in complex Hilbert spaces, we infer that $f_\pm(\sigma(\pm T_\pm))=\sigma(f_\pm(\pm T_\pm))$. In particular, we have:
\[
\pm f_\pm(\sigma(T_\pm))=f_\pm(\sigma(\pm T_\pm))=\sigma(f_\pm(\pm T_\pm))=\sigma(B_\pm)
\]
and hence $\pm f_\pm(\sigma(T_\pm)) \subset \bR^+$. The latter inclusion is equivalent to the following two:
$\sigma(T\rr_{H^{J\imath}_+}) \subset \bC_\imath^+$ and 
$\sigma(T\rr_{H^{J\imath}_-}) \subset \bC_\imath^-$. By combining these inclusions with Proposition \ref{propinterssigma}$(\mr{c})$, we immediately obtain equalities (\ref{eq:T}) and (\ref{eq:TT}).
\end{proof}


\subsection{Extension of $\boldsymbol J$ to a full left scalar multiplication of $\boldsymbol \sH$}
We conclude this section by proving how it is possible to extend the operator $J$ satisfying Theorem \ref{teoext}
to a full left scalar multiplication $\bH \ni q \mapsto L_q$ of $\sH$ in such a way that each $L_q$ commutes with $(T+T^*)\frac{1}{2}$ and $|T-T^*|\frac{1}{2}$.

\begin{theorem}\label{newtheorem}
Let $\sH$ be a quaternionic Hilbert space, let $T \in \gB(\sH)$ be a normal operator and let $\imath \in \bS$. Suppose that $T$ decomposes as follows:
\[
T=A+JB,
\]
where $A$ and $B$ are self--adjoint operators in $\gB(\sH)$ and $J \in \gB(\sH)$ is an anti self--adjoint and unitary operator commuting with $A$ and $B$ (for example, one can choose\insec$J$ as in the statement of Theorem \ref{teoext} and define $A:=(T+T^*)\frac{1}{2}$ and $B:=|T-T^*|\frac{1}{2}$). Then there exists a left scalar multiplication $\bH \ni q \mapsto L_q$ of $\sH$ such that $L_\imath=J$ and, for every $q \in \bH$, $L_qA=AL_q$ and $L_qB=BL_q$.
\end{theorem}
\begin{proof}
Let $N$ be a Hilbert basis of $\sH$ and let $\bH \ni q \mapsto L_q=
\sum_{z \in N}zq\b z|\cdot \k$ be the left scalar multiplication of $\sH$ induced by $N$. Observe that, if $S \in \gB(\sH)$ is an operator such that $\b z|Sz' \k \in \bR$ for every $z,z'\in N$,
then $L_qS=SL_q$. Indeed, if $z' \in N$, one has:
\[
(L_qS)(z')=\sum_{z \in N}zq\b z|Sz'\k=\sum_{z \in N}z\b z|Sz'\k q=S(z')q=S(z'q)=(SL_q)(z').
\]
Thanks to this fact, it is sufficient to find a Hilbert basis $N$ of $\sH$ such that the induced left scalar multiplication $\bH \ni q \mapsto L_q$ has the following properties: $L_\imath=J$ and $\b z|Az'\k, \b z|Bz'\k \in \bR$ for every $z,z'\in N$. Actually, it suffices to construct a Hilbert basis $N$ of the $\bC_\imath$--Hilbert space $\sH_+^{J\imath}$  satisfying
\beq \label{real}
\b z|A_+z' \k, \b z|B_+z' \k \in \bR \quad \text{if $z,z'\in N$},
\eeq
where $A_+$ and $B_+$ denote the operators in $\gB(\sH_+^{J\imath})$ obtained restricting $A$ and $B$ to $\sH_+^{J\imath}$, respectively. Indeed, by point $(\mr{f})$ of Proposition \ref{propJ}, $N$ is also a Hilbert basis of $\sH$ and $J=\sum_{z \in N} z\imath \langle z|\cdot\rangle$. Let $T_+ \in \gB(\sH_+^{J\imath})$ be the normal operator obtained restricting $T$ to $\sH_+^{J\imath}$. Observe that $A_+$ and $B_+$ are self--adjoint commuting operators and $T_+=A_++\imath B_+$ in $\gB(\sH_+^{J\imath})$. From the Spectral Representation Theorem (see Chap.X, sec.5 of \cite{DS}), there exists an orthogonal decomposition
$\bigoplus_{\ell \in \Lambda} \sH_\ell$ of $\sH_+^{J\imath}$ into closed subspaces and, for every $\ell \in \Lambda$, a positive $\sigma$--additive Borel measure $\mu_\ell$ on $\sigma(T_+)$ and an isometric isomorphism $U_\ell$ from $L^2_\ell:=L^2(\sigma(T_+),\bC;\mu_\ell)$ to $\sH_\ell$ such that $\sH_\ell$ is invariant under $T_+$ and $T_+^*$ (and thus
under $A_+$ and $B_+$), and
\[
(U_\ell T_+^{(\ell)} U_\ell^{-1})(f_\ell)(x+\imath y)=(x+\imath y)f_\ell(x+\imath y)
\]
for every $f_\ell \in L^2_\ell$, where $T_+^{(\ell)}$ is the restriction of $T_+$ from $\sH_\ell$ into itself. It follows that
\[
(U_\ell^{-1} A_+^{(\ell)}U_\ell)(f_\ell)(x+\imath y)=x f_\ell(x+\imath y)
\]
and
\[
(U_\ell^{-1} B_+^{(\ell)}U_\ell)(f_\ell)(x+\imath y)=y f_\ell(x+\imath y)
\]
for every $\ell \in \Lambda$ and $f_\ell \in L^2_\ell$, where $A_+^{(\ell)}$ and $B_+^{(\ell)}$ denotes the restrictions of $A_+$ and $B_+$ from $\sH_\ell$ into itself.

Fix $\ell \in \Lambda$ and define the operators $A':=U_\ell^{-1} A_+^{(\ell)}U_\ell$ and $B':=U_\ell^{-1} B_+^{(\ell)}U_\ell$ in $\gB(L^2_\ell)$. Let $\mc{P}_\ell$ be the set of all orthonormal subsets $F$ of $L^2_\ell$ such that each function in $F$ is real--valued. Equip $\mc{P}_\ell$ with the partial ordering induced by the inclusion. Since $\mc{P}_\ell$ is non--empty and inductive, Zorn's lemma ensures the existence of a maximal element $M_\ell$ of $\mc{P}_\ell$. Let us prove that $M_\ell$ is an Hilbert basis of $L^2_\ell$. Let $g \in L^2_\ell$ such that $g \perp M_\ell$ in $L^2_\ell$; that is, it holds:
\[
0=\int_{\sigma(T_+)}\overline{g}f \, d\mu_\ell=\int_{\sigma(T_+)}\mr{Re}(g)f \, d\mu_\ell-\imath \int_{\sigma(T_+)}\mr{Im}(g)f \, d\mu_\ell
\]
for each $f \in M_{\ell}$. It follows that $\mr{Re}(g) \perp M_\ell$ and $\mr{Im}(g) \perp M_\ell$. By the maximality of $M_\ell$, we infer that $\mr{Re}(g)=0=\mr{Im}(g)$ and hence $g=0$, as desired. Observe that, for every $f,f' \in L^2_\ell$, we have:
\beq \label{real1}
\int_{\sigma(T_+)}\overline{f}(A'f') \, d\mu_\ell=\int_{\sigma(T_+)}xf(x+\imath y)f'(x+\imath y) \, d\mu_\ell \in \bR
\eeq
and
\beq \label{real2}
\int_{\sigma(T_+)}\overline{f}(B'f') \, d\mu_\ell=\int_{\sigma(T_+)}yf(x+\imath y)f'(x+\imath y) \, d\mu_\ell \in \bR.
\eeq

Define $N_\ell:=U(M_\ell)$ for every $\ell \in \Lambda$ and $N:=\bigcup_{\ell \in \Lambda}N_\ell$. The reader observes that, given any $\ell \in \Lambda$, (\ref{real1}) and (\ref{real2}) are equivalent to (\ref{real}) with $z,z' \in N_\ell$. On the other hand, by construction, if $\ell,\ell' \in \Lambda$ with $\ell \neq \ell'$, $z \in N_\ell$ and $z' \in N_{\ell'}$, then $\b z|A'z'\k=0=\b z|B'z'\k$. This implies (\ref{real}) and completes the proof.
\end{proof}

As a consequence, we obtain:

\begin{corollary}
Let $\sH$ be a quaternionic Hilbert space and let $T \in \gB(\sH)$ be a normal operator. Then there exists a unitary operator $U:\sH \lra \sH$ such that
\[
UTU^*=T^*.
\]
\end{corollary}
\begin{proof}
Decompose $T$  as in
Theorem~\ref{teoext}:  $T=(T+T^*)\frac{1}{2}+J|T-T^*|\frac{1}{2}$. Define $A:=(T+T^*)\frac{1}{2}$ and $B:=|T-T^*|\frac{1}{2}$. Choose $\imath \in \bS$. By Theorem \ref{newtheorem}, there exists a left scalar multiplication $\bH \ni q \mapsto L_q$ of $\sH$ such that $L_\imath=J$ and, for every $q \in \bH$, $L_qA=AL_q$ and $L_qB=BL_q$. Let $p \in \bS$ such that $p\imath\overline{p}=-\imath$. Define $U:=L_p$. It holds:
\begin{align*}
UTU^* &=U(A+JB)U^*=L_pAL_{\overline{p}}+L_pJL_{\overline{p}}L_pBL_{\overline{p}}=\\
&=A+L_{p\imath\overline{p}}B =
A+L_{-\imath}B=A-JB=T^*.
\end{align*}
The proof is complete.
\end{proof}

We remind the reader that the preceding result is false in the complex setting: if $H$ is a complex Hilbert space and $T$ is the normal operator on $H$ obtained multiplying the identity operator by $i$, then it is immediate to verify that there does not exist any unitary operator $U$ on $H$ such that $UTU^*=T^*$.


\section{Relevant $C^*$-algebras of slice functions}

\textit{Throughout this section, $\K$ will denote a non--empty subset of $\bC$, invariant under complex conjugation.}


\subsection{Slice functions}
We recall basic definitions and results concerning slice functions, taking \cite{GhPe_AIM} and \cite{GhPe_Trends} as references.

Consider the complexification $\bH_\bC:=\bH \otimes_\bR \bC$ of $\bH$. We represent the elements $w$ of $\bH_\bC$ by setting $w=q+ip$ with $q,p \in \bH$, where $i^2=-1$. The product of $\bH_\bC$ is given by the following equality:
\beq \label{eq:H_C}
(q+ip)(q'+ip')=qq'-pp'+i(qp'+pq').
\eeq

Let us introduce the notion of stem function on $\K$.

\begin{definition}
Let $F:\K \lra \bH_\bC$ be a function and let $F_1,F_2:\K \lra \bH$ be the components of $F=F_1+iF_2$. We say that $F$ is a \emph{stem function} on $\K$ if $(F_1,F_2)$ forms an even--odd pair with respect to the imaginary part of $z \in \bC$; that is, $F_1(\bar{z})=F_1(z)$ and $F_2(\overline{z})=-F_2(z)$ for every $z \in \K$.

If $F_1$ and $F_2$ are continuous, then $F$ is called \emph{continuous}.

\end{definition}

\begin{remark} \label{rem:tilde-bar}
Given $w=q+ip \in \bH_\bC$, denote by $\overline{w}$ the element $q-ip$ of $\bH_\bC$. It is immediate to verify that a function $F:\K \lra \bH_\bC$ is a stem function if and only if it is complex--intrinsic; that is, $F(\overline{z})=\overline{F(z)}$ for every $z \in \K$ such that $\overline{z} \in \K$.
\end{remark}

As discussed in Remark \ref{remin}$(3)$, the quaternions have the following two properties, which describe their ``slice'' nature:
\begin{itemize}
 \item $\bH=\bigcup_{\jmath \in \bS}\bC_\jmath$,
 \item $\bC_\jmath \cap \bC_\kappa=\bR$ for every $\jmath,\kappa \in \bS$ with $\jmath \neq \pm\kappa$.
\end{itemize}
In this way, every $q \in \bH$ can be written as follows:
\[
q=\alpha+\jmath\beta \quad \text{for some $\alpha,\beta \in \bR$ and $\jmath \in \bS$}.
\]
If $q \in \bR$, then $\alpha=q$, $\beta=0$ and $\jmath$ is an arbitrary element of $\bS$. If $q \in \bH \setminus \bR$, then $q$ belongs to a unique ``slice complex plane'' $\bC_\jmath$
 and hence $q$ has only two expressions:
\beq \label{eq:q}
q=\alpha+\jmath\beta=\alpha+(-\jmath)(-\beta),
\eeq
where $\alpha=\mr{Re}(q)$, $\beta=\pm|\mr{Im}(q)|$ and $\jmath=\pm\mr{Im}(q)/|\mr{Im}(q)|$.

Recall the definition of circularization $\OO_\K$ of $\K$ given in (\ref{circularization}):
\[
\OO_\K:=\{\alpha+\jmath\beta \in \bH \,|\, \alpha,\beta \in \bR, \, \alpha+i\beta \in \K, \, \jmath \in \bS\}.
\]

We are now in position to define slice functions.

\begin{definition}\label{defslice}
Each stem function $F=F_1+iF_2:\K \lra \bH_\bC$ on $\K$ induces a \emph{(left) slice function} $\I(F):\OO_\K \lra \bH$ on $\OO_\K$ as follows: if $q=\alpha+\jmath\beta \in \OO_\K$ for some $\alpha,\beta \in \bR$ and $\jmath \in \bS$, then 
\[
\I(F)(q):=F_1(z)+\jmath F_2(z) \quad \text{if $z=\alpha+i\beta \in \K$}.
\]

If $F$ is continuous, then $f$ is called \emph{continuous slice functions} on $\OO_\K$. We denote by $\Sl(\OO_\K,\bH)$ the set of all continuous slice functions on $\OO_\K$.
\end{definition}

The notion of slice function $\I(F)$ just given is well--posed. Indeed, if $q \in \bR$, then $F_2(z)=0$ and hence $\I(F)(q)=F_1(q)$, independently from the choice of $\jmath$ in $\bS$. Moreover, if $q$ belongs to $\bH \setminus \bR$ and has expressions (\ref{eq:q}), then it holds:
\[
\I(F)(\alpha+(-\jmath)(-\beta))=F_1(\overline{z})+(-\jmath)F_2(\overline{z})=F_1(z)+\jmath F_2(z)=\I(F)(\alpha+\jmath\beta),
\]
where $z=\alpha+i\beta$.

\begin{remark}\label{remslices}
$(1)$ Every slice function is induced by a unique stem function. Indeed, if $f$ is induced by some stem function $F_1+iF_2$, $\jmath$ is a fixed element of $\bS$, $z=\alpha+i\beta$ is an arbitrary point of $\K$ and $q:=\alpha+\jmath\beta \in \OO_\K$, then $F_1(z)=\frac{1}{2}(f(q)+f(\overline{q}))$ and $F_2(z)=-\jmath\frac{1}{2}(f(q)-f(\overline{q}))$. It follows that $f$ satisfies the following \emph{representation formula}:
\[
\textstyle
f(\alpha+\imath\beta)=\frac{1}{2}(f(q)+f(\overline{q}))-\imath\jmath\frac{1}{2}(f(q)-f(\overline{q}))
\]
for every $\imath \in \bS$ (see Subsection 3.3 of \cite{GhPe_AIM} for details). As an immediate consequence, we infer that, if two slice functions on $\OO_\K$ coincide on $\OO_\K \cap \bC_\jmath$ for some $\jmath \in \bS$, then they coincide on the whole $\OO_\K$.

$(2)$ All continuous slice functions $f:\OO_\K \lra \bH$ are continuous in the usual topo\-logical sense (see Proposition 7(1) of \cite{GhPe_AIM}); that is, $\Sl(\OO_\K,\bH) \subset \mscr{C}(\OO_\K,\bH)$.

$(3)$ Given $n \in \bN$ and $a \in \bH$, the function $\bC \ni z \mapsto z^na=\mr{Re}(z^n)a+i\mr{Im}(z^n)a \in \bH_\bC$ is a stem function inducing the slice function $\bH \ni q \mapsto q^na \in \bH$. It follows that all polynomial functions $\sum_{h=0}^dq^ha_h$ on $\bH$ and all convergent power series $\sum_{h \in \bN}q^ha_h$ on some ball of $\bH$ are slice functions.
\end{remark}

The pointwise product of slice functions is not necessarily a slice function. On the contrary, as it is immediate to see, if $F=F_1+iF_2$ and $G=G_1+iG_2$ are stem functions, then their product 
\beq \label{eq:FG}
FG=(F_1G_1-F_2G_2)+i(F_1G_2+F_2G_1)
\eeq
is again a stem function. This suggests a natural way to define the product of slice functions.

\begin{definition} \label{sliceprod}
Given two slice functions $f=\I(F)$ and $g=\I(G)$ on $\OO_\K$, we define the \emph{slice product} $f \cdot g$ of $f$ and $g$ as the slice function $\I(FG)$ on $\OO_\K$.
\end{definition} 

\begin{remark} \label{rem:prod}
The slice product defines an operation on $\Sl(\OO_\K,\bH)$. Indeed, formula (\ref{eq:FG}) ensures that the product of two continuous stem functions is again a continuous stem function. 
\end{remark}

Let us define the $\bH$--intrinsic and $\bC_\jmath$--slice functions.

\begin{definition} \label{def:real}
Let $F=F_1+iF_2$ be a stem function on $\K$ and let $f:\OO_\K \lra \bH$ be the slice function induced by $F$. We say that $f$ is a \emph{$\bH$--intrinsic slice function} if $f(\overline{q})=\overline{f(q)}$ for every $q \in \OO_\K$. This is equivalent to require that $F_1$ and $F_2$ are real--valued (see Lemma \ref{lemstemslice} below). For this reason, we denote by $\RS(\OO_\K,\bH)$ the subset of $\Sl(\OO_\K,\bH)$ consisting of all continuous $\bH$--intrinsic slice functions on $\OO_\K$. 

Given $\jmath \in \bS$, we say that $f$ is a \emph{$\bC_\jmath$--slice function} if $F_1$ and $F_2$ are $\bC_\jmath$--valued. We denote by $\jS(\OO_\K,\bH)$ the subset of $\Sl(\OO_\K,\bH)$ consisting of all continuous $\bC_\jmath$--slice functions on $\OO_\K$.
\end{definition}

We underline that, in Definition 10 of \cite{GhPe_AIM}, $\bH$--intrinsic slice functions are called real slice functions.

Evidently, $\RS(\OO_\K,\bH) \subset \jS(\OO_\K,\bH)$ for every $\jmath \in \bS$.

$\bH$--intrinsic slice functions have nice characterizations.

\begin{lemma} \label{lemstemslice} 
Given a slice function $f=\I(F_1+iF_2):\OO_\K \lra \bH$, the following three conditions are equivalent:
\begin{itemize}
 \item[$(\mr{i})$] $f$ is $\bH$--intrinsic.
 \item[$(\mr{ii})$] $F_1$ and $F_2$ are real--valued.
 \item[$(\mr{iii})$] $f(\OO_\K \cap \bC_\jmath) \subset \bC_\jmath$ for every $\jmath \in \bS$.
\end{itemize}
\end{lemma}
\begin{proof}
The equivalence of conditions $(\mr{ii})$ and $(\mr{iii})$ is proved in Proposition~10 of \cite{GhPe_AIM}. In this way, in order to complete the proof, it suffices to show that $(\mr{i})$ is equivalent to $(\mr{ii})$. Let $F=F_1+iF_2$ be the stem function inducing $f$, let $\jmath \in \bS$, let $q=\alpha+\jmath\beta \in \OO_\K$ for some $\alpha,\beta \in \bR$ and let $z:=\alpha+i\beta \in \K$. If $(\mr{ii})$ holds, then we have:
\[
\overline{f(q)}=\overline{F_1(z)+\jmath{F_2(z)}}=F_1(z)-\jmath F_2(z)=f(\overline{q}).
\]
This proves implication $(\mr{ii}) \Longrightarrow (\mr{i})$. Assume now that $f$ is $\bH$--intrinsic. Let $F_2^0(z),F_2^1(z) \in \bR$, $\imath \in \bS$ and $p \in \OO_\K$ such that $F_2(z)=F_2^0(z)+\imath F_2^1(z)$ and $p=\alpha+\imath\beta$. Since
\[
F_1(z)-\imath F_2(z)=f(\overline{p})= \overline{f(p)}=\overline{F_1(z)}-\overline{F_2(z)}\imath,
\]
we infer that
\[
F_1(z)-\overline{F_1(z)}=\imath F_2(z)-\overline{F_2(z)}\imath=-2F^1_2(z) \in \mr{Im}(\bH) \cap \bR=\{0\}.
\]
It follows that $F_1(z) \in \bR$, $F^1_2(z)=0$ and hence $F_2(z) \in \bR$ as well.
\end{proof}

Now we introduce the notion of circular slice function.

\begin{definition} \label{def:circular-slice}
A slice function $\I(F_1+iF_2)$ on $\OO_\K$ is called \emph{circular slice function} if $F_2$ vanishes on the whole $\K$. We denote by $\Sc(\OO_\K,\bH)$ the subset of $\Sl(\OO_\K,\bH)$ consisting of all continuous circular slice functions on $\OO_\K$.
\end{definition}

\begin{lemma} \label{lemstemslice-c} 
 Given a slice function $f:\OO_\K \lra \bH$, the following two conditions are equivalent:
\begin{itemize}
 \item[$(\mr{i})$] $f$ is a circular slice function.
 \item[$(\mr{ii})$] $f(\overline{q})=f(q)$ for every $q \in \OO_\K$.
\end{itemize}
\end{lemma}
\begin{proof}
The implication $(\mr{i}) \Longrightarrow (\mr{ii})$ is evident. Let $F_1+iF_2:\K \lra \bH$ be the stem function inducing $f$. Suppose that $(\mr{ii})$ holds. Then, for every $\alpha,\beta \in \bR$ and $\imath \in \cS$ with $q:=\alpha+\imath\beta \in \OO_\K$, we have that
\[
F_1(\alpha,\beta)-\imath F_2(\alpha,\beta)=f(\overline{q})=f(q)=F_1(\alpha,\beta)+\imath F_2(\alpha,\beta)
\]
or, equivalently, $\imath F_2(\alpha,\beta)=0$. It follows immediately that $F_2$ is null and hence $f \in \Sc(\OO_\K,\bH)$.
\end{proof}

Given a basis of $\bH$, the functions in $\Sl(\OO_\K,\bH)$, and hence in $\jS(\OO_\K,\bH)$ and in $\Sl_c(\OO_\K,\bH)$, can be expressed in terms of functions in $\RS(\OO_\K,\bH)$ as follows.

\begin{lemma} \label{lem:RS}
Let $\{1,\jmath,\kappa,\delta\}$ be a basis of $\bH$. Then the map
\[
\big(\RS(\OO_\K,\bH)\big)^4 \ni (f_0,f_1,f_2,f_3) \mapsto f_0+f_1\jmath+f_2\kappa+f_3\delta \in \Sl(\OO_\K,\bH)
\]
is bijective. In particular, it follows that, given any $f \in \Sl(\OO_\K,\bH)$, there exist, and are unique, $f_0,f_1,f_2,f_3 \in \RS(\OO_\K,\bH)$ such that
\beq \label{eq:f}
f=f_0+f_1\jmath+f_2\kappa+f_3\delta.
\eeq
Moreover, it holds:
\begin{itemize}
 \item[$(\mr{a})$] If $\jmath \in \cS$, then $f$ belongs to $\jS(\OO_\K,\bH)$ if and only if $f_2=f_3=0$.
 \item[$(\mr{b})$] $f$ belongs to $\Sc(\OO_\K,\bH)$ if and only if $f_0$, $f_1$, $f_2$ and $f_3$ are real--valued.
\end{itemize}
\end{lemma}
\begin{proof}
Let $\{f_\ell=\I(F_1^\ell+iF_2^\ell)\}_{\ell=0}^3$ be $\bH$--intrinsic slice functions in $\RS(\OO_\K,\bH)$ and let $f:\OO_\K \lra \bH$ be the continuous slice function $f:=f_0+f_1\jmath+f_2\kappa+f_3\delta$. Define the function $F:\K \lra \bH_\bC$ by setting
\[
F:=(F_1^0+F_1^1\jmath+F_1^2\kappa+F_1^3\delta)+i(F_2^0+F_2^1\jmath+F_2^2\kappa+F_2^3\delta).
\]
It is immediate to verify that $F$ is the (unique) stem function inducing the slice function $f_0+f_1\jmath+f_2\kappa+f_3\delta=f$. Thanks to the uniqueness of $F$, it follows that the map mentioned in the statement is injective. Let us prove that it is also surjective. Let $g$ be a slice function in $\Sl(\OO_\K,\bH)$ induced by the stem function $G=G_1+iG_2$ and let $\{G_1^\ell,G_2^\ell\}_{\ell=0}^3$ be the real--valued continuous functions on $\K$ such that
\[
G_m=G_m^0+G_m^1\jmath+G_m^2\kappa+G_m^3\delta
\quad \mbox{ if }m \in \{1,2\}.
\]
Define the stem functions $\{G^\ell\}_{\ell=0}^3$ on $\K$ and the $\bH$--intrinsic slice functions $\{g_\ell\}_{\ell=0}^3$ in $\RS(\OO_\K,\bH)$ by setting $G^\ell:=G_1^\ell+iG_2^\ell$ and $g_\ell:=\I(G^\ell)$. Since $\I(G^1\jmath)=g_1\jmath$, $\I(G^2\kappa)=g_2\kappa$, $\I(G^3\delta)=g_3\delta$ and $G=G^0+G^1\jmath+G^2\kappa+G^3\delta$, we infer that $g=g_0+g_1\jmath+g_2\kappa+g_3\delta$, which proves the desired surjectivity.

Point $(\mr{a})$ and the fact that $f_0$, $f_1$, $f_2$ and $f_3$ are real--valued when $f \in \Sc(\OO_\K,\bH)$ are easy consequences of the above argument. Finally, suppose that $f_0$, $f_1$, $f_2$ and $f_3$ are real--valued. By Lemma \ref{lemstemslice}, we know that $f_\ell(\overline{q})=f_\ell(q)$ for every $q \in \OO_\K$ and $\ell \in \{0,1,2,3\}$. It follows that $f(\overline{q})=f(q)$ for every $q \in \OO_\K$. Now Lemma \ref{lemstemslice-c} ensures that $f \in\Sc(\OO_\K,\bH)$.
\end{proof}

\begin{remark} \label{rem:prod2}
The reader observes that, if $f,g \in \RS(\OO_\K,\bH)$, then $f \cdot g=g \cdot f \in \RS(\OO_\K,\bH)$. Indeed, if $f$ is induced by $F=F_1+iF_2$ and $g$ by $G=G_1+iG_2$ with $F_1,F_2,G_1,G_2$ real--valued, then formula (\ref{eq:FG}) ensures that also the components of $FG$ are real--valued and $FG=GF$. It is worth observing that, in this case, $f \cdot g$ coincides with the pointwise product $fg$ (see Remark 7 of \cite{GhPe_AIM} for details).

By similar considerations, we see that, if $f,g \in \Sc(\OO_\K,\bH)$, then $f \cdot g \in \Sc(\OO_\K,\bH)$ and $f \cdot g=fg$. However, in general, $f \cdot g$ is different from $g \cdot f$.

Finally, if $f,g \in \jS(\OO_\K,\bH)$, then $f \cdot g \in \jS(\OO_\K,\bH)$ and $f \cdot g=g \cdot f$. The first assertion follows again from formula (\ref{eq:FG}). The second is an immediate consequence of the representation formula for slice functions and of the fact that $f \cdot g$ and $g \cdot f$ coincide on $\OO_\K \cap \bC_\jmath$ (see Remark \ref{remslices}$(1)$). Moreover, it is immediate to verify that $f \cdot g$ coincides with $fg$ on $\OO_\K \cap \bC_\jmath$, but they can be different outside $\bC_\jmath$.
\end{remark}


\subsection{A quaternionic two-sided Banach unital $\boldsymbol{C^*}$-algebra structure on $\boldsymbol{\bH_\bC}$}
Let us introduce a structure of quaternionic two--sided Banach unital $C^*$--algebra on $\bH_\bC$ (see Section \ref{subsec:two-sided} for the definition). We need that structure later on. 

Identify $\bH$ with the subset $\{q+ip \in \bH_\bC \, | \, p=0\}$ of $\bH_\bC$ and define the quaternionic left and right multiplications $(q,w) \mapsto qw$ and $(w,q) \mapsto wq$ on $\bH_\bC$ via formula (\ref{eq:H_C}). It is immediate to verify that the set $\bH_\bC$, equipped with the usual sum, with these left and right scalar multiplications and with product (\ref{eq:H_C}), is a quaternionic two--sided algebra with unity $1$. Define an involution $w \mapsto w^*$ on $\bH_\bC$ by setting
\beq \label{eq:w^*}
w^*:=\overline{q}-i\overline{p}
\quad \mbox{if }w=q+ip.
\eeq
Observe that, given $w=q+ip$ and $y=q'+ip'$ in $\bH_\bC$, it holds:
\begin{align} \label{eq:wy^*}
(wy)^*&=(\overline{qq'-pp'})-i \, (\overline{qp'+pq'})= \nonumber\\
&=(\overline{q'} \, \overline{q}-\overline{p'} \, \overline{p})-i(\overline{p'} \, \overline{q}+
\overline{q'} \, \overline{p})=y^*w^*.
\end{align}
It follows that the involution $w \mapsto w^*$ is a $^*$--involution on $\bH_\bC$ and hence, equipping $\bH_\bC$ with such a $^*$--involution, we obtain a quaternionic two--sided unital $^*$--algebra. 

Consider the function $\|\cdot\|_{\bH_\bC}:\bH_\bC \lra \bR^+$ defined by setting
\beq \label{eq:norm-H_C} \|w\|_{\bH_\bC}:=\left(|q|^2+|p|^2+2|\mr{Im}(p \, \overline{q})|\right)^{1/2}
\quad \mbox{if }w=q+ip.
\eeq
It is worth observing that $ \|w^*\|_{\bH_\bC}= \|w\|_{\bH_\bC}$, since $|\mr{Im}(-\overline p q)|=|\mr{Im}(p\overline q )|$, and
\beq \label{eq:equiv-norms}
(|q|^2+|p|^2)^{1/2} \leq \|w\|_{\bH_\bC} \leq |q|+|p| \quad \mbox{if }w=q+ip.
\eeq

\begin{proposition} \label{prop:two-sided-H_C}
The following assertions hold.
\begin{itemize}
 \item[$(\mr{a})$] Given $w=q+ip \in \bH_\bC$, we have that $\|w\|_{\bH_\bC}=\sup_{\jmath \in \bS}|q+\jmath p|$. Moreover, if $\mr{Im}(p \, \overline{q})=\imath\beta$ for some $\imath \in \bS$ and $\beta \geq 0$, then $\|w\|_{\bH_\bC}=|q-\imath p|$.
 \item[$(\mr{b})$] The function $\|\cdot\|_{\bH_\bC}$ is a quaternionic Banach $C^*$--norm on $\bH_\bC$. In this way, the set $\bH_\bC$, equipped with the usual sum, with the product defined in (\ref{eq:H_C}), with the quaternionic left and right scalar multiplications induced by such a product, with the $^*$--involution defined in (\ref{eq:w^*}) and with the norm defined in (\ref{eq:norm-H_C}), is a quaternionic two--sided Banach $C^*$--algebra with unity\insec$1$.
\end{itemize}
\end{proposition}
\begin{proof}
$(\mr{a})$ Let $\phi:\bH_\bC \lra \bR^+$ be the function defined by setting $\phi(w):=\|w\|_{\bH_\bC}^2$. Fix $w=q+ip \in \bH_\bC$. We must prove that $\phi(w)=\sup_{\jmath \in \bS}|q+\jmath p|^2$. If $q=0$, then this is evident. Suppose that $q \neq 0$. Let $\alpha,\beta \in \bR$ and $\imath \in \bS$ such that $\beta \geq 0$ and $p \, \overline{q}=\alpha+\imath\beta$. Observe that  $\phi(w)=\left|q-\imath p\right|^2$. Indeed, $\beta=|\mr{Im}(p \, \overline{q})|$ and it holds:
\begin{align*}
\left|q-\imath p\right|^2&=|q|^{-2}\left|(q-\imath p)\bar q\right|^2=|q|^{-2}\left((|q|^2+\beta)^2+\alpha^2\right)=\\
&=|q|^{-2}\left(|q|^4+|q|^2|p|^2+2\beta|q|^2\right)=|q|^2+|p|^2+2\beta=\phi(w).
\end{align*}
Moreover, for every $\jmath \in \cS$, we have:
\begin{align*}
\left|q+\jmath p\right|^2 &=|q|^2+|p|^2+2 \, \mr{Re}(\jmath p \, \overline{q})=|q|^2+|p|^2+2 \, \mr{Re}(\jmath(\alpha+\imath\beta))=\\
&=|q|^2+|p|^2+2 \, \mr{Re}(\jmath\imath)\beta \leq |q|^2+|p|^2+2\beta=\phi(w).
\end{align*}
It follows that $\phi(w)=\sup_{\jmath \in \bS}|q+\jmath p|^2=|q-\imath p|^2$, as desired.

$(\mr{b})$ Firstly, observe that formula (\ref{eq:norm-H_C}) implies at once that $\|w\|_{\bH_\bC}=0$ if and only if $w=0$. The triangular inequality for $\|\cdot\|_{\bH_\bC}$ is an immediate consequence of point\insec$(\mr{a})$. It is evident that $\|1\|_{\bH_\bC}=1$. It is also easy to see that $\|w^*w\|_{\bH_\bC}=\|w\|_{\bH_\bC}^2$. Indeed, we have that $ww^*=|q|^2+|p|^2-2i\mr{Im}(p\,\overline  q)$  and hence $(\mr{a})$ ensures that
\begin{align*}
\|ww^*\|_{\bH_\bC} &=\textstyle \sup_{\jmath \in \bS}\big| |q|^2+|p|^2+2\jmath\mr{Im}( p \, \overline q)\big|=\\
&=|q|^2+|p|^2+2|\mr{Im}( p \, \overline q)|=\|w\|_{\bH_\bC}^2,
\end{align*}
and therefore also $\|w^*w\|_{\bH_\bC} =\|w\|_{\bH_\bC}^2$.

Choose $q' \in \bH$. It hold:
\begin{align*}
\phi(wq') &=|qq'|^2+|pq'|^2+2|\mr{Im}(p q'\overline{q'} \, \overline{q})|=\\
&=|q|^2|q'|^2+|p|^2|q'|^2+2|\mr{Im}(p \, \overline{q})| \, |q'|^2=\phi(w)|q'|^2
\end{align*}
and
\begin{align*}
\phi(q'w) &=|q'q|^2+|q'p|^2+2|\mr{Im}(q'p \, \overline{q} \, \overline{q'})|=\\
&=|q'|^2|q|^2+|q'|^2|p|^2+2|q'\mr{Im}(p \, \overline{q})\overline{q'}|=\\
&=|q'|^2|q|^2+|q'|^2|p|^2+2|q'|^2|\mr{Im}(p \, \overline{q})|=|q'|^2\phi(w).
\end{align*}
This proves that $\|wq'\|_{\bH_\bC}=|q'|\|w\|_{\bH_\bC}=\|q'w\|_{\bH_\bC}$. In particular, it follows that $\|\cdot\|_{\bH_\bC}$ is a norm on $\bH_\bC \simeq \bR^4 \times \bR^4$ in the usual sense, and hence it is equivalent to the euclidean one. This ensures that the metric on $\bH_\bC$ induced by such a norm is complete.

Let $p' \in \bH$ and let $y:=q'+ip'$. In order to complete the proof, it remains to show that $\|wy\|_{\bH_\bC} \leq \|w\|_{\bH_\bC}\|y\|_{\bH_\bC}$ or, equivalently, $\phi(wy) \leq \phi(w)\phi(y)$. Bearing in mind that
$\phi(w)=|q|^2+|p|^2+2|\mr{Im}(p \, \overline{q})|$, $\phi(y)=|q'|^2+|p'|^2+2|\mr{Im}(p' \, \overline{q'})|$ and
\begin{align*}
\phi(wy)&=\phi((qq'-pp')+i(qp'+pq'))=\\
&=|qq'-pp'|^2+|qp'+pq'|^2+2\big|\mr{Im}((qp'+pq')(\overline{q'} \, \overline{q}-\overline{p'} \, \overline{p}))\big|,
\end{align*}
an explicit computation gives that
\[
\phi(w)\phi(y)-\phi(wy)=2(a-b)+4|\mr{Im}(p \, \overline{q}) \mr{Im}(p' \, \overline{q'})|-4 \, \mr{Re}(q \, \mr{Im}(p' \, \overline{q'}) \, \overline{p}),
\]
where
\begin{align*}
a&:=|q|^2|\mr{Im}(p' \, \overline{q'})|+|p|^2|\mr{Im}(p' \, \overline{q'})| +|q'|^2|\mr{Im}(p \, \overline{q})|+|p'|^2|\mr{Im}(p \, \overline{q})|,\\
b&:=\big| q \,\mr{Im}(p' \, \overline{q'}) \overline{q} \, +p \, \mr{Im}(p' \, \overline{q'}) \overline{p} +|q'|^2 \, \mr{Im}(p \, \overline{q})+|p'|^2 \, \mr{Im}(p \, \overline{q}) \big|
\end{align*}
The triangular inequality implies that $b \leq a$ and hence
\[
\phi(w)\phi(y)-\phi(wy) \geq 4|\mr{Im}(p \, \overline{q}) \mr{Im}(p' \, \overline{q'})|-4 \, \mr{Re}(q \, \mr{Im}(p' \, \overline{q'}) \, \overline{p}).
\]
Let $\imath' \in \bS$ and $r \in \bR$ such that $\mr{Im}(p' \, \overline{q'})=\imath' r$. Choose $\jmath' \in \bS$ in such a way that $\{1,\imath',\jmath',\imath'\jmath'\}$ is an orthonormal basis of $\bH$. Write $\overline p q=x_0+i'x_1+j'x_2+\imath' \jmath' x_3$, where $x_0,x_1,x_2,x_3\in\bR$. By a direct computation, we obtain:
\begin{align*}
&|\mr{Im}(p \, \overline{q}) \mr{Im}(p' \, \overline{q'})|^2-\big(\mr{Re}(q \, \mr{Im}(p' \, \overline{q'}) \, \overline{p})\big)^2=r^2(|\mr{Im}(\overline p\,q)|^2-\mr{Re}( \imath' \overline{p}\,q)^2)=\\
&=r^2\left((x_1^2+x_2^2+x_3^2) -x_1^2\right) \geq 0.
\end{align*}
It follows that $\phi(w)\phi(y)-\phi(wy) \geq 0$.
\end{proof}


\subsection{$\boldsymbol{C^*}$-algebra structures on slice functions} 

The aim of this section is to define a structure of quaternionic two--sided Banach unital $C^*$--algebra on $\Sl(\OO_\K,\bH)$. As a consequence, $\Sc(\OO_\K,\bH)$ will turn out to be a quaternionic two--sided Banach unital $C^*$--subalgebra of $\Sl(\OO_\K,\bH)$. Furthermore, restricting the scalars to $\bR$ and to a fixed $\bC_\jmath$, $\RS(\OO_\K,\bH)$ will be a commutative real Banach unital $C^*$--subalgebra of $\Sl(\OO_\K,\bH)$ and $\jS(\OO_\K,\bH)$ will be a commutative $\bC_\jmath$--Banach unital $C^*$--subalgebra of $\Sl(\OO_\K,\bH)$. We refer again the reader to Section\insec\ref{subsec:two-sided} for the definitions of this kind of algebraic structures. 

Let $f=\I(F)$ and $g=\I(G)$ be slice functions in $\Sl(\OO_\K,\bH)$. The usual pointwise sum $f+g$ of $f$ and $g$ defines an element in $\Sl(\OO_\K,\bH)$, because $F+G$ is continuous and $f+g=\I(F)+\I(G)=\I(F+G)$. Equip $\Sl(\OO_\K,\bH)$ with the slice product. As observed in Remark \ref{rem:prod2}, the slice product of functions in $\RS(\OO_\K,\bH)$ is again a function in $\RS(\OO_\K,\bH)$ and coincides with the pointwise product. The same is true for $\Sc(\OO_\K,\bH)$. The slice product defines an operation on $\jS(\OO_\K,\bH)$ as well. Furthermore, the slice product on $\RS(\OO_\K,\bH)$ and on $\jS(\OO_\K,\bH)$ is commutative.

Let us define quaternionic left and right scalar multiplications $(q,f) \mapsto q \cdot f$ and $(f,q) \mapsto f \cdot q$ on $\Sl(\OO_\K,\bH)$. Given $q \in \bH$, denote by $c_q:\OO_\K \lra \bH$ the function constantly equal to $q$. Such a function belongs to $\Sl(\OO_\K,\bH)$. Indeed, it is the slice function induced by the stem function  on $\K$ constantly equal to $q$. For convenience, denote $c_1$ also by the symbol $1_{\OO_\K}$. Define:

\beq \label{eq:scalar-slice}
q \cdot f:=c_q \cdot f
\quad \mbox{and} \quad
f \cdot q:=f \cdot c_q,
\eeq
where $c_q \cdot f$ and $f \cdot c_q$ are slice products. It is easy to see that $f \cdot q$ coincides with the pointwise scalar multiplication $fq$ for every $q \in \bH$. If $q \in \bR$ or $f \in \Sc(\OO_\K,\bH)$, then also $q \cdot f$ is equal to the pointwise scalar multiplication $qf$. Otherwise, $qf$ is not, in general,  a slice function and hence is different from $q \cdot f$. It is immediate to see that the set $\Sl(\OO_\K,\bH)$, equipped with the pointwise sum and with the left and right scalar multiplications (\ref{eq:scalar-slice}) is a quaternionic two--sided algebra with unity $1_{\OO_\K}$.

Let us introduce an involution $f \mapsto f^*$ on $\Sl(\OO_\K,\bH)$. Fix a stem function $F=F_1+iF_2$ on $\K$ and define $f:=\I(F)$. Denote by $F^*:\K \lra \bH_\bC$ the stem function sending $z$ into $(F(z))^*$, where $(F(z))^*$ is defined as in (\ref{eq:w^*}). More explicitly, we have that $F^*=\overline{F_1}-i\overline{F_2}$, where $\overline{F_h}:\K \lra \bH$ is defined by setting $\overline{F_h}(z):=\overline{F_h(z)}$ with $h \in \{1,2\}$. Define:
\beq \label{eq:f^*}
f^*:=\I(F^*).
\eeq
Let $g$ be another function in $\Sl(\OO_\K,\bH)$. By (\ref{eq:wy^*}), we infer that
\[
(f \cdot g)^*=g^* \cdot f^*.
\]
Given $q \in \bH$, it is easy to verify that $(c_q)^*=c_{\overline{q}}$. It follows that
\[
(q \cdot f)^*=f^* \cdot \overline{q}
\quad \mbox{and} \quad
(f \cdot q)^*=\overline{q} \cdot f^*.
\]
Since it is evident that $(f+g)^*=f^*+g^*$, the involution defined in (\ref{eq:f^*}) is a $^*$--involution of $\Sl(\OO_\K,\bH)$.

The reader observes that, if $f \in \RS(\OO_\K,\bH)$ or $f \in \Sc(\OO_\K,\bH)$, then $f^* \in \RS(\OO_\K,\bH)$ or $f^* \in \Sc(\OO_\K,\bH)$, respectively. Moreover, in these cases, we have that $f^*=\overline{f}$, where $\overline{f}:\OO_\K \lra \bH$ denotes the conjugated function sending $q$ into $\overline{f(q)}$. If $f \in \jS(\OO_\K,\bH)$, then $\overline{F_1}$ and $\overline{F_2}$ are $\bC_\jmath$--valued and hence $f^* \in \jS(\OO_\K,\bH)$.

Consider now the supremum norm $\|\cdot\|_\infty:\Sl(\OO_\K,\bH) \lra \R^+$ defined by setting
\beq \label{eq:slice-norm}
\|f\|_{\infty}:=\sup_{q \in \OO_\K}|f(q)|.
\eeq
Thanks to point $(\mr{a})$ of Proposition \ref{prop:two-sided-H_C}, we infer that
\beq \label{eq:|f|}
\|f\|_{\infty}=\sup_{z \in \K}\|F(z)\|_{\bH_\bC}.
\eeq
The norm $\|\cdot\|_{\infty}$ is a quaternionic Banach $C^*$--norm on $\Sl(\OO_\K,\bH)$. To see this, we must prove that the set $\Sl(\OO_\K,\bH)$, equipped with the distance induced by  the norm $\|\cdot\|_\infty$, is a complete metric space. Let $\{f_n=\I(F_n)\}_{n \in \bN}$ be a Cauchy sequence in $\Sl(\OO_\K,\bH)$. Thanks to (\ref{eq:|f|}), we have that $\{F_n\}_{n \in \bN}$ is a Cauchy sequence in $\mscr{C}(\K,\bH_\bC)$, where $\bH_\bC$ is equipped with the norm $\|\cdot\|_{\bH_\bC}$ and $\mscr{C}(\K,\bH_\bC)$ with the corresponding supremum norm. Since $\mscr{C}(\K,\bH_\bC)$ is a (real) Banach space, $\{F_n\}_{n \in \bN}$ converges to some continuous stem function $F$ on $\K$. Using (\ref{eq:|f|}) again, we infer that $\{f_n\}_{n \in \bN} \to \I(F)$ in $\Sl(\OO_\K,\bH)$. This shows the completeness of $\Sl(\OO_\K,\bH)$.

The above discussion proves the following basic result.

\begin{theorem} \label{thm:two-sided-C^*}
The following assertions hold:
\begin{itemize}
 \item[$(\mr{a})$] The set $\Sl(\OO_\K,\bH)$, equipped with the pointwise sum, with the left and right scalar multiplications defined in (\ref{eq:scalar-slice}), with the slice product, with the $^*$--involution defined in (\ref{eq:f^*}) and with the supremum norm defined in (\ref{eq:slice-norm}), is a quaternionic two--sided Banach $C^*$--algebra with unity $1_{\OO_\K}$.
 \item[$(\mr{b})$] The set $\RS(\OO_\K,\bH)$ is a commutative real Banach unital $C^*$--subalgebra of $\Sl(\OO_\K,\bH)$.
 \item[$(\mr{c})$] Given $\jmath \in \bS$, the set $\jS(\OO_\K,\bH)$ is a commutative $\bC_\jmath$--Banach unital $C^*$--subalgebra of $\Sl(\OO_\K,\bH)$.
 \item[$(\mr{d})$] The set $\Sc(\OO_\K,\bH)$ is a quaternionic two--sided Banach unital $C^*$--subalgebra of $\Sl(\OO_\K,\bH)$.
\end{itemize}
\end{theorem}

\begin{remark} \label{rem:sectional}
Thanks to point $(\mr{a})$ of Proposition \ref{prop:two-sided-H_C}, it is immediate to verify that $\|f\|_\infty=\sup_{q \in \OO_\K \cap \bC_\jmath}|f(q)|$ for every $f \in \jS(\OO_\K,\bH)$.
\end{remark}


\subsection{$\boldsymbol{\bH}$-intrinsic polynomial density}
Recall that $\bR[X,Y]$ denotes the ring of real polynomials in the indeterminates $X$ and $Y$, \textit{with coefficients on the right}.

Let us introduce a particular class of $\bH$--intrinsic slice functions (see Definition \ref{def:real}).

\begin{definition}
Let $g:\OO_\K \lra \bH$ be a slice function induced by the stem function $G=G_1+iG_2:\K \lra \bH_\bC$. We say that $g$ is a \emph{polynomial $\bH$--intrinsic slice function} on $\OO_\K$ if there exist polynomials $Q_1$ and $Q_2$ in $\bR[X,Y]$ such that $G_1(z)=Q_1(\alpha,\beta)$ and $G_2(z)=Q_2(\alpha,\beta)$ for every $z=\alpha+i\beta \in \K$. We denote by $\PRS(\OO_\K,\bH)$ the subset of $\RS(\OO_\K,\bH)$ consisting of all polynomial $\bH$--intrinsic slice functions on $\OO_\K$.
\end{definition}

\begin{remark} \label{rem:PRS}
$(1)$ In the preceding definition, it is always possible to assume that $Q_1$ is even in $Y$ and $Q_2$ in odd in $Y$; that is, $Q_1(X,-Y)=Q_1(X,Y)$ and $Q_2(X,-Y)=-Q_2(X,Y)$ in $\bR[X,Y]$.

$(2)$  Bearing in mind points $(1)$ and $(3)$ of Remark \ref{remslices}, one can easily prove that, if there exists a polynomial $P \in \bR[X]$ such that $g(q)=P(q)$ on $\OO_\K$, then $g \in \PRS(\OO_\K,\bH)$. However, in general, there exist functions in $\PRS(\OO_\K,\bH)$, which cannot be expressed in this way. If $\K$ has an interior point in $\bC$, an example is given by the polynomial $\bH$--intrinsic slice functions sending $q \in \OO_\K$ into $\overline{q} \in \bH$.

$(3)$ It is easy to verify that $\PRS(\OO_\K,\bH)$ is a commutative real Banach unital $C^*$--subalge\-bra of $\RS(\OO_\K,\bH)$. 
\end{remark}

We now prove a version of the Weierstrass approximation theorem for functions in $\RS(\OO_\K,\bH)$.

\begin{proposition} \label{prop:Weierstrass}
If $\K$ is a non--empty compact subset of $\bC$ invariant under complex conjugation, then  $\PRS(\OO_\K,\bH)$ is dense in $\RS(\OO_\K,\bH)$.
\end{proposition}
\begin{proof}
Let $f \in \RS(\OO_\K,\bH)$ induced by the stem function $F=F_1+iF_2$ and let $\vep$ be a positive real number. By the Weierstrass approximation theorem, there exist polynomials $P_1,P_2 \in \R[X,Y]$ such that 
\[
\sup_{\alpha+i\beta \in \K}|F_1(\alpha+i\beta)-P_1(\alpha,\beta)|<\frac{\vep}{2}
\]
and
\[
\sup_{\alpha+i\beta \in \K}|F_2(\alpha+i\beta)-P_2(\alpha,\beta)|<\frac{\vep}{2}.
\]
Define $Q_1,Q_2 \in \R[X,Y]$ by setting
\[
Q_1(X,Y):=\frac{P_1(X,Y)+P_1(X,-Y)}{2}
\]
and
\[
Q_2(X,Y):=\frac{P_2(X,Y)-P_2(X,-Y)}{2}.
\]
Evidently, $Q_1$ is even in $Y$ and $Q_2$ is odd in $Y$. Moreover, given any $z=\alpha+i\beta \in \K$, it holds:
\begin{align*}
|F_1(z)-Q_1(\alpha,\beta)|&=  \left|\frac{F_1(z)+F_1(\overline{z})}{2}-\frac{P_1(\alpha,\beta)+P_1(\alpha,-\beta)}{2}\right|=\\
&=\left|\frac{F_1(z)-P_1(\alpha,\beta)}{2}+\frac{F_1(\overline{z})-P_1(\alpha,-\beta)}{2}\right| \leq\\
&\leq  \left|\frac{F_1(z)-P_1(\alpha,\beta)}{2}\right|+\left|\frac{F_1(\overline{z})-P_1(\alpha,-\beta)}{2}\right|<\\
&<
\frac{\vep}{4}+\frac{\vep}{4}=\frac{\vep}{2}.
\end{align*}
Similarly, we have that $|F_2(z)-Q_2(\alpha,\beta)|<\vep/2$. Denote by $g$ the function in $\PRS(\OO_\K,\bH)$ induced by the stem function $G=G_1+iG_2$ on $\K$ defined by setting
\[
G_1(z):=Q_1(\alpha,\beta)
\quad \mbox{and} \quad
G_2(z):=Q_2(\alpha,\beta) \quad \mbox{if }z=\alpha+i\beta \in \K.
\]
Thanks to equality (\ref{eq:|f|}) and to the second inequality of (\ref{eq:equiv-norms}), we obtain:
\begin{align*}
\|f-g\|_\infty &=\sup_{z \in \K}\|F(z)-G(z)\|_{\bH_\bC} \leq \\
&\leq \sup_{z \in \K}(|F_1(z)-G_1(z)|+|F_2(z)-G_2(z)|) < \frac{\vep}{2}+\frac{\vep}{2}=\vep.
\end{align*}
The proof is complete.
\end{proof}


\section{Continuous slice functional calculus for normal operators}

\textit{Throughout this final part, $\sH$ will denote a fixed quaternionic Hilbert space and $T$ will be a fixed normal operator in $\gB(\sH)$.} We remind the reader that $\gB(\sH)$ indicates the set of all (bounded right $\bH$--linear) operators of $\sH$.


\subsection{Polynomial $\boldsymbol{\bH}$-intrinsic slice functions of a normal operator $\boldsymbol T$}

Define the self--adjoint operator $A \in \gB(\sH)$ and the positive operator $B \in \gB(\sH)$ by setting
\[
A:=(T+T^*)\frac{1}{2}
\quad \mbox{and} \quad
B:=|T-T^*|\frac{1}{2}.
\]
Theorem \ref{teoext} ensures the existence of an anti self--adjoint and unitary operator $J \in \gB(\sH)$ such that
\begin{align}
& T=A+JB, \nonumber\\
& \text{$A$, $B$ and $J$ commute mutually}, \nonumber\\
\label{eq:uniqueness-J}
&\text{$J$ is uniquely determined by $T$ on $\mi{Ker}(T-T^*)^\perp$}.
\end{align}

The reader observes that $T^*=A-JB$ and hence $B=-J(T-T^*)\frac{1}{2}$.

Denote by $\K$ the non--empty compact subset of $\bC$, invariant under complex conjugation, such that
\[
\OO_\K=\ssp(T).
\]

\begin{definition}\label{defpolT}
Let $g \in \PRS(\ssp(T),\bH)$ be a polynomial $\bH$--intrinsic slice function on $\ssp(T)$ and let $G:\K \lra \bH_\bC$ be the stem function inducing $g$. Choose polynomials $Q_1,Q_2 \in \bR[X,Y]$ such that $G(\alpha+i\beta)=Q_1(\alpha,\beta)+iQ_2(\alpha,\beta)$ for every $\alpha+i\beta \in \K$. We define the operator $g(T) \in \gB(\sH)$ by setting
\beq \label{gT}
g(T):=Q_1(A,B)+JQ_2(A,B).
\eeq
\end{definition}

The definition of $g(T)$ just given is consistent as we see in the next result.

\begin{lemma} \label{lem:unique}
The definition of $g(T)$ given in (\ref{gT}) depends only on $g$ and $T$, not on the operator $J$ and on the polynomials $Q_1$ and $Q_2$ we chose.
\end{lemma}
\begin{proof}
By Corollary \ref{cor:well-def-pol}, the operators $Q_1(A,B)$ and $Q_2(A,B)$ depend only on $g$ and $T$. It follows that the operator $Q_1(A,B)+JQ_2(A,B)$ depends only on $g$, $T$ and $J$. Let us prove that it is independent from $J$. By Remark \ref{rem:PRS}$(1)$, we may suppose that $Q_2$ is odd in $Y$. In this way, we have that $Q_2(A,B)=CB$ for some $C \in \gB(\sH)$ and hence
\[
\mi{Ker}(T-T^*)=\mi{Ker}(B) \subset \mi{Ker}(Q_2(A.B)).
\]
It follows that the operator $JQ_2(A,B)$ vanishes on $\mi{Ker}(T-T^*)$ and hence it is uniquely determined by $T$ on $\mi{Ker}(T-T^*)$. On the other hand, by (\ref{eq:uniqueness-J}), the operator $JQ_2(A,B)=Q_2(A,B)J$ is uniquely determined by $T$ on $\mi{Ker}(T-T^*)^\perp$ as well. This completes the proof.
\end{proof}

\begin{proposition} \label{propcentrale}
 For every $g \in \PRS(\ssp(T),\bH)$, $g(T)$ is a normal operator in $\gB(\sH)$, which commutes with $J$ and satisfies the following equalities: 
\beq \label{spectralmapPR}
\|g(T)\|=\|g\|_\infty
\eeq
and
\beq \label{spectralmapPR2}
\ssp(g(T))=g(\ssp(T)).
\eeq
\end{proposition}
\begin{proof}
We begin as in the proof of Corollary \ref{cor:well-def-pol}. Since $A^*=A$, $B^*=B$, $J^*=-J$ and $A$, $B$ and $J$ commute mutually, it is immediate to verify that $g(T)$ is a normal operator in $\gB(\sH)$ and $J$ commutes with $Q_1(A,B)$, $Q_2(A,B)$ and hence with $g(T)$.

Fix $\imath \in \bS$ and identify $\bC$ with $\bC_\imath$ via the real isomorphism $\alpha+i\beta \mapsto \alpha+\imath\beta$. Apply the spectral radius formula (see (\ref{raggiospett})) and Proposition \ref{propinterssigma}$(\mr{c})$ to $g(T)$. We obtain that 
\beq \label{eq:ggTT}
\|g(T)\|=\sup\left\{|q| \in \bR^+ \left| \, q \in  \sigma\big(g(T)|_{\sH^{J\imath}_+}\big) \right.\right\},
\eeq
where $g(T)|_{\sH^{J\imath}_+}$ denotes the operator in $\gB(\sH^{J\imath}_+)$ obtained restricting $g(T)$ to $\sH^{J\imath}_+$. Thanks to Proposition \ref{propinterssigma}, we know that $A(\sH^{J\imath}_+) \subset \sH^{J\imath}_+$ and $B(\sH^{J\imath}_+) \subset \sH^{J\imath}_+$. Denote by $A|_{\sH^{J\imath}_+}$ and $B|_{\sH^{J\imath}_+}$ the operators in $\gB(\sH^{J\imath}_+)$ obtained restricting $A$ and $B$ to $\sH^{J\imath}_+$, respectively. Bearing in mind that $A=(T+T^*)\frac{1}{2}$, $B=-J(T-T^*)\frac{1}{2}$ and $Ju=u\imath$ for every $u \in \sH^{J\imath}_+$, we infer that $A|_{\sH^{J\imath}_+}=\big(T|_{\sH^{J\imath}_+}+(T|_{\sH^{J\imath}_+})^*\big)\frac{1}{2}$,  $B|_{\sH^{J\imath}_+}=\big(T|_{\sH^{J\imath}_+}-(T|_{\sH^{J\imath}_+})^*\big)\frac{1}{2\imath}$ and
\beq \label{eq:gT|}
g(T)|_{\sH^{J\imath}_+}=Q_1\big(A|_{\sH^{J\imath}_+},B|_{\sH^{J\imath}_+}\big)+Q_2\big(A|_{\sH^{J\imath}_+},B|_{\sH^{J\imath}_+}\big) \, \imath.
\eeq

Let $G$ be the stem function on $\K$ inducing $g$ and let $Q_1,Q_2 \in \bR[X,Y]$ be polynomials such that  $G(\alpha+i\beta)=Q_1(\alpha,\beta)+iQ_2(\alpha,\beta)$ for every $\alpha+i\beta \in \K$. Consider $G$ as a continuous function from $\K \subset \bC_\imath$ to $\bC_\imath$ by setting $G(\alpha+\imath\beta)=g(\alpha+\imath\beta)=Q_1(\alpha)+Q_2(\alpha,\beta) \, \imath$. Combining (\ref{eq:gT|}) with the continuous functional calculus theorem for complex normal operators (see \cite\cite{RudinARC,Analysisnow,Moretti}}), we see at once that $g(T)|_{\sH^{J\imath}_+}=G(T|_{\sH^{J\imath}_+})$ and
\beq \label{eq:sigma}
\sigma(g(T)|_{\sH^{J\imath}_+})=G(\sigma(T|_{\sH^{J\imath}_+})).
\eeq
Thanks (\ref{eq:ggTT}), (\ref{eq:sigma}) and Proposition \ref{propinterssigma}$(\mr{c})$, we infer that
\begin{align*}\label{est1}
\|g(T)\|&=
\sup \left\{|G(\alpha+\imath\beta)| \in \R^+ \left| \,  \alpha+\imath\beta \in \sigma\big(T|_{\sH^{J\imath}_+}\big)\right.\right\}=\\
&=\sup \left\{\left. \left(|Q_1(\alpha,\beta)|^2+|Q_2(\alpha,\beta)|^2\right)^{1/2} \in \bR^+ \right| \, \alpha+\imath\beta \in \ssp(T) \cap \bC_\imath \right\}=\\
&=\sup \{|g(q)| \in \bR^+ \, | \, q \in \ssp(T)\}=\|g\|_\infty.
\end{align*}
Using (\ref{eq:sigma}) and Proposition \ref{propinterssigma}$(\mr{c})$ again, together with Lemma~\ref{lemstemslice}$(\mr{ii})$ and with the fact that $G$ is the stem function inducing $g$, we obtain:
\begin{align*}
\ssp(g(T)) \cap \bC_\imath &=\sigma(g(T)|_{\sH^{J\imath}_+}) \cup  \overline{\sigma(g(T)|_{\sH^{J\imath}_+})}=G(\sigma(T|_{\sH^{J\imath}_+}))\cup  \overline{G(\sigma(T|_{\sH^{J\imath}_+}))}=\\
&=G(\sigma(T|_{\sH^{J\imath}_+})) \cup G(\overline{\sigma(T|_{\sH^{J\imath}_+})})=G\left( \sigma(T|_{\sH^{J\imath}_+}) \cup \overline{\sigma(T|_{\sH^{J\imath}_+})}\right)=\\
&=g\left(\ssp(T) \cap \bC_\imath \right)=g(\ssp(T)) \cap \bC_\imath.
\end{align*}
Since $\ssp(g(T))$ and $g(\ssp(T))$ are both circular subsets of $\bH$, it follows that they are equal.
\end{proof}


\subsection{Continuous $\boldsymbol{\bH}$-intrinsic slice functions of a normal operator $\boldsymbol T$}
\textit{In this section, we assume that the set $\RS(\ssp(T),\bH)$ is equipped with the commutative real Banach unital $C^*$--algebra structure given in Theorem \ref{thm:two-sided-C^*}(b).}

\textit{Furthermore, we consider $\gB(\sH)$ as a real Banach  $C^*$--algebra with unity $\1$}; that is, $\gB(\sH)$ is equipped with the pointwise sum, with the real scalar multiplication defined in $(\ref{eq:rT})$, with the composition as product, with the adjunction $T \mapsto T^*$ as $^*$--involution and with the norm defined in $(\ref{QN})$.

\begin{theorem}\label{teofinale1}
There exists, and is unique, a continuous $^*$--homomorphism 
\[
\Psi_{\bR,T}:\RS(\ssp(T),\bH) \ni f \mapsto f(T) \in  \gB(\sH)
\]
of real Banach unital $C^*$--algebras such that:
\begin{itemize}
 \item[$(\mr{i})$] $\Psi_{\bR,T}$ is unity--preserving; that is, $\Psi_{\bR,T}(1_{\ssp(T)})=\1$.
 \item[$(\mr{ii})$] $\Psi_{\bR,T}(\mi{id})=T$, where $\mi{id}:\ssp(T) \hookrightarrow \bH$ denotes the inclusion map.
\end{itemize}
The following further facts hold true.
\begin{itemize}
\item[$(\mr{a})$] If $f \in \RS(\ssp(T),\bH)$ and $J \in \gB(\sH)$ is an anti self--adjoint and unitary operator satisfying (\ref{OMEGA}) and (\ref{decA0-bis}), then $\Psi_{\bR,T}(f)$ is normal and commutes with $J$.
\item[$(\mr{b})$] $\Psi_{\bR,T}$ is isometric; that is, $\|f(T)\|=\|f\|_\infty$ for every $f \in \RS(\ssp(T),\bH)$.
\item[$(\mr{c})$] For every $f \in \RS(\ssp(T),\bH)$, the following \emph{continuous $\bH$--intrinsic slice spectral map property} holds:  
\[
\ssp(f(T))=f(\ssp(T)).
\]
\end{itemize}
\end{theorem}
\begin{proof}
Uniqueness of $\Psi_{\bR,T}$ is simply proved. Suppose that $\Psi_{\bR,T}$ exists. If $\Psi:\RS(\ssp(T),\bH) \lra \gB(\sH)$ is another continuous $^*$--homomorphism satisfying $(\mr{i})$ and $(\mr{ii})$, then it coincides with 
$\Psi_{\bR,T}$ on $\PRS(\ssp(T),\bH)$. On the other hand, thanks to Proposition~\ref{prop:Weierstrass}, the set $\PRS(\ssp(T),\bH)$ is dense in $\RS(\ssp(T),\bH)$ and hence, by continuity, $\Psi$ and $\Psi_{\bR,T}$ coincide everywhere on $\RS(\ssp(T),\bH)$.

Let us pass to prove the existence of $\Psi_{\bR,T}$. First, define the map
\[
\psi_{\bR,T}:\PRS(\ssp(T),\bH) \lra \gB(\sH)
\]
by setting $\Psi_{\bR,T}(g):=g(T)$ as in Definition~\ref{defpolT}. It is immediate to verify that $\psi_{\bR,T}$ is a unity--preserving $^*$--homomorphism of real Banach unital $^*$--algebras sending $\mi{id}$ into $T$. Moreover, thanks to Proposition \ref{propcentrale}, $\psi_{\bR,T}$ satisfies conditions $(\mr{a})$, $(\mr{b})$ and $(\mr{c})$ with ``$\Psi_{\bR,T}$'' and ``$f \in \RS(\ssp(T),\bR)$'' replaced by ``$\psi_{\bR,T}$'' and ``$g \in \PRS(\ssp(T),\bR)$'', respectively. In particular, $\psi_{\bR,T}$ is isometric and hence, thanks to the density of $\PRS(\ssp(T),\bH)$ in $\RS(\ssp(T),\bH)$, it admits a unique extension $\Psi_{\bR,T}$ defined on the whole $\RS(\ssp(T,\bH))$. Evidently, by continuity, $\Psi_{\bR,T}$ is a $^*$--homomorphism of real Banach unital $C^*$--algebras, satisfying $(\mr{a})$ and $(\mr{b})$. 

Let us show that $\Psi_{\bR,T}$ verifies $(\mr{c})$. Fix $f \in \RS(\ssp(T),\bH)$. First, we prove that $f(\ssp(T)) \subset \ssp(f(T))$. Choose $q \in \ssp(T)$ and a sequence $\{g_n\}_{n \in \bN}$ in $\PRS(\ssp(T),\bH)$ converging to $f$. Observe that $\{g_n(q)\}_{n \in \bN} \to f(q)$ in $\bH$, $\{g_n(T)\}_{n \in \bN} \to f(T)$ in $\gB(\sH)$ and hence $\{\Delta_{g_n(q)}(g_n(T))\}_{n \in \bN} \to \Delta_{f(q)}(f(T))$ in $\gB(\sH)$. Recall that the set $\mscr{S}$ of operators in $\gB(\sH)$, which do not admit a two--sided inverse in $\gB(\sH)$, is closed in $\gB(\sH)$ (see Proposition \ref{propcompleteness}$(\mr{c})$). By (\ref{spectralmapPR2}), each operator $\Delta_{g_n(q)}(g_n(T))$ belongs to $\mscr{S}$. It follows that $\Delta_{f(q)}(f(T)) \in \mscr{S}$ or, equivalently, $f(q) \in \ssp(f(T))$.

It remains to prove that $\ssp(f(T)) \subset f(\ssp(T))$. Let $p \not\in f(\ssp(T))$. We must show that $p \not\in \ssp(f(T))$; that is, $\Delta_p(f(T))$ has a two--sided inverse in $\gB(\sH)$. Let $\Delta_p:\bH \lra \bH$ be the polynomial real function sending $q$ into $q^2-q(p+\overline{p})+|p|^2$. The continuous $\bH$--intrinsic slice function $f$ on $\ssp(T)$ and $\Delta_p$ can be composed, and the composition is still a continuous $\bH$--intrinsic slice function $\Delta_pf:\ssp(T) \lra \bH$. Indeed, if $F$ is the stem function inducing $f$, then $\Delta_pf$ is induced by the stem function $\Delta_pF:=F^2-F(p+\overline{p})+|p|^2$. The function $\Delta_pf$ is nowhere zero. Let us prove this assertion. On the contrary, suppose that there exists $y \in \ssp(T)$ such that $\Delta_pf(y)=0$. This is equivalent to say that $f(y) \in \cS_p$. Since $f$ is a $\bH$--intrinsic slice function, it would follow that $f(\cS_y)=\cS_p$ and hence, it being $\cS_y \subset \ssp(T)$, we would infer that $p \in f(\ssp(T))$, which contradicts our hypothesis. Thanks to the fact that $\Delta_pf$ is nowhere zero, we can define the function $f':\ssp(T) \lra \bH$ by setting $f'(q):=(\Delta_pf(q))^{-1}$. It is immediate to verify that $f'$ is a continuous $\bH$--intrinsic slice function, the one induced by the stem function sending $z$ into the inverse of $\Delta_pF(z)$ in $\bH_\bC$. Since $\Psi_{\bR,T}$ is a unity--preserving homomorphism, we infer that
\beq \label{eq:inverse}
\Delta_p(f(T))f'(T)=\Delta_pf(T)f'(T)=\1=f'(T)\Delta_pf(T)=f'(T)\Delta_p(f(T)).
\eeq
This means that $\Delta_p(f(T))$ has a two--sided inverse in $\gB(\sH)$, as desired.
\end{proof}

We conclude this section with two corollaries we will use later.

\begin{corollary} \label{cor:restr}
Let $\sH$ be a quaternionic Hilbert space, let $T \in \gB(\sH)$ be a normal operator, let $J \in \gB(\sH)$ be an anti self--adjoint and unitary operator satisfying (\ref{OMEGA}) and (\ref{decA0-bis}), let $\imath \in \cS$ and let $f \in \RS(\ssp(T),\bH)$. Denote by  $f|_{\bC_\imath}:\sigma(T|_{\sH^{J\imath}_+}) \lra \bC_\imath$ the continuous function obtained restricting $f$. Then the restriction $\Psi_{\bR,T}(f)|_{\sH^{J,\imath}_+}$ of $\Psi_{\bR,T}(f)$ to $\sH^{J,\imath}_+$ defines an operator in $\gB(\sH^{J,\imath}_+)$ and it holds:
\[
\Psi_{\bR,T}(f)|_{\sH^{J,\imath}_+}=f|_{\bC_\imath}(T|_{\sH^{J\imath}_+}),
\]
where $f|_{\bC_\jmath} (T|_{\sH_+^{J\jmath}})$ indicates the operator in $\gB(\sH^{J\jmath}_+)$ defined in the framework of standard continuous functional calculus in $\bC_\imath$--Hilbert spaces.
\end{corollary}
\begin{proof}
The fact that $\Psi_{\bR,T}(f)|_{\sH^{J,\imath}_+}$ is a well--defined operator in $\gB(\sH^{J,\imath}_+)$ follows immediately from Theorem \ref{teofinale1}$(\mr{a})$ and Proposition \ref{propinterssigma}. By Corollary \ref{corollaruCplus}, we know that $\sigma(T|_{H^{J\imath}_+})=\ssp(T) \cap \bC_\imath^+$. Let $w \in \mscr{C}(\ssp(T) \cap \bC_\imath^+,\bC_\imath)$. Lemma \ref{lemstemslice} and the representation formula for slice functions (see Remark \ref{remslices}$(1)$) imply the existence and the unicity of a function $W \in \RS(\ssp(T),\bH)$ such that $W|_{\ssp(T) \cap \bC_\imath^+}=w$. Using again Theorem \ref{teofinale1}$(\mr{a})$, we infer that the ope\-rator $\Psi_{\bR,T}(W) \in \gB(\sH)$ is normal and $J$ commutes both with $\Psi_{\bR,T}(W)$ and $\Psi_{\bR,T}(W)^*=\Psi_{\bR,T}(W^*)$. In this way, Proposition \ref{propinterssigma} and Theorem \ref{teofinale1} ensures that the map
\[
\mscr{C}(\ssp(T) \cap \bC_\imath^+,\bC_\imath) \ni w \mapsto \Psi_{\bR,T}(W)|_{\sH^{J\imath}_+} \in \gB(\sH^{J\imath}_+)
\]
is a $\bC_\imath$--complex $^*$--homomorphism, sending $1_{\ssp(T) \cap \bC_\imath^+}$ into the identity operator on $\sH^{J\imath}_+$ and the inclusion map $\ssp(T) \cap \bC_\imath^+ \hookrightarrow \bC_\imath$ into $T|_{\sH^{J\imath}_+}$. By the uniqueness of the continuous functional calculus $^*$--homomorphism in $\bC_\imath$--Hilbert spaces, it holds: $\Psi_{\bR,T}(W)|_{\sH^{J\imath}_+}=w(T|_{\sH^{J\imath}_+})$ for every $w \in \mscr{C}(\ssp(T) \cap \bC_\imath^+,\bC_\imath)$, as desired.
\end{proof}

\begin{corollary} \label{cor:kk}
Let $J \in \gB(\sH)$ be an anti self--adjoint and unitary operator sati\-sfying (\ref{OMEGA}) and (\ref{decA0-bis}), and let $K \in \gB(\sH)$ be another anti self--adjoint and unitary operator such that $JK=-K J$ and $K$ commutes both with $A=(T+T^*)\frac{1}{2}$ and with $B=|T-T^*|\frac{1}{2}$. Choose $f \in \RS(\ssp(T),\bH)$ and define $f(T):=\Psi_{\bR,T}(f)$ and $f^*(T):=\Psi_{\bR,T}(f^*)$. Then it holds:
\begin{align}
\label{eq:fjjf}
& f(T)J=J \, f(T),\\
\label{eq:fkkf-}
& f(T)K=K \, f^*(T).
\end{align}
\end{corollary}
\begin{proof}
First, suppose that $f=\I(F_1+iF_2) \in \PRS(\ssp(T),\bH)$. Let $Q_1,Q_2 \in \bR[X,Y]$ such that $F_m(\alpha,\beta)=Q_m(\alpha,\beta)$ for every $m \in \{1,2\}$ and $(\alpha,\beta) \in \K$, where $\K$ is the subset of $\bC$, invariant under complex conjugation, such that $\OO_\K=\ssp(T)$. By Definition \ref{defpolT} and Lemma \ref{lem:unique}, we have that $f(T)=Q_1(A,B)+JQ_2(A,B)$. Since $J$ and $K$ commute both with $A$ and with $B$, and $JK=-K J$, equalities (\ref{eq:fjjf}) and (\ref{eq:fkkf-}) are evident. Now the corollary follows immediately from Proposition \ref{prop:Weierstrass} and from the continuity of $\Psi_{\bR,T}$ (see Theorem \ref{teofinale1}).
\end{proof}

\begin{remark} 
When $T\in\gB(\sH)$ is self--adjoint, an easy inspection shows that the functions of the operator $T$ defined in Theorem~\ref{teoA1} coincide with those defined in Theorems~\ref{teofinale1}.
\end{remark}


\subsection{Continuous $\boldsymbol{\bC_\jmath} \,$-slice functions of a normal operator $\boldsymbol T$}
We can pass to discuss the analogue for continuous $\bC_\jmath$--slice functions of what done for continuous $\bH$--intrinsic slice functions.

\textit{Fix $\jmath \in \cS$ and equip $\jS(\ssp(T),\bH)$ with the structure of commutative (two--sided) $\bC_\jmath$--Banach unital $C^*$--algebra given in Theorem \ref{thm:two-sided-C^*}(c)}.

Recall that, given $q \in \bC_\jmath$ and $f \in \jS(\ssp(T),\bH)$, the left and right scalar multiplications $q \cdot f$ and $f \cdot q$ in $\jS(\ssp(T),\bH)$ are defined as the slice products $q \cdot f:=c_q \cdot f$ and $f \cdot q:=f \cdot c_q$, where $c_q$ denotes the slice function on $\ssp(T)$ constantly equal to $q$. We know that $q \cdot f$ and $f \cdot q$ are equal and coincide with the pointwise scalar product $fq$. In this way, taking into account only the right scalar multiplication $(f,q) \mapsto f \cdot q$, we can consider $\jS(\ssp(T),\bH)$ as a standard commutative $\bC_\jmath$--Banach unital $C^*$--algebra.

\textit{Fix an anti self--adjoint and unitary operator $J \in \gB(\sH)$ sati\-sfying (\ref{OMEGA}) and (\ref{decA0-bis})}. We remind the reader that condition (\ref{OMEGA}) requires that $J$ commute with $T$ and $T^*$, and condition (\ref{decA0-bis}) imposes that $T$ decomposes as follows: $T=A+JB$, where $A:=(T+T^*)\frac{1}{2}$ and $B:=|T-T^*|\frac{1}{2}$.

Fix a left scalar multiplication $\bH \ni q \mapsto L_q$ of $\sH$ such that $L_\jmath=J$. Such a left scalar multiplication of $\sH$ exists by Proposition \ref{propJ}$(\mr{a})$ and makes $\gB(\sH)$ a quaternionic two--sided Banach unital $C^*$--algebra, via Theorem \ref{thm:two-sided}. Restricting the scalars from $\bH$ to $\bC_\jmath$, one defines on $\gB(\sH)$ a structure of  two--sided $\bC_\jmath$--Banach unital $C^*$--algebra. It is important to observe that such a structure on $\gB(\sH)$ depends only on real scalar multiplication (\ref{eq:rT}) and on $J$, and not on the fixed left scalar multiplication $\bH \ni q \mapsto L_q$ of $\sH$. Indeed, if $q=\alpha+\jmath\beta \in \bC_\jmath$ with $\alpha,\beta \in \bR$, then it holds: 
\[
qT=T\alpha+J(T\beta)
\quad \mbox{ and } \quad
Tq=T\alpha+(T\beta)J.
\]

Since $J$ commutes with $T$, then $qT=Tq$ for every $q \in \bC_\jmath$ and hence $\gB(\sH)$ can be considered as a standard $\bC_\jmath$--Banach unital $C^*$--algebra by taking into account only the right scalar multiplication $(T,\alpha+\jmath\beta) \mapsto T\alpha+(T\beta)J$. 

\textit{We assume that $\gB(\sH)$ is equipped with that structure of $\bC_\jmath$--Banach unital $C^*$--algebra}.

We are now in a position to present our next result.

\begin{theorem} \label{teofinale2}
There exists, and is unique, a continuous $^*$--homomorphism 
\[
\Psi_{\bC_\jmath,T}:\jS(\ssp(T),\bH) \ni f \mapsto f(T) \in \gB(\sH)
\]
of $\bC_\jmath$--Banach unital $C^*$--algebras such that:
\begin{itemize}
 \item[$(\mr{i})$] $\Psi_{\bC_\jmath,T}$ is unity--preserving; that is, $\Psi_{\bR,T}(1_{\ssp(T)})=\1$.
 \item[$(\mr{ii})$] $\Psi_{\bC_\jmath,T}(\mi{id})=T$, where $\mi{id}:\ssp(T) \hookrightarrow \bH$ denotes the inclusion map.
\end{itemize}
The following further facts hold true.
\begin{itemize}
 \item[$(\mr{a})$] $\Psi_{\bC_\jmath,T}$ extends $\Psi_{\bR,T}$ in the following sense. Let $f \in \jS(\ssp(T),\bH)$ and let $f_0$ and $f_1$ be the unique functions in $\RS(\ssp(T),\bH)$ such that $f=f_0+f_1\jmath$ (see Lemma \ref{lem:RS}). Then it holds:
\beq \label{PhibcPhiT}
\Psi_{\bC_\jmath,T}(f)=\Psi_{\bR,T}(f_0)+\Psi_{\bR,T}(f_1)J.
\eeq
\item[$(\mr{b})$] For every $f \in \jS(\ssp(T),\bH)$, $\Psi_{\bC_\jmath,T}(f)$ is normal and commutes with $J$.
 \item[$(\mr{c})$] For every $f \in \jS(\ssp(T),\bH)$, the following \emph{continuous $\bC_\jmath$--slice spectral map property} holds:  
\[
\ssp(f(T))=\OO_{f(\ssp(T) \cap \bC_\jmath^+)}.
\]
 \item[$(\mr{d})$] $\Psi_{\bC_\jmath,T}$ is norm decreasing; that is, $\|f(T)\| \leq \|f\|_\infty$ if  $f \in \jS(\ssp(T),\bH)$. More precisely, it holds:
\beq \label{eq:fj}
\|f(T)\|=\|f|_{\ssp(T) \cap \bC_\jmath^+}\|_\infty
\eeq
for every $f \in \jS(\ssp(T),\bH)$.
 \item[$(\mr{e})$] The kernel of $\Psi_{\bC_\jmath,T}$ consists of all functions in $\jS(\ssp(T),\bH)$ vanishing on $\ssp(T) \cap \bC_\jmath^+$. More precisely, a function $f \in \mi{Ker}(\Psi_{\bC_\jmath,T})$ if and only if there exists $g \in \mscr{C}(\ssp(T) \cap \bC_\jmath^-,\bC_\jmath)$ with $g|_{\ssp(T) \cap \bR}=0$ such that
 \[ f(\alpha+\imath\beta)=\frac{1}{2}(1+\imath\jmath) \, g(\alpha-\jmath\beta)
\]
for every $\alpha \in \bR$, $\beta\in\bR^+$ and $\imath \in \cS$ with $\alpha+\imath\beta \in \ssp(T)$. 
\end{itemize}
\end{theorem}
\begin{proof}
We begin proving the uniqueness of $\Psi_{\bC_\jmath,T}$. Assume that $\Psi_{\bC_\jmath,T}$ exists. Since it is a $\bC_\jmath$ $^*$--homomorphism satisfying $(\mr{i})$ and $(\mr{ii})$, if a function $f=f_0+f_1\jmath$ in $\jS(\ssp(T),\bH)$ is decomposed as in $(\mr{a})$, then one has $\Psi_{\bC_\jmath,T}(f)=\Psi_{\bC_\jmath,T}(f_0)+\Psi_{\bC_\jmath,T}(f_1)J$ and then, by Theorem~\ref{teofinale1}, $\Psi_{\bC_\jmath,T}(f_0)=\Psi_{\bR,T}(f_0)$ and $\Psi_{\bC_\jmath,T}(f_1)=\Psi_{\bR,T}(f_1)$. It follows that $\Psi_{\bC_\jmath,T}$ is unique.

Concerning the existence, we have now a natural way to define $\Psi_{\bC_\jmath,T}$:
\[
\Psi_{\bC_\jmath,T}(f):=\Psi_{\bR,T}(f_0)+\Psi_{\bR,T}(f_1)J
\]
if $f=f_0+f_1\jmath \in \jS(\ssp(T),\bH)$ with $f_0,f_1 \in \RS(\ssp(T),\bH)$. Evidently, $\Psi_{\bC_\jmath,T}$ sati\-sfies $(\mr{i})$, $(\mr{ii})$ and $(\mr{a})$. Denote $\Psi_{\bC_\jmath,T}(f)$ simply by $f(T)$ for every $f \in\jS(\ssp(T),\bH)$. 

Let us verify that $\Psi_{\bC_\jmath,T}$ is a $\bC_\jmath$ $^*$--homomorphism. Fix two functions $f=f_0+f_1\jmath$ and $g=g_0+g_1\jmath$ in $\jS(\ssp(T),\bH)$ decomposed as in $(\mr{a})$. For simplicity, denote $\Psi_{\bR,T}(f_0)$ by $f_0(T)$ and $\Psi_{\bR,T}(f_1)$ by $f_1(T)$. It holds:
\begin{align}
f \cdot g &=f_0g_0+f_0g_1\jmath+f_1 \cdot \jmath \cdot g_0+f_1 \cdot \jmath \cdot g_1 \cdot \jmath=\nonumber\\
&=f_0g_0+f_0g_1\jmath+f_1g_0\jmath+f_1 \cdot \jmath \cdot \jmath \cdot g_1=\nonumber\\
\label{eq:jmath}
&=(f_0g_0-f_1g_1)+(f_0g_1+f_1g_0)\jmath
\end{align}
and hence, bearing in mind that $\Psi_{\bR,T}$ is a real $*$--homomorphism satisfying point $(\mr{a})$ of Theorem \ref{teofinale1}, we have:
\begin{align*}
(f \cdot g)(T)
&=f_0(T)g_0(T)-f_1(T)g_1(T)+f_0(T)g_1(T)J+f_1(T)g_0(T)J=\\
&=f_0(T)g_0(T)+f_0(T)g_1(T)J+f_1(T)Jg_0(T)+f_1(T)Jg_1(T)J=\\
&=(f_0(T)+f_1(T)J)(g_0(T)+g_1(T)J)=f(T)g(T).
\end{align*}
Let $q=\alpha+\jmath\beta \in \bC_\jmath$ with $\alpha,\beta \in \bR$. Replacing $g$ with $c_q$ in (\ref{eq:jmath}), we obtain that $f \cdot q=(f_0\alpha-f_1\beta)+(f_1\alpha+f_0\beta)\jmath$. Since $(f_1(T)\alpha)J=(f_1(T)J)\alpha$ and $-f_1(T)\beta=(f_1(T)J\beta)J$, we infer that
\begin{align*}
(f \cdot q)(T) &=(f_0(T)\alpha-f_1(T)\beta)+(f_1(T)\alpha+f_0(T)\beta)J=\\
&=(f_0(T)+f_1(T)J)\alpha+\big((f_0(T)+f_1(T)J)\beta\big)J=\\
&=f(T) \, q.
\end{align*} 

It remains to show that $\Psi_{\bC_\jmath,T}$ preserves the $^*$--involutions. Observe that $f^*=f_0^*-\jmath \cdot f_1^*=f_0^*-f_1^*\jmath$. Since $f_0^*(T)=f_0(T)^*$, $f_1^*(T)=f_1(T)^*$ and $J$ commutes with $f_1^*(T)$, we have that $-f_1^*(T)J=(f_1(T)J)^*$ and hence
\[
f^*(T)=f_0^*(T)-f_1^*(T)J=(f_0(T)+f_1(T)J)^*=f(T)^*.
\]

The operators $f_0(T)$, $f_0^*(T)$, $f_1(T)$, $f_1^*(T)$ and $J$ commute mutually. This fact implies at once that $f(T)$ is normal and commutes with $J$; that is, $(\mr{b})$ is proved.

Let us show $(\mr{c})$. First, we need to show that
\beq \label{finalid}
f(T)|_{\sH^{J\jmath}_+}=f|_{\bC_\jmath}(T|_{\sH^{J\jmath}_+}).
\eeq
By Corollary \ref{cor:restr}, we have that $f_0(T)|_{\sH^{J\jmath}_+}=f_0|_{\bC_\jmath}(T_{\sH^{J\jmath}_+})$ and $f_1(T)|_{\sH^{J\jmath}_+}=f_1|_{\bC_\jmath}(T_{\sH^{J\jmath}_+})$, where $f_0|_{\bC_\jmath}$ and $f_1|_{\bC_\jmath}$ denote the operators in $\gB(\sH^{J\jmath}_+)$ defined in  the (standard) functional calculus in $\bC_\jmath$--Hilbert spaces. It holds:
\begin{align*}
f(T)|_{\sH^{J\jmath}_+} &=\big(f_0(T)+f_1(T)J\big)|_{\sH^{J\jmath}_+}=f_0(T)|_{\sH^{J\jmath}_+}+J|_{\sH^{J\jmath}_+}f_1(T)|_{\sH^{J\jmath}_+}=\\
&=f_0|_{\bC_\jmath}(T|_{\sH^{J\jmath}_+})+\left(f_1|_{\bC_\jmath}(T|_{\sH^{J\jmath}_+})\right)\jmath=\\
&=\left(f_0|_{\bC_\jmath}+(f_1\jmath)|_{\bC_\jmath}\right)(T|_{\sH^{J\jmath}_+})=f|_{\bC_\jmath}(T|_{\sH^{J\jmath}_+}),
\end{align*}
which proves (\ref{finalid}). We are now in a position to prove $(\mr{c})$. Apply Proposition~\ref{propinterssigma} $(\mr{c})$ to $f(T)$. We obtain that
\beq \label{quasifine}
\sigma_S(f(T)) \cap \bC_\jmath = \sigma(f(T)|_{\sH_+^{J\jmath}}) \cup \overline{ \sigma(f(T)|_{\sH_+^{J\jmath}})}.
\eeq
By combining (\ref{finalid}), (\ref{quasifine}) and Corollary~\ref{corollaruCplus} with the continuous functional calculus theo\-rem for $\bC_\jmath$ normal operators, we infer that
\begin{align*}
\ssp(f(T)) \cap \bC_\jmath &=\sigma(f|_{\bC_\jmath} (T|_{\sH_+^{J\jmath}})) \cup \overline{\sigma( f|_{\bC_\jmath} (T|_{\sH_+^{J\jmath}}))}=\\
&=f|_{\bC_\jmath}(\sigma(T|_{\sH_+^{J\jmath}})) \cup \overline{f|_{\bC_\jmath}(\sigma(T|_{\sH_+^{J\jmath}}))}=\\
&=f(\ssp(T) \cap \bC^+_\jmath) \cup \overline{f(\ssp(T) \cap \bC^+_\jmath)}.
\end{align*}

This proves $(\mr{c})$. 

Let us prove $(\mr{d})$. Since $f(T)$ is normal, we can apply the spectral radius formula (see (\ref{raggiospett})) obtaining the equality
$\|f(T)\|=\sup\{|q| \in \bR^+ \, | \, q \in \ssp(f(T))\}$. Piecing together the latter equality with $(\mr{c})$ and with Remark \ref{rem:sectional}, we infer at once $(\mr{d})$.

Point $(\mr{e})$ is an immediate consequence of $(\mr{d})$ and of the representation formula for slice functions (see Remark \ref{remslices}$(1)$).
\end{proof}

\begin{remark} \label{rem:nonunico-j}
Unless the case in which $\mi{Ker}(T-T^*)=\{0\}$, the map $\Psi_{\bC_\jmath,T}$ depends on the choice of $J$. Indeed, $J$ is uniquely determined by $T$ on $\mi{Ker}(T-T^*)^\perp$ (see Theorem \ref{teoext}), but it can be chosen in many ways on $\mi{Ker}(T-T^*)$ (see Remark \ref{remarkJunded}). In this way, if $c_\jmath$ is the function on $\ssp(T)$ constantly equal to\insec$\jmath$, then $\Psi_{\bC_\jmath,T}(c_\jmath)$ is equal to $J$ and hence it is not uniquely determined by $T$ on $\mi{Ker}(T-T^*) \neq \{0\}$.
\end{remark}


\subsection{Continuous slice functions of a normal operator $\boldsymbol T$: the circular and general cases} \textit{Fix an anti self--adjoint and unitary operator $J \in \gB(\sH)$ sati\-sfying (\ref{OMEGA}) and (\ref{decA0-bis})}.

\textit{Fix $\jmath, \kappa \in \cS$ with $\jmath\kappa=-\kappa\jmath$ and fix a left scalar multiplication $\bH \ni q \mapsto L_q$ of $\sH$ such that $L_\jmath=J$ and, for every $q \in \bH$, $L_qA=AL_q$ and $L_qB=BL_q$}. The existence of such a left scalar multiplication of $\sH$ is ensured by Theorem \ref{newtheorem}.

\textit{Define $K:=L_\kappa$}. By Proposition \ref{propprod}, we infer that $K$ is an anti--self adjoint and unitary operator in $\gB(\sH)$ such that $JK=-K J$.

\textit{Finally, we equip $\gB(\sH)$ with the structure of quaternionic two--sided Banach unital $C^*$--algebra induced by the fixed left scalar multiplication $\bH \ni q \mapsto L_q$ of $\sH$, as described in Theorem \ref{thm:two-sided}}.

\subsubsection{The circular case}  We now consider continuous circular slice functions, which, differently from the cases of continuous $\bH$--intrinsic slice and $\bC_\jmath$--slice functions we treated above, concerns \textit{non--commutative} structures.

\textit{We assume that the set $\Sc(\ssp(T),\bH)$ is equipped with the qua\-ternionic two--sided Banach unital $C^*$--algebra structure given in Theorem \ref{thm:two-sided-C^*}(d).}

Our next result is as follows.

\begin{theorem}\label{teofinale3}
There exists, and is unique, a continuous $^*$--homomorphism 
\[
\Psi_{c,T}:\Sc(\ssp(T),\bH) \ni f \mapsto f(T) \in \gB(\sH)
\]
of quaternionic two--sided Banach unital $C^*$--algebras such that:
\begin{itemize}
 \item[$(\mr{i})$] $\Psi_{c,T}$ is unity--preserving; that is, $\Psi_{\bR,T}(1_{\ssp(T)})=\1$.
 \item[$(\mr{ii})$] $\Psi_{c,T}(\mi{id})=T$, where $\mi{id}:\ssp(T) \hookrightarrow \bH$ denotes the inclusion map.
\end{itemize}
The following further facts hold true.
\begin{itemize}
 \item[$(\mr{a})$] $\Psi_{c,T}$ extends $\Psi_{\bR,T}$ and $\Psi_{\bC_\jmath,T}$ in the following sense. Let $f \in \Sc(\ssp(T),\bH)$ and let $f_0,f_1,f_2,f_3$ be the unique functions in $\RS(\ssp(T),\bH)$ such that $f=f_0+f_1\jmath+f_2\kappa+f_3\jmath\kappa$ (see Lemma \ref{lem:RS}). Then it holds:
\begin{align*} 
\Psi_{c,T}(f) &=\Psi_{\bR,T}(f_0)+\Psi_{\bR,T}(f_1)J+\Psi_{\bR,T}(f_2)K+\Psi_{\bR,T}(f_3)JK=\\
&=\Psi_{\bC_\jmath,T}(f_0+f_1\jmath)+\Psi_{\bC_\jmath,T}(f_2+f_3\jmath)K.
\end{align*}
\item[$(\mr{b})$] For every $f \in \Sc(\ssp(T),\bH)$, $\Psi_{c,T}(f)$ is normal.
 \item[$(\mr{c})$] For every $f \in \Sc(\ssp(T),\bH)$, the following \emph{continuous circular slice spectral map property} holds:  
\[
\ssp(f(T)) \subset \OO_{f(\ssp(T))}.
\]
 \item[$(\mr{d})$] $\Psi_{c,T}$ is norm decreasing; that is, $\|f(T)\| \leq \|f\|_\infty$ if $f \in \Sc(\ssp(T),\bH)$.
\end{itemize}
\end{theorem}

Before presenting the proof of this result, we underline that point $(\mr{d})$ of the preceding theorem can be improved as follows:

\vspace{.7em}
 
\begin{itemize}
 \item[$(\mr{d}')$] \textit{$\Psi_{c,T}$ is isometric; that is, $\|f(T)\|=\|f\|_\infty$ if $f \in \Sc(\ssp(T),\bH)$.}
\end{itemize}

\vspace{.7em}

We will prove this stronger property of $\Psi_{c,T}$ in the forthcoming paper \cite{GhMoPe2} making use of a spectral representation theorem for normal operators on quaternionic Hilbert spaces.

\begin{proof}[Proof of Theorem \ref{teofinale3}]
Let us prove that $\Psi_{c,T}$ is unique. Suppose that $\Psi_{c,T}$ exists. Let $f=f_0+f_1\jmath+f_2\kappa+f_3\jmath\kappa$ be a function in $\Sc(\ssp(T),\bH)$ decomposed as in $(\mr{a})$. Thanks to Theorem~\ref{teoext} and to the fact that $\Psi_{c,T}$ is a quaternionic $^*$--homomorphism satisfying $(\mr{i})$ and $(\mr{ii})$, we infer immediately the uniqueness of $\Psi_{c,T}$.

Let us pass to define $\Psi_{c,T}$ assuming the truthfulness of $(\mr{a})$:
\[
\Psi_{c,T}(f):=f_0(T)+f_1(T)J+f_2(T)K+f_3(T)JK
\]
if $f=f_0+f_1\jmath+f_2\kappa+f_3\jmath\kappa$ is a function in $\Sc(\ssp(T),\bH)$ decomposed as in $(\mr{a})$, where $f_\ell(T):=\Psi_{\bR,T}(f_\ell)$ for every $\ell \in \{0,1,2,3\}$. It is evident that $\Psi_{c,T}$ satisfies $(\mr{i})$ and $(\mr{ii})$. Let $\ell \in \{0,1,2,3\}$. By Lemma \ref{lem:RS}, we have that $f_\ell$ belongs to $\RS(\ssp(T),\bH)$ and is real--valued. Thanks to the latter fact, we have that $f_\ell^*=\overline{f_\ell}=f_\ell$. Moreover, Corollary \ref{cor:kk} ensures that $f_{\ell}(T)J=Jf_\ell(T)$ and  $f_{\ell}(T)K=K f_\ell(T)$. In this way, the functions $f$ and their images $\Psi_{c,T}(f)$ under $\Psi_{c,T}$ behave like quaternions. This implies at once that $\Psi_{c,T}$ is a $^*$--homomorphism satisfying $(\mr{b})$.

It remains to prove $(\mr{c})$. Indeed, $(\mr{d})$ can be easily deduced from $(\mr{c})$ as we did at the end of the proof of Theorem \ref{teofinale2}. In order to prove point $(\mr{c})$, we follow the strategy used at the end of the proof of Theorem \ref{teofinale1}. Let $f=\I(F_1) \in \Sc(\ssp(T),\bH)$, let $f(T):=\Psi_{c,T}(f)$ and let $p \not\in \OO_{f(\ssp(T))}$. We must show that $\Delta_p(f(T))$ has a two--sided inverse in $\gB(\sH)$. Let $\Delta_pf:\ssp(T) \lra \bH$ be the function sending $q$ into $f^2(q)-f(q)(p+\overline{p})+|p|^2$. Such a function belongs to $\Sc(\ssp(T),\bH)$, because it is the slice function induced by the stem function $\Delta_pF_1:=F_1^2-F_1(p+\overline{p})+|p|^2$. The function $\Delta_pf$ is nowhere zero. Indeed, it there would exist $y \in \ssp(T)$ such that $\Delta_pf(y)=0$, then $f(y) \in \cS_p$ and hence $p \in \OO_{f(\ssp(T))}$, which is impossible. In this way, we can define $f':\ssp(T) \lra \bH$ by setting $f'(q):=(\Delta_pf(q))^-1$. This function is an element of $\Sc(\ssp(T),\bH)$ induced by the stem function $(\Delta_pF_1)^{-1}$. Evidently, equalities (\ref{eq:inverse}) (with $f'(T):=\Psi_{c,T}(f')$) hold also in this situation. The proof is complete.
\end{proof}

\begin{remark} \label{rem:nonunico-c}
The map $\Psi_{c,T}$ depends on the choice of $J$ if $\mi{Ker}(T-T^*) \neq \{0\}$ (see Remark \ref{rem:nonunico-j}) and always on $K$. Indeed, $K$ is not uniquely determined by $T$ and $\Psi_{c,T}(c_\kappa)=K$ if $c_\kappa$ is the function on $\ssp(T)$ constantly equal to $\kappa$.
\end{remark}

\subsubsection{Continuous slice functions of $T$}
We now come to the general case:  we extend the previous definitions of the operator $f(T)$ to every continuous slice function $f\in\Sl(\ssp(T),\bH)$.

Given $f\in \Sl(\ssp(T),\bH)$, let $f_0,f_1,f_2,f_3$ be the unique functions in $\RS(\ssp(T),\bH)$ such that $f=f_0+f_1\jmath+f_2\kappa+f_3\jmath\kappa$ (see Lemma \ref{lem:RS}). Define
\[\label{f(T)}
f(T):=f_0(T)+f_1(T)J+f_2(T)K+f_3(T)JK,
\]
where $f_\ell(T)$ denotes $\Psi_{\bR,T}(f_\ell)$
 for every $\ell \in \{0,1,2,3\}$. From the definition, it follows immediately that the map $f\mapsto f(T)$ is $\R$--linear. Since $\Psi_{\bR,T}$ is continuous, the map $f\mapsto f(T)$ is also continuous; that is, there exists a positive constant $C$ such that
\[
\|f(T)\|\le C \|f\|_\infty
\]
for every $f \in \Sl(\ssp(T),\bH)$. Furthermore, we have:

\begin{proposition}\label{teofinale4}
Given $f,g\in \Sl(\ssp(T),\bH)$, the following facts hold.
\begin{itemize}
 \item[$(\mr{a})$] If $f\in \jS(\ssp(T),\bH)$ or $g\in \Sc(\ssp(T),\bH)$, then
\[(f\cdot g)(T)=f(T)g(T).\]
 \item[$(\mr{b})$] 
If $p\in \bC_\jmath$ and $q \in \bH$, then
\[
(p\cdot g)(T)=p \, g(T)\text{\quad and\quad}(f\cdot q)(T)=f(T) q.
\]
 \item[$(\mr{c})$] Let $\tilde f=f_0+f_1\jmath+f_2^*\kappa+f_3^*\jmath\kappa$. Then $\tilde f\in \Sl(\ssp(T),\bH)$ and it holds:
\[f(T)^*=\tilde{f}^*(T)\]
\end{itemize}
\end{proposition}

\begin{proof}
Let $f=(f_0+f_1\jmath)+(f_2+f_3\jmath)\kappa=f^1+f^2\kappa$, where $f^1:=f_0+f_1\jmath, f^2:=f_2+f_3\jmath\in  \jS(\ssp(T),\bH)$, and similarly for $g=g^1+g^2\kappa$. From $\jmath\kappa=-\kappa\jmath$, it follows that $k\cdot f=f'\cdot k$, where  $f':=(f_0-f_1\jmath)+(f_2-f_3\jmath)\kappa$. Therefore
\[
f\cdot g=(f^1+f^2\kappa)\cdot (g^1+g^2\kappa)=f^1\cdot g^1 -f^2\cdot (g^2)'+\left(f^1\cdot g^2+f^2\cdot (g^1)'\right)\kappa
\]
and
\[
(f\cdot g)(T)=f^1(T) g^1(T) -f^2(T) (g^2)'(T)+\left(f^1(T) g^2(T)+f^2(T) (g^1)'(T)\right)K.
\]
On the other hand, from \ref{cor:kk} we get that
\begin{align*}
f(T)g(T)&=(f^1(T)+f^2(T)K)(g^1(T)+g^2(T)K)=\\&=f^1(T) g^1(T) -f^2(T)(g^2)^*(T)+\left(f^1(T)g^2(T)+f^2(T) (g^1)^*(T)\right)K.
\end{align*}
From Lemma \ref{lem:RS}, we get that if $f\in \jS(\ssp(T),\bH)$, then $f^2=0$, while if $g\in \Sc(\ssp(T),\bH)$, then $(g^1)'=(g^1)^*$ and $(g^2)'=(g^2)^*$. In both cases, $(f\cdot g)(T)$ coincides with $f(T)g(T)$ and $(\mr{a})$ is proved. Part $(\mr{b})$ is an immediate consequence of $(\mr{a})$.

It remains to prove $(\mr{c})$. Since $\tilde{f}^*=(f_0+f_1\jmath +f_2^*\kappa+f_3^*\jmath\kappa)^*=f_0^*-f_1^*\jmath-f_2\kappa-f_3\jmath\kappa$, $(\mr{c})$ is a consequence of the following equality:
\begin{align*}
f(T)^*&=f_0(T)^*-Jf_1(T)^*-K f_2(T)^*+K J f_3(T)^*=\\&=f_0^*(T)-f_1^*(T)J- f_2(T)K- f_3(T)JK.
\end{align*}
This proves the proposition.
\end{proof}


\subsection{Slice regular functions of a normal operator $\boldsymbol T$}
\label{sec:sliceregular}

As we recalled in the Introduction, an important subclass of slice functions is the one of \emph{slice regular functions}, those slice functions that are induced by holomorphic stem functions. They were introduced in \cite{GeSt2006CR,GeSt2007Adv} as quaternionic power series and later generalized in \cite{CoSaSt2009Israel,GhPe_AIM,GhPe_Trends}. A functional calculus for slice regular functions of a (bounded right $\bH$--linear) operator on quaternionic two--sided Banach module has been developed in \cite{libroverde} as an effective generalization of the (classical) holomorphic functional calculus. We recall the definition given in \cite{libroverde}.

\begin{definition}\cite[Def.~4.10.4]{libroverde}
Let $V$ be a quaternionic two--sided Banach module, let $T \in \gB(V)$ be an operator and let $f:U \lra \bH$ be a slice regular function defined on a circular open neighborhood of $\ssp(T)$ in $\bH$. Fix any $\jmath \in \cS$ and define the element $f(T)_{reg}$ of $\gB(V)$ by setting
\beq \label{fTreg}
f(T)_{reg}:=\frac1{2\pi}\int_{\partial(U\cap\bC_\jmath)}S_L^{-1}(s,T)\,ds\, \jmath^{-1} \, f(s)\,.
\eeq
Here $S_L^{-1}(s,x)$ denotes the \emph{Cauchy kernel} for (left) slice regular functions.
\end{definition}

\begin{proposition}
Let $T \in \gB(\sH)$ be a normal operator and let $f:U \lra \bH$ be a slice regular function defined on a circular open neighborhood of $\ssp(T)$ in $\bH$. Then $f|_{\ssp(T)} \in \Sl(\ssp(T),\bH)$ and it holds:
\beq \label{calculi}
f(T)_{reg}=f|_{\ssp(T)}(T);
\eeq
that is, the two functional calculi coincide if $T$ is normal and $f$ is slice regular.
\end{proposition}
\begin{proof}
Firstly consider the case of a  $\bH$--intrinsic slice regular $f\in\RS(\OO,\bH)$, with $\OO$ open neighbourhood of $\ssp(T)$. If $f$ is a polynomial $\sum_{n=0}^m q^n a_n$ with real coefficients $a_n$, then $f(T)_{reg}$ and $f\rr_{\ssp(T)}(T)=\Psi_{\bR,T}(f)$ are both equal to $\sum_{n=0}^m T^n a_n$ (cf.~\cite[Theorem~4.8.10]{libroverde}). If $f$ is a \emph{rational} slice regular function, of the form $f=p^{-1}p'$ for two $\bH$--intrinsic slice regular polynomials $p$ and $p'$, with $p=\I(P)=\I(P_1+iP_2)\ne0$ on $\ssp(T)$ and $p^{-1}=\I((P_1+iP_2)/P^2)$, the algebraic properties stated in Theorem~\ref{teofinale1} and \cite[Proposition~4.11.6]{libroverde} give:
\[\Psi_{\bR,T}(p)\Psi_{\bR,T}(f)=\Psi_{\bR,T}(p')=p'(T)_{reg}=p(T)_{reg}f(T)_{reg}=\Psi_{\bR,T}(p)f(T)_{reg}\:,\]
from which it follows that $\Psi_{\bR,T}(f)=f(T)_{reg}$, since $\Psi_{\bR,T}(p)$ is invertible.
We now show that every $f=\I(F)\in\RS(\OO,\bH)$ can be approximated in the supremum norm by rational $\bH$--intrinsic slice regular functions (see also \cite{CoSaSt2011PAMS} for a Runge Theorem for slice regular functions).
Let $\K\subset\bC$, invariant with respect to conjugation, such that $\OO_\K=\ssp(T)$.
Given $\epsilon>0$, the classical Runge Theorem assures the existence of a rational function $R$, with poles outside $\K$, such that $|F(z)-R(z)|<\epsilon/2$ for every $z\in \K$. Let $R'(z):=(R(z)+\overline{R(\bar z)})/2$. Then $R'$ is a  rational stem function, whose induced slice function $r$ is regular and $\bH$--intrinsic on a neighbourhood of $\ssp(T)$. Therefore
\[
\|f-r\|_{\ssp(T)}=\sup_{z\in \K}|F(z)-R'(z)|\le\left|\frac{F(z)-R(z)}2\right|+\left|\frac{F(\bar z)-R(\bar z)}2\right|<\epsilon\:.
\]
Using \cite[Theorem~4.10.6]{libroverde} and the continuity of $f\mapsto f(T)$, we obtain the equality (\ref{calculi}) for every $f\in\RS(\OO,\bH)$.

Now we come to the case of a generic slice regular $f\in\Sl(\OO,\bH)$, decomposed as $f=f_0+f_1\jmath+f_2\kappa+f_3\jmath\kappa$, with $f_k \in  \RS(\OO,\bH)$. From the above and \cite[Proposition~4.11.1]{libroverde} we get (\ref{calculi}):
\begin{align*}
f(T)_{reg}&=f_0(T)_{reg}+f_1(T)_{reg}J+f_2(T)_{reg}K+f_3(T)_{reg}JK=\\&= f_0\rr_{\ssp(T)}(T)+f_1\rr_{\ssp(T)}(T)J+f_2\rr_{\ssp(T)}(T)K+f_3\rr_{\ssp(T)}(T)JK=\\&= f\rr_{\ssp(T)}(T)\:.
\end{align*}
This completes the proof.
\end{proof}

		

\begin{thebibliography}{10}

\bibitem{Adler}
Stephen~L. Adler, \emph{Quaternionic quantum mechanics and quantum fields},
  International Series of Monographs on Physics, vol.~88, The Clarendon Press
  Oxford University Press, New York, 1995. \MR{1333599 (97d:81001)}

\bibitem{Aerts}
Diederik Aerts, \emph{Quantum axiomatics}, Handbook of quantum logic and
  quantum structures---quantum logic, Elsevier/North-Holland, Amsterdam, 2009,
  pp.~79--126. \MR{2724647 (2011j:81011)}

\bibitem{BC}
Enrico~G. Beltrametti and Gianni Cassinelli, \emph{The logic of quantum
  mechanics}, Encyclopedia of Mathematics and its Applications, vol.~15,
  Addison-Wesley Publishing Co., Reading, Mass., 1981, With a foreword by Peter
  A. Carruthers. \MR{635780 (83d:81008)}

\bibitem{BvN}
Garrett Birkhoff and John von Neumann, \emph{The logic of quantum mechanics},
  Ann. of Math. (2) \textbf{37} (1936), no.~4, 823--843. \MR{1503312}

\bibitem{BDS}
F.~Brackx, Richard Delanghe, and F.~Sommen, \emph{Clifford analysis}, Research
  Notes in Mathematics, vol.~76, Pitman (Advanced Publishing Program), Boston,
  MA, 1982. \MR{697564 (85j:30103)}

\bibitem{CT}
G.~Cassinelli and P.~Truini, \emph{Quantum mechanics of the quaternionic
  {H}ilbert spaces based upon the imprimitivity theorem}, Rep. Math. Phys.
  \textbf{21} (1985), no.~1, 43--64. \MR{802061 (87g:81038)}

\bibitem{CoSaSt2009Israel}
Fabrizio Colombo, Irene Sabadini, and Daniele~C. Struppa, \emph{Slice monogenic
  functions}, Israel J. Math. \textbf{171} (2009), 385--403. \MR{2520116
  (2010e:30039)}

\bibitem{libroverde}
\bysame, \emph{Noncommutative functional calculus}, Progress in Mathematics,
  vol. 289, Birkh\"auser/Springer Basel AG, Basel, 2011, Theory and
  applications of slice hyperholomorphic functions. \MR{2752913}

\bibitem{CoSaSt2011PAMS}
\bysame, \emph{The {R}unge theorem for slice hyperholomorphic functions}, Proc.
  Amer. Math. Soc. \textbf{139} (2011), no.~5, 1787--1803. \MR{2763766
  (2012c:30104)}

\bibitem{DS}
Nelson Dunford and Jacob~T. Schwartz, \emph{Linear operators. {P}art {II}:
  {S}pectral theory. {S}elf adjoint operators in {H}ilbert space}, With the
  assistance of William G. Bade and Robert G. Bartle, Interscience Publishers
  John Wiley \& Sons\ New York-London, 1963. \MR{0188745 (32 \#6181)}

\bibitem{ebb}
H.-D. Ebbinghaus, H.~Hermes, F.~Hirzebruch, M.~Koecher, K.~Mainzer,
  J.~Neukirch, A.~Prestel, and R.~Remmert, \emph{Numbers}, Graduate Texts in
  Mathematics, vol. 123, Springer-Verlag, New York, 1991, With an introduction
  by K. Lamotke, Translated from the second 1988 German edition by H. L. S.
  Orde, Translation edited and with a preface by J. H. Ewing, Readings in
  Mathematics. \MR{1415833 (97f:00001)}

\bibitem{Emch}
G{\'e}rard Emch, \emph{M\'ecanique quantique quaternionienne et relativit\'e
  restreinte. {I}}, Helv. Phys. Acta \textbf{36} (1963), 739--769. \MR{0176811
  (31 \#1083)}

\bibitem{EGL09}
Kurt Engesser, Dov~M. Gabbay, and Daniel Lehmann (eds.), \emph{Handbook of
  quantum logic and quantum structures---quantum logic},
  Elsevier/North-Holland, Amsterdam, 2009. \MR{2724659 (2011e:81007)}

\bibitem{FJSD}
David Finkelstein, Josef~M. Jauch, Samuel Schiminovich, and David Speiser,
  \emph{Foundations of quaternion quantum mechanics}, J. Mathematical Phys.
  \textbf{3} (1962), 207--220. \MR{0137500 (25 \#952)}

\bibitem{GeSt2006CR}
Graziano Gentili and Daniele~C. Struppa, \emph{A new approach to
  {C}ullen-regular functions of a quaternionic variable}, C. R. Math. Acad.
  Sci. Paris \textbf{342} (2006), no.~10, 741--744. \MR{2227751 (2006m:30095)}

\bibitem{GeSt2007Adv}
\bysame, \emph{A new theory of regular functions of a quaternionic variable},
  Adv. Math. \textbf{216} (2007), no.~1, 279--301. \MR{2353257 (2008h:30052)}

\bibitem{GeStRocky}
\bysame, \emph{Regular functions on the space of {C}ayley numbers}, Rocky
  Mountain J. Math. \textbf{40} (2010), no.~1, 225--241. \MR{2607115
  (2011c:30124)}

\bibitem{GhMoPe2}
R.~Ghiloni, V.~Moretti, and A.~Perotti, \emph{Spectral representations of
  normal operators in quaternionic {H}ilbert spaces}, in preparation.

\bibitem{GhPe_Trends}
R.~Ghiloni and A.~Perotti, \emph{A new approach to slice regularity on real
  algebras}, Hypercomplex analysis and its Applications, Trends Math.,
  Birkh\"auser, Basel, 2011, pp.~109--124.

\bibitem{GhPe_AIM}
\bysame, \emph{Slice regular functions on real alternative algebras}, Adv.
  Math. \textbf{226} (2011), no.~2, 1662--1691. \MR{2737796 (2012e:30061)}

\bibitem{HB}
L.~P. {Horwitz} and L.~C. {Biedenharn}, \emph{{Quaternion quantum mechanics:
  Second quantization and gauge fields}}, Annals of Physics \textbf{157}
  (1984), 432--488.

\bibitem{Kaplansky}
Irving Kaplansky, \emph{Normed algebras}, Duke Math. J. \textbf{16} (1949),
  399--418. \MR{0031193 (11,115d)}

\bibitem{Kulkarni}
S.~H. Kulkarni, \emph{Representations of a class of real {$B^*$}-algebras as
  algebras of quaternion-valued functions}, Proc. Amer. Math. Soc. \textbf{116}
  (1992), no.~1, 61--66. \MR{1110546 (92k:46089)}

\bibitem{Loring} T.~A. Loring, \emph{Factorization of Matrices of Quaternions}
Expositiones Mathematicae
 \textbf{30} (2012), no. ~3,  250–-267



\bibitem{Moretti}
Valter Moretti, \emph{Spectral {T}heory and {Q}uantum {M}echanics, with an
  introduction to the {A}lgebraic {F}ormulation}, Springer-Verlag, Milano-Berlin,
2012

\bibitem{Ng}
Chi-Keung Ng, 
 \emph{On quaternionic functional analysis.}
 Math. Proc. Cambridge Philos. Soc. \textbf{143}  (2007), no.~2, 391–-406

\bibitem{Nelson}
Edward Nelson, \emph{Analytic vectors}, Ann. of Math. (2) \textbf{70} (1959),
  572--615. \MR{0107176 (21 \#5901)}

\bibitem{Analysisnow}
Gert~K. Pedersen, \emph{Analysis now}, Graduate Texts in Mathematics, vol. 118,
  Springer-Verlag, New York, 1989. \MR{971256 (90f:46001)}

\bibitem{Piron}
C.~Piron, \emph{Axiomatique quantique}, Helv. Phys. Acta \textbf{37} (1964),
  439--468. \MR{0204048 (34 \#3894)}

\bibitem{RudinARC}
Walter Rudin, \emph{Real and complex analysis}, third ed., McGraw-Hill Book
  Co., New York, 1987. \MR{924157 (88k:00002)}

\bibitem{RudinFA}
\bysame, \emph{Functional analysis}, second ed., International Series in Pure
  and Applied Mathematics, McGraw-Hill Inc., New York, 1991. \MR{1157815
  (92k:46001)}

\bibitem{Soler}
M.~P. Sol{\`e}r, \emph{Characterization of {H}ilbert spaces by orthomodular
  spaces}, Comm. Algebra \textbf{23} (1995), no.~1, 219--243. \MR{1311786
  (95k:46035)}

\bibitem{vara}
V.~S. Varadarajan, \emph{Geometry of quantum theory}, second ed.,
  Springer-Verlag, New York, 1985. \MR{805158 (87a:81009)}

\bibitem{visw}
K.~Viswanath, \emph{Normal operations on quaternionic {H}ilbert spaces}, Trans.
  Amer. Math. Soc. \textbf{162} (1971), 337--350. \MR{0284843 (44 \#2067)}

\end{thebibliography}

\providecommand{\bysame}{\leavevmode\hbox to3em{\hrulefill}\thinspace}
\providecommand{\MR}{\relax\ifhmode\unskip\space\fi MR }
\providecommand{\MRhref}[2]{%
  \href{http://www.ams.org/mathscinet-getitem?mr=#1}{#2}
}
\providecommand{\href}[2]{#2}

\end{document}